\documentclass[11pt]{article}
\usepackage{amsfonts}
\usepackage{amsmath,amsthm,mathtools,amscd,amssymb,mathrsfs,setspace}
\usepackage{latexsym, epsf, epsfig}
\usepackage{color}
\usepackage[hmargin=.95in,vmargin=.95in]{geometry}
\usepackage{graphicx}
\usepackage{hyperref}
\usepackage{cite}
\usepackage{multirow}
\usepackage{float} 

\newcommand{\nn}{\nonumber}

\newcommand{\cA}{{\mathcal{A}}}

\newcommand{\cE}{{\mathcal{E}}}

\newcommand{\bu}{\mathbf u}
\newcommand{\bv}{\mathbf v}
\newcommand{\bw}{\mathbf w}
\newcommand{\bF}{\mathbf F}

\newcommand{\bV}{\mathbf V}

\renewcommand{\vec}[1]{\mathbf{#1}}

\newcommand{\norm}[1]{\left\| {#1} \right\|}
\newcommand{\braket}[1]{\left\langle {#1} \right\rangle}

\newcommand{\btau}{\boldsymbol{\tau}}
\newcommand{\bxi}{\boldsymbol{\xi}}
\newcommand{\bzeta}{\boldsymbol{\zeta}}
\newcommand\restri[1]{\left.{#1}\right|_{\Gamma_I}}
\newcommand\weakto\rightharpoonup

\theoremstyle{plain}
\newtheorem{theorem}{Theorem}[section]

\newtheorem{corollary}[theorem]{Corollary}

\newtheorem{definition}{Definition}
\theoremstyle{remark}
\newtheorem{remark}{Remark}[section]
\numberwithin{equation}{section} \numberwithin{theorem}{section}
\numberwithin{remark}{section} 

\begin{document}

\title{Weak and Strong Solutions for A Fluid-Poroelastic-Structure Interaction via A Semigroup Approach}
 \author{{\small \begin{tabular}[t]{c@{\extracolsep{.8em}}c@{\extracolsep{.8em}}c}
     George Avalos &       {  Elena Gurvich} &{  Justin T. Webster} \\
\it Univ. Nebraska-Lincoln  \hskip.2cm  & \hskip.4cm \it Univ. of Maryland, Baltimore County   \hskip.2cm  & \hskip.4cm \it Univ. of Maryland, Baltimore County   \\
\it Lincoln, NE &  \it Baltimore, MD &\it Baltimore, MD\\
gavalos@math.unl.edu &  gurv-3@umbc.edu & websterj@umbc.edu \\
\end{tabular}}}
\maketitle

\begin{abstract}
\noindent  A filtration system comprising a Biot poroelastic solid coupled to an incompressible Stokes free-flow is considered in 3D. Across the flat 2D interface, the Beavers-Joseph-Saffman coupling conditions are taken. In the inertial, linear, and non-degenerate case, the hyperbolic-parabolic coupled problem is posed through a dynamics operator on a chosen energy space, adapted from Stokes-Lam\'e coupled dynamics. A semigroup approach is utilized to circumvent issues associated to mismatched trace regularities at the interface. The generation of a strongly continuous semigroup for the dynamics operator is obtained via a non-standard maximality argument. The latter employs a mixed-variational formulation in order to invoke the Babu\v{s}ka-Brezzi theorem. The Lumer-Philips theorem then yields semigroup generation, and thereby, strong and generalized solutions are obtained. As the dynamics are linear, an argument by density then obtains the existence of weak solutions; we extend this argument to the  case where the Biot compressibility of constituents degenerates. Thus, for the inertial linear Biot-Stokes filtration system, we provide a clear elucidation of strong solutions and a construction of weak solutions, as well as their regularity through associated estimates.
\vskip.25cm
\noindent Keywords: {fluid-poroelastic-structure interaction, poroelasticity, Biot, filtration problem, Beavers-Joseph-Saffman, semigroup methods}
\vskip.25cm
\noindent
{\em 2020 AMS MSC}: 74F10, 76S05, 35M13, 76M30, 35D30 
\vskip.4cm
\noindent Acknowledgments: The first author was partially supported by NSF-DMS 1907823 and Simons Grant MP-TSM-00002793; the third author was partially supported by NSF-DMS 2307538.
\end{abstract}

\section{Introduction}

There are a wide range of scientific and engineering applications where one considers the presence of free fluid flow (e.g., Stokes or Navier-Stokes), adjacent to/interacting with a porous media flow (e.g., Darcy) \cite{infsup1,infsup2}---see also the literature reviews in \cite{yotovLagrange,yotov2022}. The porous medium which houses the flow may also be elastic, introducing a two-way dynamic {\em poroelastic} interaction \cite{coussy,GGbook,poroapps}. The interactive dynamics of these two flows occurs across an interface, itself demonstrating complex physics, as driven through both kinematic and stress-matching conditions. These multi-physics systems are often referred to as fluid-poroelastic-structure interactions, or {\em FPSI}s for short \cite{multilayered,bukac,yotov2015,sunny3,filtration2}.  Two predominant FPSI scenarios are: (i) blood flow around/through biological tissue and (ii) flow adjacent to a soil bed or deposited sediment, as well as a myriad of variations of these situations. From a modeling point of view, we are interested in capturing the dynamics near the interface, how the free-flow dynamics may drive the the saturated poroelastic flow, as well as the associated elastic deformations. Owing to this breadth of applications, the topic of FPSIs has been of notable recent interest, see, e.g., \cite{multilayered,numerical2,both,yotovNSB,MRT,yotovMultipoint,yotov2022,yotovBSeye,sunny1,sunny2,sunny3,rectplate}. 

An initial preponderance of modeling, numerical, and analytical work on poroelasticity was associated to applications in geosciences \cite{wheel,show2000,applied2,applied}, going back to the original modeling and analysis of Biot \cite{biot,biot2,biot3} and Terzaghi \cite{terzaghi}. Recent renewed interest in these systems is driven by biological applications, such as the modeling of organs \cite{GGbook, yotovBSeye,bgsw,gilbert3,gilbert4}, the design of artificial organs and tissue scaffolding \cite{multilayered,rectplate,sunny3,sunny1}, and the optimization of stents \cite{sunny2, AGM,canicbio1,mikelicnon}. In the work at hand, we will focus on the general {\em mathematical aspects} of the so called {\em filtration} problem \cite{showfiltration}; namely, the {\em partial differential equation} (PDE) analysis of the interaction of a 3D free flow with a dynamic 3D poroelastic region\footnote{There are no issues in the analysis in considering 2D spatial domains coupled via a 1D interface.}. In this case, the former will be modeled by the inertial Stokes equations, and the latter by the inertial Biot equations. The Biot model comprises a Lam\'e system of elasticity (which may be elliptic or hyperbolic) and a parabolic pressure equation that is allowed to degenerate \cite{coussy,show2000}. The mathematical works on Biot dynamics can be divided into inertial and quasistatic cases. In the former, the Biot dynamics are formally equivalent to thermoelasticity \cite{show2000}, which is well-studied \cite{thermo,redbook}. In the quasistatic case, the Biot system is an implicit, possibly degenerate, parabolic system; some foundational mathematical works  are \cite{show2000, auriault}. Recent work on the PDE theory of Biot systems has considered nonlinear coupling and time-dependent coefficients, arising due to biological applications \cite{bw,bmw,bgsw,mikelicnon}, along with other recent work on nonlinear Biot systems \cite{both, cao, showrecent,showsu}. 

{\bf To summarize}: the central focus in this paper is the fully inertial Biot-Stokes filtration system, considered  in 3D, coupled across a flat, 2D interface. Though we present an entirely linear analysis of solutions, {\em we believe there is no reference in the present literature} that addresses weak and strong solutions comprehensively. Our work here clearly provides a functional framework, a priori estimates for finite energy solutions, and an associated regularity class, including for certain low-regularity time derivatives in the equations. Specifically we will provide well-posedness of the filtration system for strong solutions in the non-degenerate case, as well as a construction of weak solutions, without any regularization or additional dynamics imposed across the interface. We consider the physically-relevant interface conditions, referred to as the Beavers-Joseph-Saffman conditions, taken on the flat interface, which have been rigorously derived and justified \cite{showfiltration,infsup1,mikelicBJS}. These are slip-type conditions, constituting a challenging hyperbolic-parabolic coupling. We will also consider weak solutions in the case where the Biot compressibility is allowed to vanish, which is a degenerate case \cite{show2000,bmw}. 

\subsection{Literature Review and Further Discussion} 
We now provide a simple reference system to assist in describing relevant past work and the contributions of the work at hand. Precise statements of the Main Results can be found in Section \ref{mainresults}, after the problem has been posed through the relevant operators and spaces.

Consider a poroelastic displacement variable $\bu$ with associated Biot pressure $p_b$, defined on a 3D domain $\Omega_b$, as well as a fluid velocity variable $\bv$ with pressure $p_f$, defined on a 3D domain $\Omega_f$. Let $\cE$ be a positive, second order elasticity-type operator (e.g., the Lam\'e operator) and let $A$ be a positive second order diffusion-type operator. Then we have:
\begin{equation}\label{biotstokesdiscuss}
\begin{cases}
    \rho_b \vec{u}_{tt} +\cE\bu+ \nabla p_b = \vec{F}, &\text{ in } \Omega_b \times (0,T),\\
    [c_0p_b +  \nabla \cdot \vec{u}]_t +A p_b = S, &\text{ in } \Omega_b \times (0,T)\\
    \rho_f \vec{v}_t - \Delta \vec{v} + \nabla p_f = \vec{F}_f,\quad \quad \nabla\cdot\vec{v} = 0, \quad  &\text{ in } \Omega_f \times (0,T).\end{cases}
\end{equation}
For the following discussion, we may formally consider the constants $\rho_b, \rho_f, c_0 \ge 0$; the physical descriptions of the constants are given in Section \ref{coupled1} and \ref{coupled2}. We refer to the non-inertial case, $\rho_b=\rho_f=0$, as {\em quasi-static}, while the case $\rho_b,\rho_f>0$ will be called {\em fully inertial}; when $c_0=0$ the  dynamics may be called {\em degenerate}.

We now point to a number of modeling, homogenization, and computational analyses on Biot-Stokes (or Navier-Stokes) FPSIs akin to \eqref{biotstokesdiscuss} with or without inertial terms \cite{earlynumerics,newhomog,bukac,yotov2015,Sanchez-Palencia,wheel} and the more recent \cite{jnr}. Several of these papers provide constructions of weak solutions in their particular context. However, to our knowledge, there are few papers that mathematically address the fully inertial Biot-Stokes system. Some of the work on inertial Biot systems arises in the context of acoustic models \cite{gilbert3,gilbert4}, in particular for bones \cite{gilbert2,gilbert1}. However, the Biot dynamics are often studied in the quasi-static framework (with $\rho_b=0$, formally), and such analyses can often take $\rho_f=0$ as well; the quasi-static Biot-Stokes or Biot-Navier-Stokes systems have been studied thoroughly over the past several years \cite{yotovLagrange,yotovBSeye,yotovNon-NewtBS,yotov2022,yotovMultipoint,yotovNSB}. In contrast, to the knowledge of the authors, the literature does not directly address well-posedness of a fully inertial system $\rho_b,\rho_f>0$ akin to \eqref{biotstokesdiscuss}. Some numerically-oriented papers are \cite{bukac,yotov2015,filtnum}, which address different computational approaches to the inertial problem, permitting geometric as well as convective nonlinearities; many of the aforementioned coupled fluid-Biot papers point to \cite{showfiltration} as their central reference for well-posedness, and do not themselves focus on that issue. Another notable reference on the inertial filtration problem is found in \cite{filtration2}. This paper handles the inertial Biot-Navier-Stokes coupling, and constructs {\em regular, local solutions} in 3D via a Galerkin construction. Finally, we remark on the central reference for filtration dynamics: \cite{showfiltration}. This short chapter considers a linear filtration system like that in \eqref{biotstokesdiscuss} with $\rho_b,\rho_f,c_0>0$, but {\em takes the Stokes flow to be slightly compressible}, yielding a dynamic equation for conservation of fluid mass. In this framework, a semigroup approach is taken and solutions are obtained by invoking the general {\em implicit theory} in \cite{show2000}, itself developed from the earlier references \cite{auriault,indiana}. In fact, two different representations of the system are considered, and, in one context, singular limits in the compressibility and inertial parameters are mentioned. 

Summarizing the analytically-oriented work on the linear filtration problem, {\em it seems there is no reference focusing on both weak and strong solutions for \eqref{biotstokesdiscuss}}, precisely characterizing spatio-temporal regularity of solutions in relation to the specified data. Thus,  {\em we aim to provide a clear framework for strong (and semigroup) solutions} of the linear system in \eqref{biotstokesdiscuss}, coupling the linear Biot and {\em incompressible Stokes} dynamics, supplemented by the physical Beavers-Joseph-Saffman coupling conditions. From these strong solutions, {\em we will provide a construction of weak solutions for all values $c_0 \ge 0$}. We will consider a toroidal spatial configuration which is conducive to simplifying the analysis away from the interface, and we can focus specifically on the dynamics across a flat and stationary interface, as a first step in this rigorous analysis of Biot-Stokes dynamics. 

The particular spatial configuration studied here is motivated by the recent \cite{multilayered}, where a system like \eqref{biotstokesdiscuss} is considered on a similar torus\footnote{as described below, we consider stacked rectangular boxes, where the lateral boundaries are identified}, but with a poroelastic plate  \cite{poroplate,mikelic} evolving across the interface. Indeed, the plate dynamics their have their own equations of motion, but plate terms also appear in kinematic and stress-matching conditions at the interface; the result is  a multi-layered FPSI model (see also \cite{AGM}) which captures a Biot-poroplate-Stokes interaction. That multi-physics system arises from physical considerations of artificial organ design. In the mathematically-oriented \cite{multilayered}, fully nonlinear regimes are considered for the Biot dynamics, as also in \cite{bmw,bgsw}. Weak solutions are obtained through Rothe's method, in a myriad of parameter regimes, including inertial, non-inertial, and degenerate. We also mention the recent work \cite{rectplate} which considers some regularized Biot dynamics interacting with a Navier-Stokes fluid, in the context of a moving domain which supports a ``rectilinear plate" dynamics on a moving interface; weak solutions are constructed for this moving boundary multilayered FPSI problem. Another cutting-edge approach involves treating the Biot-fluid interface in a ``diffuse" manner \cite{bmbm}. Although the presence of a plate at the Biot-Stokes interface provides some novel features for the dynamics in \cite{multilayered} (and is necessitated by the application considered there), it can be viewed as a sort of mathematical regularization. In particular, by prescribing dynamics on the interface, one mitigates issues associated to hyperbolic-parabolic coupling, namely, the mismatching regularities of velocity traces. A natural question, then, is: {\em can we obtain comparable results for weak solutions as in \cite{multilayered}, in the absence of the plate dynamics across the interface?} 

Our main approach here is to build a viable underlying semigroup theory for inertial Biot-Stokes dynamics, which itself has not appeared in the literature. This approach exploits a priori calculations on a suitably defined (smooth) operator domain, circumventing the need for a priori trace estimates on weak solutions---a highly non-trivial issue for this type of coupling. Subsequently, semigroup (generalized \cite{pazy}) solutions can be constructed in a standard way; following this, as the system is linear, weak solutions will be obtained via density. It is worth noting that, although the constructed weak solution will satisfy an energy estimate, it is independent to show that generic weak solutions are unique. We forgo this uniqueness discussion here and relegate it to a forthcoming work. As the Biot system implicitly contains a Lam\'e system (of dynamic elasticity), our work is built upon earlier analyses on Stokes-Lam\'e coupling \cite{avalos*,avalos,avalos08}. More specifically, based on these earlier references, we utilize a mixed-variational formulation for the resolvent system, and maximality of the dynamics operator is obtained through a non-standard use of the Babu\v{s}ka-Brezzi  theorem.

In conclusion, our central contribution is a complete theory of strong solutions, and a construction of weak solutions, for the fully inertial \eqref{biotstokesdiscuss}. We address well-posedness and  regularity through an operator-theoretic framework for the generation of a $C_0$-semigroup on an appropriately chosen phase space. In contrast to previous work, we operate directly on the incompressible Stokes dynamics.

\subsection{Notation}

We will work in the $L^2(U)$ framework, where $U \subseteq \mathbb R^n$ for $n=2,3$ is a given spatial domain. We denote standard $L^2(U)$ inner products by $(\cdot,\cdot)_U$.   Sobolev spaces of the form $H^s(U)$ and $H^s_0(U)$ (along with their duals) will be defined in the standard way \cite{kesavan}, with the $H^s(U)$ norm denoted by $||\cdot||_{s,U}$, or just $||\cdot||_{s}$ when the context is clear. For a Banach space $Y$ we denote its (Banach) dual as $Y'$, and  the associated duality pairing as $\langle \cdot, \cdot\rangle_{Y'\times Y}$. The notation $\mathscr{L}(X,Y)$ denotes the (Banach) space of bounded linear operators from Banach space $X$ to $Y$, and we write $\mathscr{L}(Y)$ when $X=Y$. We denote $\vec{x} = (x_1,x_2,x_3) \in \mathbb{R}^3$, with associated spatial differentiation by $\partial_i$. For estimates, we will use the notation $A\lesssim B$ to mean that there exists a constant $c$ for which $A \le c B$.

\subsection{Biot Dynamics}\label{coupled1}
As we consider Biot-Stokes coupling, we now provide some background and terminology for Biot's poroelastic dynamics. We follow the modeling exposition  in \cite{GGbook,show2000,bgsw}. Let $\Omega_b \subset \mathbb{R}^3$ denote a domain housing a fully-saturated poroelastic structure. We assume an isotropic and homogeneous porous medium, undergoing small elastic deformations; in this scenario, these dynamics will be modeled by Biot's equations \cite{biot,biot2,terzaghi,coussy}. The elastic (Lagrangian) displacement of the structure is denoted by $\vec{u}$ (with associated velocity $\vec{u}_t$ and acceleration $\vec{u}_{tt}$) and the Biot fluid pressure by $p_b$. The function $\bF$ represents a volumetric force on the elastic matrix, and $S$ represents a diffusive source. Then a poroelastic system on $\Omega_b$ is modeled by the following equations:
\begin{align}\label{biot1}
    \rho_b \vec{u}_{tt} - \mu\Delta\vec{u} - (\lambda + \mu)\nabla(\nabla\cdot\vec{u}) + \alpha \nabla p_b = \vec{F}, &\text{ in } \Omega_b \times (0,T),\\ \label{biot2}
    [c_0p_b + \alpha \nabla \cdot \vec{u}]_t - \nabla \cdot [k\nabla p_b] = S, &\text{ in } \Omega_b \times (0,T).
\end{align}
This system consists of the momentum equation for the balance of forces \eqref{biot1} and the diffusion equation resulting from conservation of mass  \eqref{biot2}. Both $\bu$ and $p_b$ are homogenized quantities in the poroelastic region \cite{coussy,Sanchez-Palencia}. The physical parameters of interest are:
\begin{itemize}
    \item $\rho_b\geq 0$ is the density or {\it inertial parameter} of the poroelastic region; when $\rho_b>0$, the system is said to be {\it inertial}; when $\rho_b = 0$, the system is said to be {\it quasi-static}\footnote{This convention is formal; in practice the physical density is always positive. In the quasi-static case, it is in fact the inertia that is negligible, i.e., $\rho_b\mathbf{u}_{tt} \approx 0$};
    \item $\lambda,\mu>0$ are the standard Lam\'e coefficients of elasticity \cite{ciarletbook};
    \item $\alpha>0$ is the so-called {\em Biot-Willis} (coupling) constant \cite{biot2,coussy}, scaled to the system at hand;
    \item $c_0\ge 0$ is the {\em storage coefficient} of the fluid-solid matrix, corresponding to the net {\em compressibility of constituents}; when $c_0=0$ we say that the system has {\em incompressible constitutents};
    \item $k > 0$ is the permeability (here, constant) of the porous matrix.
\end{itemize}
The so called {\em fluid content} of the system is given  by 
\begin{equation}\label{biotfluid}
    \zeta = c_0p + \alpha \nabla \cdot\vec{u},
\end{equation}
and measures the local storage of fluid \cite{show2000}.
For further details on modeling notions and homogenization, we refer to \cite{coussy,GGbook,auriault,Sanchez-Palencia}. 
The fluid content in the poroelastic case is identified with an {\em entropy measure} in the dynamics of thermoelasticity \cite{show2000,thermo}.
\begin{remark} In applications, we may allow the permeability to be a general function of space and time $k = k(\mathbf x,t)$, or a nonlinear function $k = k(p_b,\vec{u})$, in order to permit recent applications where the fluid content's effect on the permeability is relevant---see, e.g., \cite{bgsw,bmw,multilayered,showsu}. \end{remark} 

 We call $\vec{q}$ the {\em discharge velocity} (or Darcy velocity), given through Darcy's law $\vec{q} = -k\nabla p$. Finally, we assume the elastic  stress  $\sigma^E(\vec{u})$ obeys the linear strain-displacement law \cite{kesavan,show2000} given by
\begin{equation}\label{elasticstress}
    \sigma^E(\vec{u}) = 2\mu\vec{D}(\vec{u}) + \lambda(\nabla\cdot\vec{u})\vec{I},
\end{equation}
where $\vec{D}(\vec{u}) = \frac{1}{2}(\nabla\vec{u} + (\nabla\vec{u})^T)$ is the symmetrized gradient \cite{kesavan,temam}. 
In this case, $\sigma^E(\cdot)$ is the stress tensor, relative to a Lam\'e system of elasticity. 

In this work  {\em we  consider the inertial case}, with both densities $\rho_b,\rho_f>0$, so that the dynamics represent a hyperbolic-parabolic coupled {\em evolution} \cite{redbook}. As described above, much work has been done on the quasi-static Biot system with $\rho_b=0$, both as an  independent system and when coupled to a free fluid flow. We first focus on the non-degenerate case here, with $c_0>0$. We will show semigroup generation, which will provide smooth solutions that can be used as approximants in constructing weak solutions for each $c_0>0$ and fixed finite-energy initial data; finally, we will obtain weak solutions the case $c_0 = 0$ as a limit when the parameter $c_0\searrow 0$ in Section \ref{weaksolsec}.

\begin{remark} We note that the {\it inertial} ($\rho_b>0$) Biot system in the case of compressible constitutents $c_0>0$  is formally equivalent to the system of classical coupled thermoelasticity system, which describes the flow of heat through an elastic structure \cite{redbook,thermo, show2000}. In the incompressible constituents case $c_0 = 0$, the resulting dynamics constitutes an implicit {\it hyperbolic} problem.\end{remark}

\subsection{Physical Configuration and Coupled Model}\label{coupled2}
We now specialize to the toroidal spatial configuration to be considered here, namely two stacked 3D boxes, with the lateral sides (in the $x_1$ and $x_2$ directions) effectively identified (through boundary conditions). As noted above, this configuration is the natural analog to that considered \cite{multilayered}, one of our central references here. Such a setup is convenient for a preliminary analyses of this complex, multi-physics model. Specifically, by considering lateral spatial periodicity, we can focus our attention on the interface dynamics while retaining compactness of the domain; additionally, the top, bottom, and interface boundary planes remain separated. A forthcoming work will consider the physical geometry of a  bounded domain (as in \cite{avalos,avalos*,avalos08}) without invoking lateral periodicity. 

\begin{center}
    \includegraphics[width=.7\textwidth]{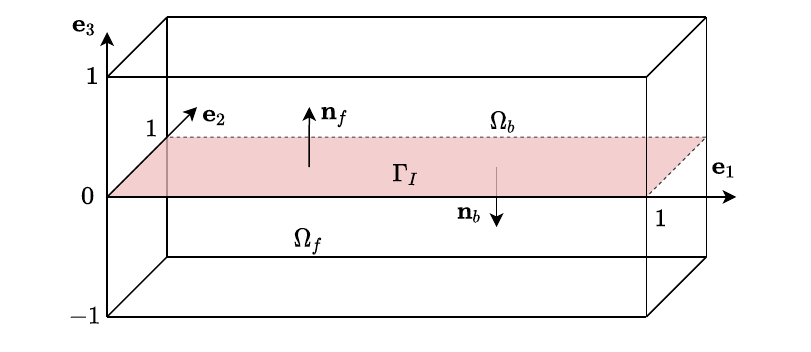}
\end{center}

\noindent Now---and for the remainder of the paper---we take $\Omega_b \equiv (0,1)^3$ to be a fully-saturated poroelastic structure. The dynamics on $\Omega_b$ are modeled by the Biot equations, described above: 
\begin{equation}\label{biot2}
\begin{cases}
    \rho_b \vec{u}_{tt} - \nabla\cdot\sigma^E(\vec{u}) + \alpha \nabla p_b = \vec{F}_b, &\text{ in } \Omega_b \times (0,T),\\
    [c_0p_b + \alpha \nabla \cdot \vec{u}]_t - \nabla \cdot [k\nabla p_b] = S, &\text{ in } \Omega_b \times (0,T).
\end{cases}
\end{equation}
Let $\Omega_f \equiv  (0,1)\times (0,1) \times (0,-1)$ be a region adjacent to $\Omega_b$,  filled with a fluid described by the incompressible Stokes equations. That is, with Eulerian velocity $\vec{v}$ and fluid pressure $p_f$, consider:
\begin{equation}\label{stokes1}
    \rho_f \vec{v}_t - 2\nu \text{div}~\mathbf D(\vec{v}) + \nabla p_f = \vec{F}_f,\quad \quad \nabla\cdot\vec{v} = 0, \quad  \text{ in } \Omega_f \times (0,T).
\end{equation}
Here, $\rho_f$ denotes the density of the fluid at a reference pressure and $\nu$ denotes the shear viscosity of the fluid. 
\begin{quote} \bf{For the remainder of this work, we will take $\rho_b,\rho_f,\nu,\alpha, \lambda,\mu>0$, and $c_0 \ge 0$.} \end{quote}

The two regions are adjoined at an interface $\Gamma_I = \partial\Omega_b \cap \partial \Omega_f = (0,1)^2$. We denote the normal vectors going out of the Biot and Stokes regions by $\vec{n}_b$ and $\vec{n}_f$, respectively. Then we orient the interface, so that on $\Gamma_I$, $\vec{n}_f  = \vec{e}_3 = -\vec{n}_b$. Denote the upper (Biot) and lower (Stokes) boundaries by
\begin{equation} \Gamma_b = \{\phi \in \partial \Omega_b~:~x_3=1\}~\text{ and }~ \Gamma_f = \{\phi \in \partial \Omega_f~:~x_3=-1\}.\end{equation} 
We consider, for these upper and lower boundary components, the homogeneous conditions
\begin{equation}\label{BC1}
    \vec{u} = \vec{0} ~~\text{ and }  ~~p_b = 0 ~\text{ on }\Gamma_b, \quad \text{ and } ~~\vec{v} = \vec{0} ~\text{ on } \Gamma_f, \quad \forall\, t \in (0,T).
\end{equation}
We consider periodic conditions on the lateral boundaries, collectively defined as $\Gamma_{\text{lat}}$, i.e., as the union of the four lateral faces (i.e., in the $x_1$ and $x_2$ directions). This is akin to identifying the lateral faces (yielding a truly toroidal domain), or, equivalently, identifying diametrically opposed quantities  for all {\em  traces which are appropriately defined}---see \eqref{295}.

Across the interface, we will enforce conservation of mass and momentum. Denoting the {\em total} poroelastic and fluid stress tensors as
\begin{align*}
    \sigma_b = \sigma_b(\vec{u},p_b) &= \sigma^E(\vec{u}) - \alpha p_b\vec{I}, \quad \quad \sigma_f = \sigma_f(\vec{v},p_f) = 2\nu\vec{D}(\vec{v}) - p_f\vec{I},
\end{align*}
respectively, we take the following coupling conditions on $\Gamma_I\times (0,T)$:
\begin{align}
    -k\nabla p_b\cdot\vec{e}_3 &= \vec{v} \cdot \vec{e}_3 - \vec{u}_t \cdot \vec{e}_3, \label{IC1}\\
    \beta(\vec{v} - \vec{u}_t)\cdot \btau &= -\btau \cdot \sigma_f \vec{e}_3,\label{IC2}\\
    \sigma_f\vec{e}_3 &= \sigma_b\vec{e}_3,\label{IC3}\\
    p_b &= -\vec{e}_3 \cdot \sigma_f\vec{e}_3 
    .\label{IC4}
\end{align}
We use the notation of  $\btau$ generically for  tangential vectors on $\Gamma_I$, i.e.,  $\btau = \vec{e}_i$ for $i = 1,2$. The {\it Beavers-Joseph-Saffman} \cite{mikelicBJS} condition in \eqref{IC2} says that the tangential stress is proportional to the {\it slip rate}, with proportionality constant $\beta>0$, called the {\it slip length} (see also \cite {yotovBJS,mikelicBJS}, and references therein). This condition generates frictional dissipation along the boundary, scaled by $\beta$, in the energy balance. The conservation of fluid mass across the interface is given by \eqref{IC1}, the {\it kinematic coupling condition}. The balance of total stresses in \eqref{IC3} is required by the conservation of momentum. The {\it dynamic coupling condition} \eqref{IC4} maintains the balance of normal components of the stresses across $\Gamma_I$.

Our first main result below centers on semigroup generation for an operator encoding the Cauchy problem \eqref{biot2}--\eqref{IC4}, taking  null sources, i.e.,  $\vec{F}_b=\mathbf 0, ~S=0$, and ~$\vec{F}_f=\mathbf 0$ in \eqref{biot2}--\eqref{stokes1}. We will subsequently obtain strong, generalized (semigroup), and weak solutions under appropriate hypotheses for non-zero $\vec{F}_b, S$, and $\vec{F}_f$ as they appear in \eqref{biot2} and \eqref{stokes1}.

\subsection{Energy Balance} \label{ebal}
To elucidate the relevant norms, and motivate the functional setup  ascribed to the Biot-Stokes dynamics, we will consider the energy balance in this section.
Formally testing the system \eqref{biot2}--\eqref{stokes1} with $(\vec{u}_t, p_b, \vec{v})$, resp., zeroing out  sources, and assuming sufficient regularity of the solution to justify integration by parts and invoke boundary conditions (in \eqref{IC1}--\eqref{IC3} at the interface $x_3=0$ and periodic boundary conditions on $\Gamma_{\text{lat}}$), we obtain:
\begin{align*}
0=&~    \rho_b(\vec{u}_{tt},\vec{u}_t)_{\Omega_b} ~- (\nabla\cdot\sigma^E(\vec{u}),\vec{u}_t)_{\Omega_b} + \alpha (\nabla p_b, \vec{u}_t)_{\Omega_b} + c_0(p_{b,t},p_b)_{\Omega_b} + \alpha (\nabla\cdot \vec{u}_t,p_b)_{\Omega_b} \\[.1cm]
    &- (\nabla\cdot [k\nabla p_b],p_b)_{\Omega_b}+~ \rho_f(\vec{v}_t,\vec{v})_{\Omega_f} -2\nu (\nabla\cdot \mathbf D(\bv),\vec{v})_{\Omega_f} + (\vec{v},\nabla p_f)_{\Omega_f}\\[.2cm]
    =&~ \frac{1}{2}\frac{d}{dt}\Big[\rho_b\|\vec{u}_t\|_{0,b}^2 + \|\vec{u}\|_E^2 + c_0\|p_b\|_{0,b}^2 + \rho_f\|\vec{v}\|_{0,f}^2\Big]  + k\|\nabla p_b\|_{0,b}^2 + 2\nu\|\vec{D}(\vec{v})\|_{0,f}^2 + \beta \|(\vec{v}-\vec{u}_t)\cdot\boldsymbol{\tau} \|_{\Gamma_I}^2,
\end{align*}
where we have specified a norm through the bilinear form \begin{equation}\label{korny} \|\vec{u}\|^2_E \equiv  (\sigma^E(\vec{u}),\vec{D}(\vec{u}))_{\mathbf L^2(\Omega)},\end{equation} which is equivalent to the standard $\mathbf H^1(\Omega_b)$ norm by Korn and Poincar\'e's inequalities \cite{kesavan,temam}. Then, defining the total (trajectory) energy $e(t)$ as \begin{equation} \label{energy} e(t) \equiv  \frac{1}{2}\Big[\rho_b\|\vec{u}_t\|_{0,b}^2 + \|\vec{u}\|_E^2 + c_0\|p_b\|_{0,b}^2 + \rho_f\|\vec{v}\|_{0,f}^2\Big],\end{equation} and a dissipator $d_0^t$ as \begin{equation}\label{dissipator} d_0^t \equiv \int_0^t\big[k\|\nabla p_b\|_{0,b}^2 
+ 2\nu \|\vec{D}(\vec{v})\|_{0,f}^2 
+ \beta \|(\vec{v}-\vec{u}_t)\cdot\boldsymbol{\tau} \|_{\Gamma_I}^2 \big] d\tau,\end{equation}
we obtain the formal energy balance:
\begin{equation} \label{eident} e(t) +d_0^t = e(0).\end{equation}
This identity will hold immediately for smooth solutions (e.g., in the domain of the semigroup generator), and, will be extended to semigroup (generalized) solutions for finite energy data.  It will also follow (as an inequality) for the weak solutions, constructed in Section \ref{weaksolsec} through an approximation argument. 
We  refer to the {\em energy inequality} as
\begin{align} \label{eidentt}
e(t) + d_0^t \le & ~ e(0) \\ \label{eidenttt}
e(t) + d_0^t \lesssim & ~ e(0) + \int_0^t\big[||\mathbf F_f(\tau)||^2_{\vec{L}^2(\Omega_b)}+||\mathbf F_b(\tau)||^2_{[\mathbf H^1_{\#,*}(\Omega_f)\cap \mathbf V]'}+||S(\tau)||_{[H^1_{\#,*}(\Omega_b)]'}^2 \big]d\tau
\end{align}
in the first case when $\mathbf F_f = \mathbf F_b = S \equiv 0$, and in the latter case when the sources are present---the function spaces being defined in detail in Section \ref{absModel}. 

\begin{remark} The energy identity in \eqref{eident}, having dissipation in both dynamics, naturally elicits the question of semigroup decay, as investigated in \cite{thermo,AGM,avalos,redbook}. Specifically, we note that the Biot system does contain dissipative features (via $p_b$ on $\Omega_b$ and through the tangential slip-friction on $\Gamma_I$ in \eqref{dissipator}); on the other hand, the strength of the damping is not immediately clear. In particular, the dissipative features of thermoelasticity, and the partially-damped Stokes-Lam\'e system, are non-trivial. Thus the stability and decay rates for the coupled Biot-Stokes system are of interest.\end{remark}

\section{Relevant Definitions and Main Results}\label{mainresults}
We first restate the initial-boundary-value problem of interest here (simplifying some operators):
\begin{align}\label{systemfull}
    \rho_b \vec{u}_{tt} - \nabla\cdot\sigma^E(\vec{u}) + \alpha \nabla p_b =&~ \vec{F}_b, &\text{ in } \Omega_b \times (0,T),\\ \label{systemfull0}
    [c_0p_b + \alpha \nabla \cdot \vec{u}]_t - k\Delta p_b =&~ S, &\text{ in } \Omega_b \times (0,T),\\
    \rho_f \vec{v}_t - \nu \Delta \vec{v} + \nabla p_f =&~ \vec{F}_f,\quad \quad \nabla\cdot\vec{v} = 0, &  \text{ in } \Omega_f \times (0,T).\label{systemfull1}
\end{align}
Upper and lower boundary conditions are given  by:
\begin{equation}\label{uplow}
    \vec{u} = \vec{0} ~~\text{ and } ~~ p_b = 0~ \text{ on }~\Gamma_b, \quad ~\text{ and } ~~\vec{v} = \vec{0} ~\text{ on }~ \Gamma_f, ~~~t \in (0,T).\end{equation}
The interface conditions on $\Gamma_I\times (0,T)$ are given by:
\begin{align}
    -k\nabla p_b\cdot\vec{e}_3 &= \vec{v} \cdot \vec{e}_3 - \vec{u}_t \cdot \vec{e}_3, \label{IC1*}\\
    \beta(\vec{v} - \vec{u}_t)\cdot \btau &= -\btau \cdot \sigma_f \vec{e}_3,\label{IC2*}\\
    \sigma_f\vec{e}_3 &= \sigma_b\vec{e}_3,\label{IC3*}\\
    p_b &= -\vec{e}_3 \cdot \sigma_f\vec{e}_3. 
    \label{IC4*}
\end{align}
We recall that periodic conditions  are taken on the lateral faces $\Gamma_{\text{lat}}\times (0,T)$, those these will be specified in  detail below (in \eqref{295} and the description of the semigroup generator in Section \ref{operator}). \\ Finally, initial conditions are prescribed for the states
\begin{equation}\label{endfull}\vec{y}(0) = [\vec{u}(0),\vec{u}_t(0),p_b(0),\vec{v}(0)]^T,\end{equation}
in line with the energy $e(t)$, as defined in \eqref{energy}.
\subsection{Operators and Spaces for the Inertial Biot-Stokes System}\label{absModel}
In this section we consider $c_0>0$.
For either domain $\Omega_i$ ($i=b,f$), let us introduce the notation $H^s_\#(\Omega_i)$ for the space of all functions from $H^s(\Omega_i)$ that are periodic in directions $x_1$ and $x_2$ (with spatial period 1): {\small
\begin{equation}\label{295}    {H}^s_\#(\Omega_i) := \{f \in {H}^s(\Omega_i)~:~ \gamma_j[f]|_{x_1 = 0} = \gamma_j[f]|_{x_1 = 1},~\gamma_j[f]|_{x_2 = 0} = \gamma_j[f]|_{x_2 = 1} \text{ for } j=0,1,...,s-1\},\end{equation}}

\noindent where we have denoted by $\gamma_j$ the $j$th standard trace map \cite{kesavan} on $\Gamma_{\text{lat}}$. This is to say: $\gamma_0(.)$ is the Dirichlet trace operator and normal derivative $\gamma_1(.)=\gamma_0[\nabla(.)\cdot \mathbf n]$, each restricted to a given face of $\Gamma_{\text{lat}}$. We take the analogous definition for the vector-valued $\mathbf H^s_\#(\Omega_i)$. Note that, at this stage, we only specify periodicity in $H^s$ in senses where the appropriate traces are defined, though we will have dual/weaker notions of periodic traces below. 

Now let
\begin{align*}
    \vec{U} &\equiv \{\vec{u} \in \vec{H}_\#^1(\Omega_b)~:~~ \vec{u}\big|_{\Gamma_b} = \vec{0}\},\\
    \vec{V} &\equiv \{\vec{v} \in \vec{L}^2(\Omega_f) ~:~~ \text{div}~ \vec{v} \equiv 0 \text{ in }\Omega_f; \,  \vec{v}\cdot\vec{n} \equiv 0\text{ on } \Gamma_f\},\end{align*}
    and define the state space
    \begin{align}
    X &\equiv  \vec{U} \times \vec{L}^2(\Omega_b) \times L^2(\Omega_b) \times \vec{V}.
\end{align}
We will also use the notation below, $\vec{H}^1_{\#,\ast}(\Omega_f)$ and ${H}^1_{\#,\ast}(\Omega_b)$, where the subscript $*$ denotes a zero Dirichlet trace on $\Gamma_f$ and $\Gamma_b$, resp., that is, on the $\{x_3=-1\}$ or $\{x_3=1\}$ faces.
We will subsequenty topologize the space $H^1_{\#,*}(\Omega_b)$ with the standard gradient norm, via the Poincar\'e inequality \cite{kesavan,bgsw}; thus $$||.||_{H^1_{\#,*}(\Omega_b)} \equiv ||\nabla ~.||_{L^2(\Omega_b)}.$$
The topology of $\mathbf U$ is induced by the norm $||\cdot||_E$ introduced before in \eqref{korny}, namely through the bilinear form \begin{equation}\label{bilform} a_E(\cdot,\cdot)=(\sigma^E(\cdot),\mathbf D(\cdot))_{\mathbf L^2(\Omega_b)}.\end{equation} As noted before, this norm is equivalent to the full $\mathbf H^1(\Omega_b)$ norm in $\mathbf U$. 

We  denote $$\cE_0 = -\rho_b^{-1}\nabla\cdot\sigma^E~\text{ and }~ A_0 = -c_0^{-1}k\Delta$$ in reference to the differential action of the principle operators in the Biot dynamics. We note that, unlike previous related work \cite{avalos*,AGM,avalos}, we do not define these as proper operators on individual (unbounded operator) domains. Indeed, we will consider their action as part of the overall dynamics operator, defined in \eqref{diffaction} and give the full operator's domain in Definition \ref{diffdomain}.

To accommodate the physical quantities $c_0,\rho_b,\rho_f>0$, we introduce some equivalent topologies via the inner-product on the finite-energy space $X$:
\begin{align}\label{innerX}
(\mathbf y_1, \mathbf y_2)_X = a_E(\bu_1,\bu_2)+\rho_b(\bw_1,\bw_2)_{\mathbf L^2(\Omega_b)}+c_0(p_1,p_2)_{L^2(\Omega_b)}+\rho_f(\mathbf v_1,\mathbf v_2)_{\mathbf L^2(\Omega_f)}.
\end{align}
 We consider the state variable $\vec{w} = \vec{u}_t$  as the elastic velocity, which is standard for hyperbolic systems. 

We then consider a Cauchy problem which captures the dynamics of the full Biot-Stokes system. Namely: find $\vec{y} = [\vec{u},\vec{w},p_b,\vec{v}]^T \in \mathcal{D}(\cA)$ (resp. $X$) such that
\begin{equation}\label{cauchy}
   \dot{ \vec{y}}(t) = \cA \vec{y}(t)+\mathcal F(t); \quad \vec{y}(0) = [u_0,u_1,p_0,v_0]^T \in \mathcal{D}(\cA) \text{ (resp. } X),
\end{equation}
where $\mathcal F(t) = [\mathbf 0, \mathbf F_b(t), S(t), \mathbf F_f(t)]^T$ (of appropriate regularity)
and the action of the  operator $\cA:\mathcal{D}(\cA)\subset X \to X$ is provided by
\begin{equation}\label{A}
    \cA \equiv \begin{pmatrix}
        0 & \vec{I} & 0 & 0\\
        -\cE_0 & 0 & -\alpha\rho_b^{-1} \nabla & 0\\
        0 & -\alpha c_0^{-1} \nabla\cdot & -A_0 & 0\\
        0 & 0 & \rho_f^{-1}G_1 & \rho_f^{-1}[\nu\Delta + G_2 + G_3]
    \end{pmatrix}.
\end{equation}
The operators $G_1, G_2,$ and $G_3$ above are certain Green's maps which enable the elimination of the pressure (and carry boundary information) \cite{avalos}. They will be defined properly in Section \ref{elimp}. Additionally, the precise definition of the domain of $\cA$ will the topic of Section \ref{gendef}. 

As is typical in  defining the action of a prospective generator, such as $\cA$, we must define $\mathcal{D}(\cA)$ strictly enough to permit dissipativity calculations (i.e., appropriate integration by parts formulae must hold, at least in dual spaces), but not so strictly so as to preclude recovering a weak solution in the domain  for the resolvent system \cite{pazy}. We note that {\em this is a challenge in this context}, and {\em defining the operator (and hence its domain) in a useful way constitutes a central contribution here}; in doing so, we must be particularly mindful of the regularity properties of the elliptic maps $G_i$, $i=1,2,3$, which is why we relegate that discussion to its own section.

\subsection{Definition of Weak Solutions}
The motivation for our functional setup and  notion of weak solutions is due to  \cite{multilayered}. In that reference, a Stokes-Biot system is considered, coupled via a 2D poroelastic plate evolving at the interface. We here consider the analogous definition of weak solutions to \cite{multilayered}, which are equivalent after ``zeroing out" the plate variables in that approach. Of course, there are other viable definitions of weak solutions for coupled hyperbolic-parabolic problems based, for instance, on: time-independent test functions, rewriting the system in its first order formulation, or enforcing Dirichlet-type coupling conditions in the test space. 
{\bf We will now define weak solutions for all $c_0\ge 0$.}

To that end, define the spaces
\begin{align*}
    \mathcal{V}_b &= \{\vec{u} \in L^\infty(0,T;\vec{U}) ~:~ \vec{u} \in W^{1,\infty}(0,T;\vec{L}^2(\Omega_b))\}, \\
    \mathcal{Q}_b &= \left\{p \in L^2(0,T; H^1_{\#,*}(\Omega_b))~:~c_0^{1/2}p \in L^\infty(0,T;L^2(\Omega_b))\right\}, \\
    \mathcal{V}_f &= L^\infty(0,T;\vec{V}) \cap L^2(0,T; \vec{H}^1_{\#,\ast}(\Omega_f) \cap \vec{V}).
\end{align*}
Then, the {\bf weak solution space} is
\newcommand\Vsol{\mathcal{V}_{\text{sol}}}
\newcommand\Vtest{\mathcal{V}_{\text{test}}}    
\[
    \mathcal{V}_\text{sol} = \mathcal{V}_b\times \mathcal{Q}_b\times\mathcal{V}_f,~\text{ for } [\bu,p_b,\bv]^T
\]
and the {\bf test space} is
\[
    \mathcal{V}_\text{test} = C^1_0([0,T); \vec{U} \times  H^1_{\#,*}(\Omega_b) \times (\vec{H}^1_{\#,\ast}(\Omega_f)\cap \vec{V})),~\text{ for } ~[\bxi,q_b,\bzeta]^T,
\]
where we denote $C^1_0$ for continuously differentiable functions with compact support on $(-\epsilon,T)$ for all $\epsilon>0$. 
Note that the solution space is built to accommodate any $c_0\geq 0$ in the single definition.

To give the weak formulation, we introduce the following time-space notations for convenience:
 $((\cdot, \cdot))_{\mathscr O}$ denotes an inner-product on $L^2(0,T;L^2(\mathscr O))$, for a spatial domain $\mathscr O$. Similarly, for a Sobolev type space defined on $\mathscr O$, say $W(\mathscr O)$, we will denote $(\langle \cdot,\cdot \rangle)_{\mathscr O}$  as pairing between $L^2(0,T;W(\mathscr O))$ and its dual $L^2(0,T;[W(\mathscr O)]')$.
We also suppress trace operators below, noting spatial restrictions from subscripts. We will consider initial conditions of the form:
     \begin{equation}\label{weakICs}
        \vec{u}(0) = \vec{u}_0, ~~ \vec{u}_t(0) = \vec{u}_1, ~~ [c_0p_b + \alpha\nabla\cdot\vec{u}](0) = d_0, ~~ \vec{v}(0) = \vec{v}_0.
    \end{equation}
   Finally, we consider source data of regularity (at least): 
    $$\mathbf F_f \in L^2\big(0,T;(\vec{H}^1_{\#,\ast}(\Omega_f)\cap \vec{V})'\big),~~\mathbf F_b \in L^2(0,T;\vec{L}^2(\Omega_b)),~~\text{and}~~S \in L^2(0,T; (H^{1}_{\#,*}(\Omega_b))').$$

\begin{definition}\label{weaksols}
    We say that $[\vec{u},p_b,\vec{v}]^T \in \mathcal{V}_{\text{sol}}$ is a weak solution to \eqref{systemfull}--\eqref{endfull} if for every test function $[\bxi,q_b,\bzeta]^T \in \mathcal{V}_\text{test}$ the following identity holds:
    \begin{align}\label{weakform}
        -&~ \rho_b((\vec{u}_t,\bxi_t))_{\Omega_b} + ((\sigma_b(\vec{u},p_b), \nabla \bxi))_{\Omega_b} - ((c_0 p_b + \alpha\nabla\cdot\vec{u}, \partial_t q_b))_{\Omega_b} + ((k\nabla p_b,\nabla q_b))_{\Omega_b} \nn\\
        &- \rho_f((\vec{v},\bzeta_t))_{\Omega_f} + 2\nu((\vec{D}(\vec{v}),\vec{D}(\bzeta)))_{\Omega_f} + (({p_b,(\bzeta-\bxi)\cdot\vec{e}_3}))_{\Gamma_I} - (({\vec{v}\cdot\vec{e}_3,q_b}))_{\Gamma_I} \nn\\
        &- (({\vec{u}\cdot\vec{e}_3,\partial_tq_b}))_{\Gamma_I} + \beta (({\vec{v}\cdot\btau,(\bzeta - \bxi)\cdot\btau}))_{\Gamma_I} + \beta(({\vec{u}\cdot\btau, (\bzeta_t - \bxi_t)\cdot\btau}))_{\Gamma_I} \nn\\
        =&~ \rho_b(\mathbf u_1,\bxi)_{\Omega_b}\big|_{t=0} + (d_0,q_b)_{\Omega_b}\big|_{t=0} + \rho_f(\vec{v}_0,\bzeta)_{\Omega_f}\big|_{t=0} + ({\vec{u}_0\cdot\vec{e}_3,q_b})_{\Gamma_I}\big|_{t=0} \nn\\
        &~- \beta({\vec{u}_0\cdot\btau, (\bzeta - \bxi)\cdot\btau})_{\Gamma_I}\big|_{t=0} + ((\vec{F}_b,\bxi))_{\Omega_b} + (\langle S,q_b\rangle)_{\Omega_b} + (\langle \vec{F}_f,\bzeta\rangle)_{\Omega_f},
    \end{align}
 for $\btau=\mathbf e_i,~i=1,2$.
\end{definition}
\begin{remark}
In the above definition, as below, the regularity of velocity traces of the form $\gamma_0[\bu_t]$ will be at issue. We note that such traces can be interpreted distributionally for $\bu \in L^2(0,T;\mathbf H^1(\Omega_b))$ as $\partial_t[\gamma_0[\mathbf u]]$, to which additional regularity can be demonstrated or ascribed.
\end{remark}
In the above formulation, several time differentiations are pushed to the time-dependent test function, as is a common treatment of hyperbolic problems as well as in \cite{multilayered}. It is worth noting, that the weak formulation circumvents {\em a priori trace regularity issues} at the interface associated to $\gamma_0[\bu_t]$ through this temporal integration by parts  (in the normal and tangential components of $\bu_t$ on $\Gamma_I$). Indeed, we note that weak solutions, as defined above, have only that $\bu_t \in L^{\infty}(0,T;\mathbf L^2(\Omega_b))$ for interior regularity.  {\em Natural} Neumann-type conditions in \eqref{IC1*}--\eqref{IC4*}  are unproblematic, as such traces do not appear in the weak formulation and will be a posteriori justified as elements of $H_{\#}^{-1/2}(\Gamma_I)$-type spaces (as characterized in Section \ref{gendef}). For {\em for a constructed weak solution}, owing to the energy identity, we will infer some regularity of temporal derivatives in the statement of \eqref{weakform} {\em a posteriori}. For instance, terms such as $\bu_{tt}$ and $[c_0p_b+\nabla \cdot \bu]_t$ in the interior, as well as $\gamma_0[\bu]_t \cdot \mathbf e_3$ will exists as elements of certain dual spaces. However, characterizing a space of distributions for these time derivatives is not straightforward, owing the coupled nature of the problem and our chosen weak formulation in \eqref{weakform} above, as well as very sensitive to the regularity of the data $\mathbf F_b, ~\mathbf F_f,$ and $S$. The term $\gamma_0[(\bv-\bu_t)\cdot \btau]$ will be a properly defined element of $L^2(0,T;L^2(\Gamma_I))$ for any weak solution constructed as a subsequential limit from smooth solutions satisfying the energy relation \eqref{eidenttt}. 

We point out that one of the main contributions of \cite{multilayered} is the development of the above notion of weak solutions. That reference also establishes that point-wise solutions are weak, and that weak solutions with additional regularity recover the equations and boundary conditions point-wise. The analysis of the dynamics here proceeds {\em mutatis mutandis},  zeroing out the plate contributions across the interface. As such, we do not demonstrate the consistency of the weak solution definition for the sake of conciseness but  refer directly to \cite[Section 3]{multilayered}. 

\begin{remark}[Regularity and Weak Solutions I] \label{regweaksol1} The gap between $\mathcal V_{\text{test}}$ and $\mathcal V_{\text{sol}}$ is precisely at issue with both regularity and uniqueness for weak solutions. The above weak formulation bypasses certain regularity issues associated to the coupling conditions; this is unproblematic for the construction of weak solutions. However, the ``price to pay" concerns a posteriori regularity of weak solutions. This is a peculiar feature for {\em linear dynamics}, but nonetheless unavoidable with this particular hyperbolic-parabolic coupling which prevents the decoupling of constituent interior dynamics; the boundary conditions involving $\bu_t$ are not ``separated" from other interior (low regularity) terms. Moreover, this elicits a particular challenge for the development of any regularity theory of weak solutions, not altogether unexpected from hyperbolic-parabolic coupled problems with ill-defined velocity traces at the interface. Since the problem is linear, even the situation with zero forcing and zero Cauchy data is not immediately resolved, i.e., the aforesaid regularity issues translate into issues with uniqueness of weak solutions. A generic weak solution need not satisfy the energy identity (since elements of $\mathcal V_{\text{sol}}$ are not readily used as test functions). A future work will specifically explore these issues. As a final, comment, we mention that the classical reference \cite{ball} includes an abstract construction of weak solutions from any given semigroup generator, {\em as well as uniqueness of said weak solutions}. We forgo using such an abstract approach here, since it entails a particular weak form at the abstract level which involves {\em smooth} test functions in domain of the adjoint of the semigroup generator. Since, we here work directly with a weak form inherited from previous work \cite{multilayered}, using physically-motivated test functions, we relegate this analysis to a future work.
\end{remark}

\subsection{Statement of Main Results}
The approach to solutions in this paper is based on  $C_0$-semigroups \cite{pazy}. As such, we remark that the Cauchy problem \eqref{cauchy}  can be solved in two standard ways: {\em strongly} or {\em in a generalized} sense. Strong solutions are  differentiable functions with the time derivative taking values in $X$, and are also continuous functions into $\mathcal D(\mathcal A)$; such strong solutions emanate from $\mathbf y_0 \in \mathcal D(\mathcal A)$ with appropriate regularity assumptions on $\mathcal F$ and point-wise satisfy \eqref{cauchy}. We may also consider data $\mathbf y_0 \in X$, yielding a generalized solution to the Cauchy problem \eqref{cauchy}; this is a $C([0,T];X)$ limit of a sequence $\mathbf y_n(t)$ of strong solutions to the Cauchy problem \eqref{A}, each with initial data $\mathbf y_0^n \in \mathcal D(\cA)$, such that $\mathbf y_0^n \to \mathbf y_0$ in $X$ \cite{pazy}. Generalized solutions solve the time-integrated version of \eqref{cauchy} 
and are defined by the variation of parameters formula, referencing the semigroup $e^{\mathcal A t}$ for the homogeneous problem
\begin{equation}\label{varpar} \mathbf y(t)=e^{\mathcal At}\mathbf y_0 +\int_0^te^{\mathcal A(t-s)}\mathcal F(s)ds.\end{equation}
 For linear problems, generalized solutions are also shown to be weak \cite{ball,pazy}. 

Our central supporting theorem in this treatment is that of semigroup generation for the operator $\cA$  on the space $X$. 
\begin{theorem}
\label{th:main1}
    The operator $\mathcal A$ on $X$ (defined by \eqref{A}, with domain $\mathcal D(\mathcal A)$ given in Definition \ref{diffdomain}) is the generator of a strongly continuous semigroup $\{e^{\cA t}: t\geq 0\}$ of contractions on $X$. Thus, for $\mathbf y_0 \in \mathcal D(\mathcal A)$, we have $e^{\mathcal A\cdot}\mathbf y_0 \in C([0,T];\mathcal D(\mathcal A))\cap C^1((0,T);X)$ satisfying \eqref{cauchy} in the {\em strong sense} with $\mathcal F=[\mathbf 0,\mathbf 0, 0,\mathbf 0]^T$; similarly, for $\mathbf y_0 \in X$, we have $e^{\mathcal A\cdot}\mathbf y_0 \in C([0,T];X)$ satisfying \eqref{cauchy} in the {\em generalized sense} with $\mathcal F=[\mathbf 0, \mathbf 0, 0,\mathbf 0]^T$.
\end{theorem}

As a corollary, we will obtain strong and weak solutions to \eqref{systemfull}--\eqref{IC4*}, including  in the non-homogeneous case with generic forces. By strong solutions to  \eqref{systemfull}--\eqref{IC4*} we mean weak solutions, as in Definition \ref{weaksols}, with enough regularity that the equations in \eqref{systemfull}--\eqref{IC4*} also hold point-wise. In fact, these will be constructed from semigroup solutions with smooth data. For the results presented below, we do not explicitly mention generalized solutions, as these are tied specifically to the semigroup framework (though the well-posedness of generalized solutions also follows from Theorem \ref{th:main1}).

\begin{corollary}\label{coro}
\hfill
\begin{enumerate}
 \item[(I)] Take $\mathbf F_f\equiv \mathbf 0,$ $\mathbf F_b\equiv \mathbf 0$, and $S\equiv 0$. 

a) Suppose $\mathbf y_0 \in \mathcal D(\mathcal A)$. Then strong solutions to \eqref{systemfull}--\eqref{IC4*}  exist and are unique. For such strong solutions, the energy relation \eqref{eident} holds.

b) Suppose $\mathbf y_0 \in X$. Then weak solutions as defined in Definition \ref{weaksols} exist. For such weak solutions, the energy inequality \eqref{eidentt} holds.

 \item[(II)] Suppose $\mathbf y_0 \in \mathcal D(\mathcal A)$ and $\mathcal F = [\mathbf 0,\mathbf  F_b, S, \mathbf F_f]^T \in H^1(0,T;X)$. Then strong solutions to \eqref{cauchy} exist and are unique. 
 
\item[(III)] Suppose $\mathbf y_0 \in X$ and $\mathbf F_f \in L^2\big(0,T;(\vec{H}^1_{\#,\ast}(\Omega_f)\cap \vec{V})\big)'$, $\mathbf F_b \in L^2(0,T;\vec{L}^2(\Omega_b))$, and $S \in L^2(0,T; [H^{1}_{\#,*}(\Omega_b)]')$. Then weak solutions to \eqref{cauchy} exist; the constructed solutions satisfy the energy inequality \eqref{eidenttt}.
\end{enumerate}
\end{corollary}
\begin{remark} We remark that strong solutions with $\mathcal F \neq \mathbf 0$ also satisfy an analog of the energy balance in \eqref{eident} (which is for the homogeneous case). In particular, we may write 
{\small \begin{equation*}
e(t)+d_0^t=e(0)+\int_0^t\big[\langle\mathbf F_f,\bv\rangle_{(\vec{H}^1_{\#,\ast}(\Omega_f)\cap \vec{V})'\times (\vec{H}^1_{\#,\ast}(\Omega_f)\cap \vec{V})}+(\mathbf F_b,\bu_t)_{\mathbf L^2(\Omega_b)}+\langle S, p_b\rangle_{ [H^{1}_{\#,*}(\Omega_b)]'\times   H^{1}_{\#,*}(\Omega_b)}\big] d\tau,
\end{equation*}} though this is not the unique way to express the identity, as dependent on the regularity of $\mathcal F$.  \end{remark}
\begin{remark} Although we provide straightforward criteria for the data to yield strong solutions from \cite{pazy}, we of course note that more minimal assumptions can be made. See \cite[Chapter 4]{pazy}.
\end{remark}

This brings us to our second main theorem: existence of weak solutions for \eqref{systemfull}--\eqref{IC4*}, independent of the storage parameter $c_0\ge 0$.
\begin{theorem}\label{th:main2} Suppose $\mathbf y_0 \in X$ and $\mathbf F_f \in L^2(0,T;\vec{H}^1_{\#,\ast}(\Omega_f)\cap \vec{V})'$, $\mathbf F_b \in L^2(0,T;\mathbf L^2(\Omega_b))$, and $S \in L^2(0,T; [H^{1}_{\#,*}(\Omega_b)]')$.
   Then for any $c_0\ge 0$, the dynamics admit weak solutions in the sense of Definition \ref{weaksols}. A constructed weak solution satisfies the energy inequality in \eqref{eidenttt}. A posteriori, for a trajectory ~$\mathbf y(t) = [\mathbf u, \mathbf u_t, p_b, \mathbf v]^T$, $t \in [0,T]$, corresponding to a weak solution, we infer the additional trace regularity of \begin{equation} \gamma_0[(\mathbf v-\mathbf u_t) \cdot \btau] \in L^2(0,T; L^2(\Gamma_I)). \end{equation} 
\end{theorem}
\begin{remark} Since, for weak solutions, $\mathbf v \in L^2(0,T; \vec{H}^1_{\#,\ast}(\Omega_f) \cap \vec{V})$, we know, by the trace theorem \cite{kesavan}, $\gamma_0[\mathbf v] \in L^2(0,T; \mathbf H^{1/2}(\Gamma_I))$. 
From this we deduce the global regularity of $\gamma_0[\mathbf u_t\cdot \btau] \in L^2(0,T;  L^2(\Gamma_I)),~~\btau = \mathbf e_1,\mathbf e_2$, a posteriori, also noting that $\bu_t$ has no a priori, well-defined trace. We will be able to give meaning to $\gamma_0[\bu_t\cdot \mathbf e_3]$ (as well as $\bu_{tt}$ and $\bv_t$) in certain dual senses below. See Section \ref{duals}, after the proof of Corollary \ref{coro} is concluded.
\end{remark}

\begin{remark}[Regularity and Weak Solutions II] \label{regweaksol2} While the generation of the semigroup provides ``smooth" solutions through regular data, the question of characterizing the domain of the generator remains. And, as mentioned in Remark \ref{regweaksol1}, the issue of regularity of weak solutions for hyperbolic-parabolic systems, such as this one, is a notorious one. While we know that smooth weak solutions recover the strong form of the problem \cite{multilayered}, it is nontrivial to characterize smoothing in the dynamics from smooth data. 
\end{remark}

\section{Dynamics Operator}\label{operator}
To define the dynamics operator---which will be shown to generate a $C_0$-semigroup on $X\equiv  \vec{U} \times \vec{L}^2(\Omega_b) \times L^2(\Omega_b) \times \vec{V}$---we will proceed to eliminate the Stokes pressure $p_f$ via elliptic theory, as motivated by \cite{avalos}. In fact, this is one of the main points of distinction with our approach here for incompressible Stokes and that in the  reference \cite{showfiltration} which treats slightly compressible Stokes. 
\subsection{Elimination of Fluid-Pressure and Associated Green's Maps}\label{elimp}
In line with \cite{avalos*}, we eliminate the pressure through appropriately defined Green's mappings \cite{redbook}. As in these previous references, elimination of the pressure through some form of a Leray projection is not practicable, owing to the coupled dynamics at the interface. 

In what follows, we note that for finite-energy (weak or semigroup) solutions, we will work with $p_f \in L^2(\Omega_f)$ for $a.e.$ $t$. This pressure acts as a Lagrange multiplier to enforce the divergence-free condition on the associated $\mathbf v$. Such $L^2$ objects do not have sufficient regularity to admit traces, and thus do not admit boundary conditions (e.g., the specification of being laterally periodic). However, in defining the auxiliary elliptic problems to eliminate $p_f$, we will need to appeal to the strong formulation of two elliptic problems in order to produce relevant ``lifting" operators which can be extended to all Sobolev indices. For instance, we declare that if a solution has $p_f \in H^1(\Omega_f)$, the (Dirichlet) lateral periodic boundary conditions will be in force for $p_f$; higher order boundary conditions are enforced in the same way.

Suppose $\mathbf F_f \equiv 0$ or $\mathbf F_f \in \mathbf V$ point-wise in time. Then we consider the fluid-pressure sub-problem, point-wise in time:
\begin{equation}\label{bigellip}
    \begin{cases}
        \Delta p_f = 0 &\text{ in } \Omega_f,\\
        \partial_{\vec{e}_3}p_f = \nu \Delta\vec{v} \cdot\vec{e}_3 &\text{ on } \Gamma_f,\\
        p_f = p_b + 2\nu\vec{e}_3\cdot \vec{D}(\vec{v})\vec{e}_3  &\text{ on } \Gamma_I,\\
        p_f ~\text{ is }~x_1\text{-periodic} ~\text{ and }~ x_2\text{-periodic}& \text{ on } \Gamma_{\text{lat}}.
    \end{cases}
\end{equation}
 Above, Laplace's equation is obtained from formally taking the divergence of the dynamic fluid equation in \eqref{stokes1}; the condition on $\Gamma_f$ is obtained as the inner product of the dynamic fluid equation with $\vec{e}_3$, restricted to $\Gamma_f$; and the condition on $\Gamma_I$ is simply \eqref{IC4}.

\begin{remark}
    There is some subtlety in the choice of the Dirichlet condition that we consider in \eqref{bigellip}. As the fluid-pressure appears in the interface condition \eqref{IC3} as well as \eqref{IC4}, we may instead take the inner product of \eqref{IC3} with $\vec{e}_3$ and write an alternative fluid-pressure sub-problem
    \begin{equation}\label{bigellip2}
        \begin{cases}
            \Delta p_f = 0 &\text{ in } \Omega_f,\\
            \partial_{\vec{e}_3}p_f = \nu \Delta\vec{v} \cdot\vec{e}_3 &\text{ on } \Gamma_f,\\
            p_f = \alpha p_b - \vec{e}_3 \cdot \sigma^E(\vec{u})\vec{e}_3 + 2\nu\vec{e}_3\cdot \vec{D}(\vec{v})\vec{e}_3  &\text{ on } \Gamma_I,\\
            p_f ~\text{ is }~x_1\text{-periodic} ~\text{ and }~ x_2\text{-periodic}& \text{ on } \Gamma_{lat}.
        \end{cases}
    \end{equation}
We then have that, from these two expressions above for $p_f$, one of either Biot or Stokes pressures can be eliminated on $\Gamma_I$.
\end{remark}

With the representation \eqref{bigellip} in mind, we define the Neumann and Dirichlet Green's maps $N_f$ and $D_I$ by
\[
    \phi \equiv N_fg \iff 
    \begin{cases}
        \Delta \phi = 0 & \text{in} ~\Omega_f;\\
        \partial_{\vec{e}_3} \phi = g &  \text{on} ~ \Gamma_f;\\
        \phi = 0 &  \text{on} ~ \Gamma_I;
    \end{cases}
    \quad \quad
    \psi \equiv D_I h \iff
    \begin{cases}
        \Delta \psi = 0 &  \text{in} ~\Omega_f;\\
        \partial_{\vec{e}_3} \psi = 0 &  \text{on} ~ \Gamma_f;\\
        \psi = h &  \text{on} ~\Gamma_I;
    \end{cases}
\]
where we have suppressed the lateral periodicity implicit in the formulations on $\Gamma_{\text{lat}}$ (as in \eqref{bigellip}).
The maps $N$ and $D$ are well-defined via classical elliptic theory for sufficiently regular data\footnote{in this setting, as periodicity effectively eliminates corners and preserves the  boundary regions as ``separated" and smooth}, and can be extended (through transposition) for weaker data \cite{redbook,lionsmag}:
\begin{align*}
    N_f &\in \mathcal{L}(H^s(\Gamma_f),  H^{s + \frac{3}{2}}(\Omega_f)),\quad D_I \in \mathcal{L}(H^s(\Gamma_I), H^{s + \frac{1}{2}}(\Omega_f)),\quad s \in \mathbb{R}.
\end{align*}

Then for given $p_b$ and $\vec{v}$ in appropriate spaces, we can write $p_f = \Pi_1p_b + \Pi_2\vec{v} + \Pi_3\vec{v}$, where
\begin{equation}\label{PiOps}
    \Pi_1p_b = D_I[(p_b)_{\Gamma_I}],\quad  \Pi_2\vec{v} = D_I\big[(\vec{e}_3 \cdot [2\nu\vec{D}(\vec{v})\vec{e}_3])_{\Gamma_I}\big], \quad \Pi_3\vec{v} = N_f[(\nu\Delta\vec{v} \cdot \vec{e}_3)_{\Gamma_f}].
\end{equation}
Finally, for $i=1,2,3$, we denote $G_i = -\nabla\Pi_i$, as invoked in \eqref{A}---for convenience, we will work directly with the $\Pi_i$ below. With these $\Pi_i$ well-defined, we may proceed to define the (prospective) semigroup generator.

\subsection{Domain and Properties of $\cA$}\label{gendef}
Let us first define some dual trace spaces, associated to the spaces $H^s_{\#}(\Omega_i)$ for $i=b,f$. We first note that the image of the Dirichlet trace map $\gamma_0: H^1(\Omega_i) \to H^{1/2}(\Omega_i)$ is well-defined by the trace theorem, as are subsequent restrictions to any faces of the cubic geometry of $\Omega_b$ or $\Omega_f$ \cite{Grisvard}. To preserve and utilize the associated duality pairing with trace spaces such as $\gamma_0[H^{1}_{\#}(\Omega_i)]$---for instance to characterize higher order traces---we now explicitly define $H_{\#}^{-1/2}(\partial \Omega_f)$. The definition will then permit  such negative traces on any face via restriction (in the sense of distributions), and an analogous definitions will also hold for spaces defined on $\Omega_b$. Now, let $\Gamma_a \subset \Gamma_{\text{lat}}$ and $\Gamma_b \subset \Gamma_{\text{lat}}$ generically represent diametrically opposed faces of $\Omega_b$ or $\Omega_f$---namely $(a,b)=(\{x_1=0\}, \{x_1=1\})$ or $(\{x_2=0\}, \{x_2=1\})$. Then
\begin{align}\label{tracespaces}
H^{-1/2}_{\#}(\partial \Omega_f)=  \Big\{g \in H^{-1/2}(\partial \Omega_f)~:~g \big|_{\Gamma_a} = g\big|_{\Gamma_b} \in H^{-1/2}\big((0,1)\times (-1,0)\big)~\text{ for each pair}~\Gamma_a, \Gamma_b\Big\}.
\end{align}
We obtain the spaces $H^{-r}_{\#}(\partial \Omega_i)$, for $r=1$ and $r=3/2$, mutatis mutandis. 

The differential action of our prospective generator $\mathcal A$ is: \begin{equation}\label{diffaction}
   \mathcal A\mathbf y\equiv  \begin{pmatrix}
        \vec{w}\\
        -\mathcal{E}_0\vec{u} - \alpha \rho_b^{-1}\nabla p_b\\
        -\alpha c_0^{-1}\nabla\cdot \vec{w} -A_0p_b\\
        \rho_f^{-1}\nu \Delta \vec{v} - \rho_f^{-1}\nabla\pi(p_b,\bv)
    \end{pmatrix}
    =
    \begin{pmatrix}\bw_1^*\\\bw_2^*\\q^*\\\mathbf f^*\end{pmatrix} \in X, \quad \vec{y} = [\vec{u},\vec{w},p_b,\vec{v}]^T \in \mathcal{D}(\mathcal{A}),
\end{equation}
where we denote $\pi = \pi(p_b,\vec{v})\equiv \Pi_1p_b+\Pi_2 \mathbf v+\Pi_3\mathbf v \in L^2(\Omega_f)$ as a solution to the following elliptic problem, for given $p_b$ and $\vec{v}$ emanating from $\vec{y} \in \mathcal{D}(\mathcal{A})$ (i.e., having adequate regularity to form the RHS below):
\begin{equation}\label{pi}
    \begin{cases}
        \Delta \pi = 0 &\in L^2(\Omega_f),\\
        \partial_{\vec{e}_3}\pi = \nu \Delta \vec{v} \cdot\vec{e}_3& \in H^{-3/2}(\Gamma_f),\\
        \pi = p_b + 2\nu\vec{e}_3\cdot \vec{D}(\vec{v})\vec{e}_3  &\in H^{-1/2}(\Gamma_I),\\
        \pi  \in H^{-1/2}_{\#}(\partial \Omega_f).
    \end{cases}
\end{equation}

From the membership $\mathcal A\mathbf y \in X$, we  then read off:  $$\vec{w} = \bw_1^\ast \in \vec{U} \subset \vec{H}_\#^1(\Omega_b).$$
The above action then generates the  {\it elliptic}  system:

$$\begin{cases}
        -\mathcal{E}_0\vec{u} - \alpha\nabla p_b = \bw_2^\ast & \in \vec{L}^2(\Omega_b),\\
        -A_0p_b = q^\ast + \alpha\nabla\cdot{\bw_1^*} &\in L^2(\Omega_b),\\
        \rho_f^{-1}[\nu \Delta \vec{v} - \nabla\pi] = \mathbf f^\ast &\in \vec{V},\\
        \nabla\cdot\vec{v} = 0,\\
     p_b\big|_{\Gamma_b} = 0,\quad \vec{v}|_{\Gamma_f} = 0, \quad \vec{u}|_{\Gamma_b} = 0\\
        \vec{v}\cdot\vec{e}_3 +  k\partial_{\vec{e}_3} p_b = \vec{w}\cdot\vec{e}_3 &\in H^{-1/2}(\Gamma_I),\\
        \beta\vec{v}\cdot\btau +  \btau \cdot [2\nu\vec{D}(\vec{v})-\pi]\vec{e}_3 =  \beta \vec{w}\cdot\btau&\in H^{-1/2}(\Gamma_I),\\
        \sigma_b\vec{e}_3 = [2\nu\mathbf D(\bv)-\pi]\mathbf e_3 & \in \vec{H}^{-1/2}(\Gamma_I),\\
        p_b = - 2\nu\vec{e}_3\cdot\mathbf D(\bv)\mathbf e_3+\pi &\in H^{-1/2}(\Gamma_I).
    \end{cases} $$
Again, $\pi=\pi(p_b,\bv)$, as given in \eqref{pi}.

 This suggests the proper definition of a domain for $\mathcal A$ below, which is one of the main contributions of this treatment. All restrictions below are interpreted as traces, in the generalized sense, where appropriate.
\begin{definition}[Domain of $\cA$] \label{diffdomain} Let $\vec{y}\in X$. Then $\vec{y} = [\vec{u},\vec{w},p,\vec{v}]^T \in \mathcal{D}(\mathcal{A})$ if and only if the following bullets hold:
\begin{itemize}
    \item $\vec{u} \in \vec{U}$  with  $\mathcal{E}_0(\vec{u}) \in \vec{L}^2(\Omega_b)$ (so that $[{\sigma_b(\vec{u})\mathbf n}]\big|_{\partial\Omega_b} \in \vec{H}^{-1/2}(\partial \Omega_b)$);
    \item $\vec{w} \in \vec{U}$;
    
    \item $p \in  H_\#^1(\Omega_b)$ with $A_0p \in L^2(\Omega_b)$ (so that $\left.\partial_{\mathbf n}p\right|_{\partial\Omega_b} \in {H}^{-1/2}(\partial \Omega_b)$);
    
        \item $[{\sigma_b(\vec{u})\mathbf n}] \in \vec{H}_{\#}^{-1/2}(\partial \Omega_b)$ (then  
       $\sigma^E(\bu)\mathbf n\cdot \mathbf n \in H^{-1/2}_{\#}(\partial \Omega_b)$);
       \item  $\partial_{\mathbf n} p \in H^{-1/2}_{\#}(\partial \Omega_b)$
    
    \item $\vec{v} \in \vec{H}_\#^1(\Omega_f) \cap \vec{V}$ with $\vec{v}|_{\Gamma_f} = \vec{0}$;
    \item There exists $\pi \in L^2(\Omega_f)$ such that $$\nu\Delta \vec{v} - \nabla\pi \in \vec{V},$$ where $\pi=\pi(p,\bv)$ is as in \eqref{pi} (and so
 $\left.\sigma_f\mathbf n\right|_{\partial \Omega_f}\in \vec{H}^{-\frac{1}{2}}(\partial\Omega_f)$ and  $\left.\pi\right|_{\partial \Omega_f} \in H^{-\frac{1}{2}}(\partial\Omega_f)$;
   
    \item $2\nu \mathbf D(\bv) \big|_{\partial \Omega_f} \in \mathbf H^{-1/2}_{\#}(\partial \Omega_f)$ and ~$\pi \big|_{\partial\Omega_f} \in H^{-1/2}_{\#}(\partial \Omega_f)$;
    
    \item  $\left.\Delta \vec{v}\cdot\mathbf n\right|_{\partial \Omega_f} \in H^{-3/2}(\partial\Omega_f)$);
    
    \item  $\restri{(\vec{v} - \vec{w})\cdot\vec{e}_3} = \restri{-k\nabla p\cdot\vec{e}_3} \in H^{-1/2}(\Gamma_I)$;
    
    \item $\restri{\beta(\vec{v}-\vec{w})\cdot\btau} = \restri{-\btau\cdot\sigma_f\vec{e}_3} \in H^{-1/2}(\Gamma_I)$;
    
    \item $\restri{-\vec{e}_3 \cdot \sigma_f\vec{e}_3} = \restri{p} \in H^{-1/2}(\Gamma_I)$;
    
    
    \item $\restri{\sigma_b\vec{e}_3} = \restri{\sigma_f\vec{e}_3} \in \vec{H}^{-1/2}(\Gamma_I)$.

\end{itemize}
\end{definition}
\noindent The above terms with  traces in  Sobolev spaces of negative indices are well-defined (as extensions) when $\vec{y} \in \mathcal{D}(\cA)$.

\begin{remark}\label{regularity*}
For any $\vec{y} \in \mathcal{D}(\cA)$, we can identify additional regularity through the trace theorem,  derivative-counting, and the use of elliptic regularity for our  toroidal domain:
\begin{itemize}
    \item  $\restri{-\btau\cdot\sigma_f\vec{e}_3}=\left.\beta(\vec{v}-\vec{w})\cdot\btau\right|_{\partial \Omega_f} \in H^{1/2}(\Gamma_I)$;~~the trace theorem applied to $\bv$ and $\bw$,
    
    \item $\restri{\partial_{\vec{e}_3}p}= \restri{(\vec{v} - \vec{w})\cdot\vec{e}_3}  \in H^{1/2}(\Gamma_I)$; ~~the trace theorem applied to $\bv$ and $\bw$,
    
    \item $p \in H_\#^2(\Omega_b)$; ~~invoking elliptic regularity for $p$ with $H^{1/2}$ Neumann trace from above,
    
    \item $\restri{-\vec{e}_3 \cdot \sigma_f \vec{e}_3} = \restri{p} \in H^{3/2}(\Gamma_I)$; ~~from the pressure coupling condition and above
    
    \item adding normal and tangential parts for the fluid normal stress, we have $$\restri{\sigma_f\vec{e}_3} = \restri{\sigma_b\vec{e}_3} \in \vec{H}^{1/2}(\Gamma_I).$$
\end{itemize}
We note that although there seems to be generous smoothing in several variables from domain membership, there is no smoothing in $\bv$ and $\pi$ through the auxiliary pressure equation \eqref{pi}.  This is a feature of the hyperbolic-parabolic coupling via low-regularity interface conditions.
\end{remark}

\section{Semigroup Generation: Theorem \ref{th:main1}}
\begin{proof}[Proof of Theorem \ref{th:main1}] We prove dissipativity directly using the definition of the $\mathcal D(\cA)$ above, and follow it by showing the range condition for $\cA$ on $X$. These two facts, in conjunction, allow the  application of the Lumer-Philips theorem \cite{pazy} to obtain semigroup generation, as stated in Theorem \ref{th:main1}; subsequently, the results in Corollary \ref{coro} and Theorem \ref{th:main2} are obtained directly.

\subsection{Dissipativity} \label{dissip} Let $\vec{y} = [\vec{u},\vec{w},p,\vec{v}]^T \in \mathcal{D}(\cA)$. Recall that $X\equiv  \vec{U} \times \vec{L}^2(\Omega_b) \times L^2(\Omega_b) \times \vec{V}$ (with appropriately defined inner-product \eqref{innerX}), and note that
\begin{align*}
    (\mathcal{A}\vec{y},\vec{y})_X &=  \left(\begin{pmatrix}
        \vec{w}\\[.1cm]
        -\mathcal{E}_0\vec{u} -\alpha\rho_b^{-1} \nabla p\\[.2cm]
        -A_0p -\alpha c_0^{-1}\nabla\cdot \vec{w}\\[.1cm]
       \rho_f^{-1}[\nu \Delta \vec{v} - \nabla\pi]
    \end{pmatrix},  \begin{pmatrix}\vec{u}\\\vec{w}\\p\\\vec{v}\end{pmatrix}\right)_X\\
    =&~ a_E(\vec{w},\vec{u}) + (\nabla\cdot \sigma^E\vec{u},\vec{w})_{0,\Omega_b} - (\alpha\nabla p,\vec{w})_{0,\Omega_b} + (k\Delta p,p)_{0,\Omega_b} \\
    &~- (\alpha\nabla\cdot\vec{w},p)_{0,\Omega_b} + (\nu\Delta\vec{v},\vec{v})_{0,\Omega_f} - (\nabla\pi,\vec{v})_{0,\Omega_f}\\
    =&~ -k\norm{\nabla p}_{0,\Omega_b}^2 - 2\nu\norm{\vec{D}(\vec{v})}_{0,\Omega_f}^2 + I_{\text{bdry}},
\end{align*}
where the conditions built into $\mathcal{D}(\cA)$ yield the sum of relevant terms for the normal and tangential coupling conditions, and
\begin{align}\label{Ibdry}
    I_{\text{bdry}} &= -\int_{\Gamma_I} \sigma_b(\vec{u},p)\vec{e}_3 \cdot \vec{w}  - \int_{\Gamma_I}\left(k\partial_{\vec{e}_3}p\right)p + \int_{\Gamma_I} \sigma_f(\vec{v},\pi)\vec{e}_3 \cdot \vec{v} =  -\beta\int_{\Gamma_I}[(\vec{v} - \vec{w})\cdot\btau]^2.
\end{align} Note that, above, we have utilized: Green's Theorem in the $\cE_0\bu$ term, justified by extension in this context;  the Divergence Theorem in  terms involving $\alpha$; and the divergence-free condition on $\bv\in \bV$. Of course, we have also invoked the periodic boundary conditions on $\Gamma_{\text{lat}}$ and homogeneous conditions on $\Gamma_f$ and $\Gamma_b$, away from the interface $\Gamma_I$

Hence, we obtain that ~$(\cA \vec{y},\vec{y})_X \leq 0$ for all $\vec{y} \in \mathcal{D}(\cA)$, and thus $\mathcal A$ is {\em dissipative} on $X$ using the domain $\mathcal D(\mathcal A).$

\subsection{Maximality}\label{max1}
To show that $\cA$  is indeed {\it m}-dissipative, we consider, for a given $\mathcal F = [\vec{f}_1,\vec{f}_2,f_3,\vec{f}_4]^T \in X$, the resolvent system $(\varepsilon\vec{I} - \cA)\vec{y} = \mathcal{F}$ for $\vec{y} = [\vec{u},\vec{w},p,\vec{v}]^T \in \mathcal{D}(\cA)$ taken with $\varepsilon=1$, i.e., given $\mathcal F \in X$ we must find $\mathbf y \in \mathcal D(\cA)$ so that
\begin{equation}\label{resolvent1}
\begin{cases}
    \vec{u} - \vec{w} = \vec{f}_1 \in \vec{U},\\
    \vec{w} + \cE_0(\vec{u}) + \rho_b^{-1}\alpha \nabla p = \vec{f}_2 \in \vec{L}^2(\Omega_b),\\
    p + c_0^{-1}\alpha\nabla\cdot\vec{w} + A_0p = f_3 \in L^2(\Omega_b),\\
    \vec{v} - \rho_f^{-1}\nu \Delta \vec{v} + \rho_f^{-1}\nabla \pi = \vec{f}_4 \in \vec{V}
\end{cases}
\end{equation}
is solved {\em strongly} (i.e., each line holds in the listed function space).

From the first equation, we set $\vec{w} = \vec{u} - \vec{f}_1$, so that we collapse the system to
\begin{equation}\label{resolvent2}
\begin{cases}
    \vec{u} + \cE_0(\vec{u}) + \rho_b^{-1}\alpha \nabla p = \vec{f}_2 +\vec{f}_1 \in L^2(\Omega_b),\\
    p + c_0^{-1} \alpha\nabla\cdot\vec{u} + A_0p = f_3+\alpha c_0^{-1}\nabla \cdot \vec{f}_1 \in L^2(\Omega_b),\\
    \vec{v} - \rho_f^{-1}\nu \Delta \vec{v} + \rho_f^{-1}\nabla \pi = \vec{f}_4 \in \vec{V}.
\end{cases}
\end{equation}
We will convert the above system in \eqref{resolvent2} to a mixed-variational weak formulation, then proceed to solve it and finally recover domain membership.
\subsubsection{Weak Solution to Resolvent System}
We will use a version of Babu\v{s}ka-Brezzi (Theorem \ref{bbt}) \cite{kesavan,avalos*} to solve the mixed-variational form of the above elliptic system.
Recall that $\vec{H}^1_{\#,*}(\Omega_f) = \{\vec{v} \in \vec{H}_\#^1(\Omega_f)~:~\vec{v}|_{\Gamma_f} = 0\}$ (and similarly for $H^1_{\#,*}(\Omega_b))$ and define the auxiliary product space
\begin{align}\label{sigmaspace}
    \Sigma = \vec{U}\times  H_{\#,*}^1(\Omega_b)\times \vec{H}^1_{\#,*}(\Omega_f),
\end{align}
where we have endowed the spaces $\vec{U}$ with the norm $\|\cdot\|_E$ and $\vec{H}^1_{\#,*}(\Omega_f)$  with the norm \begin{equation} \|\cdot\|_{f,\ast} \equiv \| \mathbf D (\cdot)\|_{0,\Omega_f}.\end{equation}  Note that  $\|\cdot\|_{f,\ast}$ is equivalent to the standard $\vec{H}^1(\Omega_f)$ norm on $\vec{H}^1_{\#,*}(\Omega_f)$, by Korn's and Poincar\'e's inequalities \cite{kesavan}. From the LHS of \eqref{resolvent2}, we generate a bilinear form for $\phi = [\vec{u},p,\vec{v}]^T \in \Sigma$ defined below as $a(\cdot,\cdot)$, as well as a continuous bilinear form
\begin{equation}\label{bBilinear}
    b(\psi,\pi) = -\int_{\Omega_f}\pi\nabla\cdot\bzeta\,d\vec{x} \quad \forall\, \psi \in \Sigma, ~\pi \in L^2(\Omega_f),~\forall~\psi=[\bxi, q, \bzeta]^T.
\end{equation}

\begin{align}\label{aBilinear}
    a(\phi,\psi) =&~ \rho_b\int_{\Omega_b}\vec{u}\cdot\bxi\,d\vec{x} + a_E(\bu,\bxi) + c_0\int_{\Omega_b} pq\,d\vec{x} + k\int_{\Omega_b}\nabla p \cdot \nabla q\,d\vec{x} \nn\\
    &+~\rho_f \int_{\Omega_f} \vec{v}\cdot\bzeta\,d\vec{x} + 2\nu\int_{\Omega_f} \vec{D}(\vec{v}) : \vec{D}(\bzeta)\,d\vec{x} \nn\\
    &-\alpha \int_{\Omega_b} p (\nabla\cdot \bxi) \, d\mathbf x +\alpha \int_{\Omega_b} (\nabla \cdot \bu)q \, d\mathbf x \nn\\
    &+ \int_{\Gamma_I}p\left((\bzeta - \bxi)\cdot\vec{e}_3\right)\,d\Gamma_I - \int_{\Gamma_I} \left((\vec{v}-\vec{u})\cdot\vec{e}_3\right)q\,d\Gamma_I  \nn\\
    &+ \beta \int_{\Gamma_I}\big((\vec{v} - \vec{u})\cdot\btau\big) \big((\bzeta - \bxi)\cdot\btau\big)\,d\Gamma_I \nn\\
    \forall\, \psi =~ [\bxi,q,\bzeta]^T & \in \Sigma.
\end{align}
We want to find unique solutions $\phi = [\vec{u},p,\vec{v}]^T \in \Sigma$ and $\pi \in L^2(\Omega_f)$ to the system
\begin{equation}\label{BBprob}
\begin{cases}
    a(\phi,\psi) + b(\psi,\pi) = \rho_b (\vec{f}_1+\vec{f}_2,\bxi)_{0,\Omega_b}  +  (c_0f_3+\alpha \nabla \cdot \vec{f}_1,q)_{0,\Omega_b} + \rho_f(\vec{f}_4,\bzeta)_{0,\Omega_f} \\
    \phantom{a(\phi,\psi) + b(\psi,\pi) =} + (\vec{f}_1\cdot\vec{e}_3,q)_{\Gamma_I} - (\vec{f}_1\cdot\btau,\beta(\bzeta-\bxi)\cdot\btau)_{\Gamma_I}\quad &\forall\, \psi = [\bxi,q,\bzeta]^T \in \Sigma, \\
    b(\phi,\eta) = 0 \quad &\forall\, \eta \in L^2(\Omega_f).
    \end{cases}
\end{equation}
We can then define $\Phi \in \Sigma'$ as \begin{align} \Phi (\psi) \equiv&~ \rho_b(\vec{f}_1+\vec{f}_2,\bxi)_{0,\Omega_b}  + c_0(f_3+\alpha \nabla \cdot \vec{f}_1,q)_{0,\Omega_b} + \rho_f(\vec{f}_4,\bzeta)_{0,\Omega_f},\nonumber\\
&+ (\vec{f}_1\cdot\vec{e}_3,q)_{\Gamma_I} - (\vec{f}_1\cdot\btau,\beta(\bzeta-\bxi)\cdot\btau)_{\Gamma_I}~~\forall \psi \in \Sigma. \end{align}
The regularity of the data $\mathcal F \in X$ is sufficient to yield directly that $\Phi \in \Sigma'$ by applying Cauchy-Schwarz and the trace theorem. 

Finally, we define one more auxiliary space:
\[
    Z \equiv \{\psi =~ [\bxi,q,\bzeta]^T \in \Sigma ~:~ \text{div}~\bzeta = 0\}.
\]
We note that $Z$ is identified as the subspace $\Sigma$ on which $b(\cdot,\eta)=0$ for every $\eta \in L^2(\Omega_f)$.
Given this space $Z$, for any $\phi = [\vec{u},p,\vec{v}]^T \in Z$, we have that $\text{div}~\vec{v} = 0$, and 
\begin{align*}
    a(\phi,\phi) =&~ \left(\rho_b\norm{\vec{u}}_{0,\Omega_b}^2  + \int_{\Omega_b}\sigma^E(\vec{u}) : \vec{D}(\vec{u})\,d\vec{x}\right) + c_0\norm{p}_{0,\Omega_b}^2 + k\norm{\nabla p}_{0,\Omega_b}^2\\
    &+~ \left(\rho_f\norm{\vec{v}}_{0,\Omega_f}^2 + 2\nu\norm{\vec{D}(\vec{v})}_{0,\Omega_f}^2\right) + \int_{\Gamma_I}\big((\vec{v} - \vec{w})\cdot\btau\big)^2\,d\Gamma_I\\
    \geq&~C(k,\nu,\rho_b,\rho_f,\lambda,\mu,c_0)\norm{\phi}_\Sigma^2 \quad\quad \forall\, \phi \in Z.
\end{align*}
That is, {\em we have shown that $a(\cdot,\cdot)$ is $Z$-elliptic.} 

Now, by the continuity of the trace mapping and the Sobolev embedding,  $\vec{H}_\#^1(\Omega_i) \xrightarrow{\gamma_0} \vec{H}^{1/2}(\partial \Omega_i) \hookrightarrow \vec{L}^2(\partial \Omega_i)$ for $i=b,f$, there exist constants $C_{tb},C_{tf}>0$ (of course, each depending on the geometry of $\Omega_i$) such that
\[
    \norm{\boldsymbol \theta}_{0,\partial \Omega_i} \leq C_{ti}\norm{\boldsymbol \theta}_{1,\Omega_i} \quad \forall\, {\boldsymbol \theta} \in \vec{H}^1_\#(\Omega_i), i=b,f.\] In what follows, we do not explicitly denote the dependence of constants on the domains $\Omega_i$ through the Poincar\'e and Korn constants.
Then for any $\phi = [\vec{u},p,\vec{v}]^T, ~~\psi = [\bxi,q,\bzeta]^T \in \Sigma$,
\begin{align}\label{bdryEst}
    \int_{\Gamma_I} \big((\vec{v} - \vec{u})\cdot\btau\big) \big((\bzeta - \bxi)\cdot\btau\big)\,d\Gamma_I 
    &\leq \left(\norm{\vec{v}\cdot\btau}_{0,\Gamma_I} + \norm{\vec{u}\cdot\btau}_{0,\Gamma_I}\right)\left(\norm{\bzeta\cdot\btau}_{0,\Gamma_I} + \norm{\bxi\cdot\btau}_{0,\Gamma_I}\right) \nn\\
    &\leq \left(C_{tf}\norm{\vec{v}}_{1,\Omega_f} + C_{tb}\norm{\vec{u}}_{1,\Omega_b}\right)\left(C_{tf}\norm{\bzeta}_{1,\Omega_f} + C_{tb}\norm{\bxi}_{1,\Omega_b}\right) .
\end{align}
Additionally, we have, in the standard fashion, obtained from Poincar\'e's and Korn's inequalities the existence of constants $C_{P_b}, C_{P_f}>0$ such that, 
\begin{align}\label{uvPoincare}
    \|\vec{u}\|_{0,\Omega_b} & \leq C_{P_b}\|\vec{u}\|_E ~~\quad \forall\, \vec{u} \in \vec{U},\nn\\
    \|\vec{v}\|_{0,\Omega_f} &\leq  C_{P_f}\|\vec{v}\|_{f,\ast} \quad \forall\, \vec{v} \in \vec{H}^1_{\#,\ast}(\Omega_f).
\end{align}
Recalling the equivalence of the $\vec{H}^1(\Omega_b)$ norm and $\|\cdot\|_E$ on $\vec{U}$, $H^1(\Omega_b)$ and $||\nabla \cdot ||_{0,b}$ on $H^1_{\#,*}(\Omega_b)$, as well as the $\vec{H}^1(\Omega_f)$ norm and $\|\cdot\|_{f,\ast}$ on $\vec{H}^1_{\#,\ast}(\Omega_f),$
we obtain that for $\phi = [\vec{u},p,\vec{v}]^T$, $\psi = [\bxi,q,\bzeta]^T \in \Sigma$,
\begin{align*}
    a(\phi,\psi) \leq&~ \rho_b\norm{\vec{u}}_{0,\Omega_b}\norm{\bxi}_{0,\Omega_b} + \norm{\sigma^E(\vec{u})}_{0,\Omega_b}\norm{\vec{D}(\bxi)}_{0,\Omega_b} + \norm{p}_{0,\Omega_b}\norm{q}_{0,\Omega_b} + k\norm{\nabla p}_{0,\Omega_b}\norm{\nabla q}_{0,\Omega_b}\\
    &+~ \norm{\vec{v}}_{0,\Omega_f} \norm{\bzeta}_{0,\Omega_f} + 2\nu\norm{\vec{D}(\vec{v})}_{0,\Omega_f} \norm{\vec{D}(\bzeta)}_{0,\Omega_f} + \int_{\Gamma_I} \big((\vec{v} - \vec{u})\cdot\btau\big) \big((\bzeta - \bxi)\cdot\btau\big)\,d\Gamma_I\\
   \leq&~ C(k,\nu,\rho_b,\rho_f,\lambda,\mu,c_0)\norm{\phi}_\Sigma\norm{\psi}_\Sigma.
\end{align*}
{\em This yields that the bilinear form $a(\cdot,\cdot)$ is continuous on $\Sigma$.}

Now, to invoke Babu\v{s}ka-Brezzi (stated below as Theorem \ref{bbt}), we must establish the {\em inf-sup condition} in our situation. We adapt the boundary value problem used in \cite{avalos*} to obtain the inf-sup property in this setting. Namely, let $\eta \in L^2(\Omega_f)$. Moreover, let $\mu(x)$ be a nonzero, smooth (spatial) cutoff function. Therewith, we consider the following boundary value problem: \\ Find $\boldsymbol\omega \in \vec{H}^1_{\#,\ast}(\Omega_f)$ such that
\begin{equation}\label{infsupBVP}
\begin{cases}
    \text{div}~\boldsymbol\omega = -\eta &\text{ in } \Omega_f,\\[.35cm]
    \boldsymbol\omega|_{\partial \Omega_f} = \begin{cases} 0 &\text{ on } \Gamma_f\cap \Gamma_{\text{lat}}, \\[.2cm]  \displaystyle{\dfrac{-\int_{\Omega_f}\eta\, d\Omega_f}{\int_{\Gamma_I}\mu ~d\Gamma_I}\mu(\mathbf x)\vec{e}_3}  &\text{ on } \Gamma_I. \end{cases} 
\end{cases}
\end{equation}
Since the given boundary data is certainly in $H^{1/2}(\partial\Omega_f)$, and moreover vanishes off $\Gamma_I$,  we have a solution $\boldsymbol \omega \in H^1_{\#,*}(\Omega_f)$ to \eqref{infsupBVP}, see, e.g., \cite[Ch. III.3]{galdi}.  We have the supporting the estimate,
\[
    \|\boldsymbol\omega\|_{1,\Omega_f} \leq c \|\eta\|_{0,\Omega_f},
\]
for some $c>0$. 
Then, for the given $\eta \in L^2(\Omega_f)$,
\begin{equation}\label{infsupv}
    \sup_{\substack{\vec{v} \in \vec{H}^1_{\#,*}(\Omega_f)\\\vec{v} \neq \vec{0} }}\frac{-\int_{\Omega_f} \eta\text{div}(\vec{v})d\vec{x}}{\|\mathbf D(\vec{v})\|_{0,\Omega_f}} \geq \frac{-\int_{\Omega_f}\eta\text{div}(\boldsymbol\omega)\,d\vec{x}}{\|\mathbf D(\boldsymbol\omega)\|_{0,\Omega_f}} \geq \frac{\|\eta\|_{0,\Omega_f}^2}{\|\mathbf D(\boldsymbol\omega)\|_{0,\Omega_f}} \geq c\norm{\eta}_{0,\Omega_f}.
\end{equation}
We can reconcile the dominating term in \eqref{infsupv} with the bilinear form $b(\cdot,\cdot):\Sigma\times L^2(\Omega_f) \to \mathbb{R}$, defined in \eqref{bBilinear}, by the relation  
\begin{equation}\label{infsupbv}
    \sup_{\substack{\phi \in \Sigma, \phi \neq \vec{0}\\\vec{u} = \vec{0}, p = 0}} \frac{b(\phi,\pi)}{\|\phi\|_\Sigma} = \sup_{\substack{\vec{v} \in \vec{H}^1_{\#,*}(\Omega_f)\\\vec{v} \neq \vec{0} }}\frac{-\int_{\Omega_f} \pi\text{div}(\vec{v})d\vec{x}}{\|\mathbf D(\vec{v})\|_{0,\Omega_f}}.
\end{equation}
Noting the subset of $\Sigma$ under the supremum, it is clear that, for the given $\eta \in L^2(\Omega_f)$,
\begin{equation*}
    \sup_{\substack{\phi \in \Sigma \\ \phi \neq \vec{0}}} \frac{b(\phi,\pi)}{\|\phi\|_\Sigma} \geq \sup_{\substack{\phi \in \Sigma, \phi \neq \vec{0}\\\vec{u} = \vec{0}, p = 0}} \frac{b(\phi,\pi)}{\|\phi\|_\Sigma}.
\end{equation*}
On the other hand, for a given $\phi \in [\bu,p,\bv]^T$, the term $\frac{b(\phi,\pi)}{\|\phi\|_\Sigma}$ is bounded above by $\frac{b(\phi,\pi)}{\|\mathbf D(\vec{v})\|_{\Omega_f}}$ for every $\phi \in \Sigma$, since the presence of any $\vec{u} \in \vec{U}$ or $p \in H^1_\#(\Omega_b)$ appears only as additional non-negative terms in $\|\phi\|_\Sigma$. Hence, 
\begin{equation}\label{infsup}
    \sup_{\substack{\phi \in \Sigma \\ \phi \neq \vec{0}}} \frac{b(\phi,\pi)}{\|\phi\|_\Sigma} = \sup_{\substack{\phi \in \Sigma, \phi \neq \vec{0}\\\vec{u} = \vec{0}, p = 0}} \frac{b(\phi,\pi)}{\|\phi\|_\Sigma} \geq c\|\eta\|_{0,\Omega_f} \quad \forall\, \eta \in L^2(\Omega_f),
\end{equation}
where the the latter inequality follows from \eqref{infsupv} and \eqref{infsupbv}. Thus, the inf-sup condition is verified, and Theorem \ref{bbt} can be invoked to guarantee a unique pair of solutions $\phi = [\vec{u},p,\vec{v}]^T \in \Sigma$, $\pi \in L^2(\Omega_f)$ to the problem \eqref{BBprob}. 

\subsubsection{Recovering the Weak Solution in the Domain}
At this point, we proceed to verify that the weak solution that we have obtained,  $\phi = [\vec{u},p,\vec{v}]^T \in \Sigma$, yields a proper element of $\mathcal D(\mathcal A)$. Namely, we want to solve \eqref{resolvent2} with the weak solution we have just obtained, and thereby recover the unique solution $[\vec{u},\vec{w},p,\vec{v}]^T \in \mathcal{D}(\cA)$ to the resolvent equation \eqref{resolvent1}. 

Recall that $\mathbf f_4 \in \mathbf V \subset \mathbf L^2(\Omega_f)$. Since $\phi \in \Sigma$ and $\pi \in L^2(\Omega_f)$ satisfy the first equation of \eqref{BBprob} for all $\psi = [\bxi,q,\bzeta]^T \in \Sigma$, we may choose $\bxi = \vec{0}$ and $q = 0$ to obtain for the pair of solutions $[\vec{v},\pi]^T \in \vec{H}^1_{\#,\ast}(\Omega_f) \times L^2(\Omega_f)$ that
\begin{align*}
   \rho_f (\vec{v},\bzeta)_{0,\Omega_f} + (2\nu\vec{D}(\vec{v}),\vec{D}(\bzeta))_{0,\Omega_f} - (\pi,\nabla\cdot\bzeta)_{0,\Omega_f} +&\int_{\Gamma_I}p\bzeta \cdot \mathbf e_3d\Gamma_I+\beta\int_{\Gamma_I}[(\bv-\bu)\cdot \btau][\bzeta \cdot \btau]d\Gamma_I\\
    =&~ \rho_f(\vec{f}_4,\bzeta)_{0,\Omega_f}, \quad \quad \forall\, \bzeta \in \vec{H}^1_{\#,\ast}(\Omega_f).
\end{align*}
Thus, by restricting to $\bzeta \in [C_0^\infty(\Omega_f)]^3 \subset \mathbf H^1_{\#,*}(\Omega_f)$, we observe that
$$\rho_f\vec{v} - \nu \Delta \vec{v} + \nabla \pi = \rho_f\mathbf f_4$$ in the sense of distributions. 
By the density of $[C_0^\infty(\Omega_f)]^3$ in $\vec{L}^2(\Omega_f)$, we then have the $a.e.$~ $\Omega_f$ equality
\begin{equation}\label{resolventd}
 \rho_f   \vec{v} - \nu \Delta \vec{v} + \nabla \pi = \rho_f\vec{f}_4 \quad \text{ in } \Omega_f.
\end{equation}
In a similar manner, we may choose test functions $[\bxi,0,\vec{0}]^T,~[\vec{0},q,\vec{0}]^T \in \Sigma$ to have that $\{[\bu,p,\bv]^T,\pi\} \in \Sigma \times L^2(\Omega_f)$ solves $a.e.$ the interior equations in \eqref{resolvent2}. Finally, with $\vec{u} \in \vec{U}$ and $\vec{f}_1 \in \vec{U}$, we take $\vec{w} = \vec{u} - \vec{f}_1 \in \vec{U}$, and rearrange to obtain
\begin{equation}\label{resolventa}
    \vec{u} - \vec{w} = \vec{f}_1 \in \vec{U},
\end{equation}
and then we may plug in $\vec{w}$ for the quantities $\vec{u}-\vec{f}_1$ that appear in \eqref{resolvent2} to yield the full resolvent system \eqref{resolvent1}.

We now proceed to demonstrate the additional requirements for the weak solution $[\vec{u},\vec{w},p,\vec{v}]^T$ of \eqref{resolvent1} to have domain membership. Since $\vec{v} \in \vec{H}^1_{\#,\ast}(\Omega_f)$, we read off from the second equation in \eqref{BBprob} that $(\eta,\nabla\cdot \bv)_{L^2(\Omega_f)}=0$ for all $\eta \in L^2(\Omega_f)$, and thus $\text{div}~\vec{v} = 0$ almost everywhere in $\Omega_f$; this and Green's identity give $\vec{v} \in \vec{H}^1_{\#,\ast}(\Omega_f) \cap \vec{V}$. Then since $[\vec{v},\pi]^T \in (\vec{H}^1_{\#,\ast}(\Omega_f) \cap \vec{V}) \times L^2(\Omega_f)$ satisfies \eqref{resolventd} $a.e.$, with $\vec{f}_4 \in \vec{V}$, we observe that 
\begin{equation}\label{vpiRegularity}
    \nu\Delta\vec{v} - \nabla \pi = \rho_f[\vec{v}-\vec{f}_4] \in \vec{V}.
\end{equation}

Utilizing \cite[Lemma 2]{avalos*} (see also \cite{temam}), and an integration by parts, we have that the pair $[\bv,\pi]^T$ has the additional regularity
\begin{equation}\label{vResReg}
\begin{split}
    \left.\pi\right|_{\partial \Omega_f} \in H^{-1/2}(\partial\Omega_f), \quad\quad  &\left.\frac{\partial \pi}{\partial\mathbf n}\right|_{\partial \Omega_f} \in H^{-3/2}(\partial\Omega_f),\\[.2cm]
    \nabla \pi \in [\mathbf H^1(\Omega_f)]',\quad\quad  & 2\nu\mathbf D(\bv)\mathbf n \in \mathbf H^{-1/2}(\partial \Omega_f),\\[.2cm]
   (\Delta \bv)\cdot \mathbf n \in H^{-3/2}(\partial \Omega_f), \quad\quad &\Delta\vec{v}  \in [H^{1}(\Omega_f)]'.
\end{split}
\end{equation}
Moreover, we can then read-off that 
\begin{equation}\label{vv} \cE_0(\bu)+\rho_b^{-1}\alpha \nabla p = -\bu+\mathbf f_2+\mathbf f_1 \in \mathbf L^2(\Omega_b) \end{equation} from the second equation in \eqref{resolvent1} and that 
\begin{equation}\label{vvv} A_0p = -p+\mathbf f_3 -c_0^{-1}\alpha \nabla \cdot \bu+{\alpha}c_0^{-1}\nabla \cdot \mathbf f_1 \in L^2(\Omega_b)\end{equation} from the third, since   $\vec{w} \in H^1(\Omega_b)$ and $p \in H^1_\#(\Omega_b)$, resp. In turn, we obtain 
$$\sigma_b(\bu)\mathbf n \in \mathbf H^{-1/2}(\partial \Omega_b),~~\partial_{\mathbf n} p \in H^{-1/2}(\partial \Omega_b).$$
Moreover, we have that 
\begin{equation}
\left[\partial_{\mathbf n}\bv-\pi\mathbf n\right] \in H^{-1/2}_{\#}(\partial \Omega_f).
\end{equation}
That is: If $\Gamma_a,~\Gamma_b \subset \Gamma_{\text{lat}}$ are opposing lateral sides of the $\Omega_f$ geometry, then we obtain immediately that
\begin{equation} 
\left[\partial_{\mathbf n}\bv-\pi\mathbf n\right]\Big|_{\Gamma_a} = \left[\partial_{\mathbf n}\bv-\pi\mathbf n\right]\Big|_{\Gamma_b}
\end{equation}
as elements of $H^{-1/2}\big((0,1)\times (-1,0)\big)$. 
Indeed, WLOG, say $$\Gamma_a \equiv \{0\} \times (0,1) \times (-1,0)~\text{ and }~ \Gamma_b \equiv \{1\} \times (0,1) \times (-1,0).$$ Then by the Sobolev Trace Theorem for Lipschitz domains (of course satisfied by the cubic geometry $\Omega_f$)---see, e.g., \cite[Thm. 3.3.8]{mclean}---given $\mathbf g \in H_0^{1/2+\epsilon}\big((0,1) \times (-1,0)\big)$, there is a $\boldsymbol\phi \in H^1(\Omega_f)$ so that 
\begin{equation}
\boldsymbol\phi  = \begin{cases} \mathbf g &~ \text{ on }~\Gamma_a \\
\mathbf g &~ \text{ on }~\Gamma_b \\
\mathbf 0 &~ \text{ on }~\partial \Omega_f \setminus (\Gamma_a \cup \Gamma_b),\end{cases}
\end{equation}
after using \eqref{BBprob} with the test function $[\mathbf 0, 0,  \boldsymbol \phi]^T \in \Sigma$ and the equation in \eqref{resolventd}. 

From here we note immediately that $\boldsymbol \phi \in \mathbf H^1_{\#}(\Omega_f)$, and $\boldsymbol \phi$ can be chosen so that $\nabla \cdot \boldsymbol \phi=0$ \cite{galdi}. Therewith, we observe
\begin{align}\nn
0 = &~\langle \partial_{\mathbf e_1}\bv, \mathbf g \rangle_{\Gamma_a}-\langle \partial_{\mathbf e_1}\bv, \mathbf g \rangle_{\Gamma_b}-\langle \pi \mathbf e_1, \mathbf g\rangle_{\Gamma_a}+\langle \pi \mathbf e_1,\mathbf g\rangle_{\Gamma_b} \\
= & ~(\nabla \bv, \nabla \mathbf \phi)_{\Omega_f}+(\Delta \bv, \mathbf \phi)_{\Omega_f}-(\nabla p, \mathbf \phi)_{\Omega_f}.
\end{align}
This establishes that 
\begin{equation}\label{work1} 2\nu \mathbf D(\bv)\mathbf n-\pi \mathbf n \in \mathbf H^{-1/2}_{\#}(\partial \Omega_f).\end{equation}
From \eqref{work1}, we will now establish that $\pi \in H^{-1/2}_{\#}(\partial \Omega_f)$ by demonstrating that $[2\nu\mathbf D(\bv)\mathbf n]\cdot\mathbf n \in H^{-1/2}_{\#}(\partial \Omega_f).$ 
With $\Gamma_a$ and $\Gamma_b$ as before, recalling that $\nabla \cdot \bv = 0$, we have that
$$\mathbf e_1 \cdot \mathbf D(\bv)\mathbf e_1=\partial_{1}v^1 = -\partial_{2}v^2-\partial_3v^3.$$
Then, for $f \in \mathscr D\big((0,1)\times (-1,0)\big)$, we have:
\begin{align*}
\big[(\mathbf D(\bv)\mathbf n)\cdot \mathbf n\big]\big|_{\Gamma_a} (f) = & ~ \int_{-1}^0\int_0^1\partial_1v^1(x_1=0)f dx_2dx_3 \\=& ~-\int_{-1}^0\int_0^1 [\partial_2v^2(x_1=0)+\partial_3v^3(x_1=0)]f dx_2dx_3 \\
=&~\int_{-1}^0\int_0^1 [v^2(x_1=0)f_2+v^3(x_1=0)f_3] dx_2dx_3 
\end{align*}
\begin{align*}
\phantom{\big[(\mathbf D(\bv)\mathbf n)\cdot \mathbf n\big]\big|_{\Gamma_a} (f)} =&~\int_{-1}^0\int_0^1 [v^2(x_1=1)f_2+v^3(x_1=1)f_3] dx_2dx_3 \\
=&~-\int_{-1}^0\int_0^1 [\partial_2v^2(x_1=1)f+\partial_3v^3(x_1=1)f] dx_2dx_3 \\
=&~\big[(\mathbf D(\bv)\mathbf n)\cdot \mathbf n\big]\big|_{\Gamma_b} (f)
\end{align*}
Then, by density, we have that \begin{equation} \label{ww}[2\nu\mathbf D(\bv)\mathbf n]\cdot\mathbf n \in H^{-1/2}_{\#}(\partial \Omega_f),~\text{ which yields from \eqref{work1} that }~ \pi \in H^{-1/2}_{\#}(\partial \Omega_f).\end{equation} 

The containments 
\begin{equation}\label{www}
\sigma_b(\bu)\mathbf n \in \mathbf H^{-1/2}_{\#}(\partial \Omega_b), ~\sigma^E(\bu)\mathbf n \in \mathbf H^{-1/2}_{\#}(\partial \Omega_b), ~\partial_{\mathbf n}p \in H^{-1/2}_{\#}(\partial \Omega_b)
\end{equation} 
follow in the same manner. 

The above regularity, along with $[\bv,\pi]^T \in \mathbf H^1_{\#,*}(\Omega_f) \times L^2(\Omega_f)$, justifies the generalized Green's theorem \cite{lionsmag}, permitting integration by parts for the term $\nu\Delta \bv-\nabla \pi \in \bV$. Similarly, the generalized Green's theorem for $A_0p \in L^2(\Omega_b)$ and $\cE_0\bu \in \mathbf L^2(\Omega_b)$ hold.  Thus we have the following identities for the variational solution $\{[\bu,p,\bv],\pi\} \in \Sigma \times L^2(\Omega_f)$:
\begin{align} \label{IBP}
(-\nu\Delta \bv+\nabla \pi, \tilde \bv)_{\mathbf L^2(\Omega_f)} = & ~2\nu(\mathbf D(\bv),\mathbf D(\tilde \bv))_{\mathbf L^2(\Omega_f)} -(\pi,\nabla \cdot \tilde \bv)_{L^2(\Omega_f)}-\langle \sigma_f(\bu,\pi)\mathbf n_{f}, \tilde \bv\rangle_{\Gamma_I} \nonumber \\
 &~\forall \tilde \bv \in \mathbf H^1_{\#,*}(\Omega_f);\\
(\sigma_b\cE_0\bu+\alpha \nabla p ,\tilde\bu)_{\mathbf L^2(\Omega_b)} = &~ a_E(\bu,\tilde \bu)-\alpha ( p, \nabla \cdot \tilde \bu)_{L^2(\Omega_b)}+\langle\sigma_b(\bu,p)\mathbf e_3, \tilde \bu \rangle_{\Gamma_I} \nonumber\\ 
 &~\forall ~ \tilde \bu \in \mathbf U;  \\
(c_0A_0p,\tilde p)_{L^2(\Omega_b)} = &~(k\nabla p, \nabla \tilde p)_{L^2(\Omega_b)}-\langle k\partial_{\mathbf n_b}p,\tilde p\rangle_{\Gamma_I} \nonumber \\
 &~ \tilde p \in H^1_{\#}(\Omega_b).\label{IBP-p}
\end{align}
Above, we have implicitly used the previously obtained \eqref{ww} and \eqref{www}.

The homogeneous boundary conditions for $\bv \in \mathbf H^1_{\#,*}(\Omega_f)$, $p \in H^1_{\#,*}(\Omega_b)$, and $\bu \in \mathbf U$ are encoded in these spaces. Moreover, the boundary conditions on $\Gamma_{\text{lat}}$ are observed directly: the periodic conditions of Dirichlet type are directly encoded in the $H^1_{\#}$ spaces; the Neumann conditions are {\em natural}, and obtained in the standard fashion.

We must now recover the interface conditions on $\Gamma_I$. Since $\phi = [\vec{u},p,\vec{v}]^T \in \Sigma$ satisfies \eqref{resolvent2}, we again invoke the generalized Green's identities, the relations: \eqref{resolventd}, \eqref{vv}, and \eqref{vvv}, as well as the now established boundary conditions away from the interface $\Gamma_I$. Recalling \eqref{BBprob}, we have that for every $\psi=[\bxi,q,\bzeta]^T \in \Sigma$
\begin{align*}
    &~\braket{\sigma_b\vec{n}_b,\bxi}_{\Gamma_I} + \braket{k\nabla p\cdot\vec{n}_b,q}_{\Gamma_I} + \braket{\sigma_f\vec{n}_f,\bzeta}_{\Gamma_I} +({p,\bzeta\cdot\vec{e}_3})_{\Gamma_I} - ({p,\bxi\cdot\vec{e}_3})_{\Gamma_I}\\
    &-( \mathbf v \cdot \mathbf e_3,q)_{\Gamma_I}+ ( {\vec{u}\cdot\vec{e}_3,q})_{\Gamma_I} + ({\beta(\vec{v}-\vec{u})\cdot\btau,\bzeta\cdot\btau})_{\Gamma_I} - ({\beta(\vec{v}-\vec{u})\cdot\btau,\bxi\cdot\btau})_{\Gamma_I}\\
    =&~ ({\vec{f}_1\cdot\vec{e}_3,q})_{\Gamma_I} - ({\vec{f}_1\cdot\btau,\beta(\bzeta-\bxi)\cdot\btau})_{\Gamma_I}.
\end{align*}
Above, the boundary integrals are interpreted (as in \eqref{IBP}) in the appropriately justified dual senses for the $\langle \cdot , \cdot \rangle$ terms.
From which we infer (identifying valid $L^2$ boundary integrals with the associated duality pairings):
\begin{align}\label{bdryVanish}
    \braket{-\sigma_b\vec{e}_3\cdot\vec{e}_3-p,\bxi\cdot\vec{e}_3}_{\Gamma_I} &+ \braket{-\sigma_b\vec{e}_3\cdot\btau - \beta[\vec{v} - (\vec{u}-\vec{f}_1)]\cdot\btau,\bxi\cdot\btau}_{\Gamma_I} \nn\\
    &+ \braket{-k\nabla p\cdot\vec{e}_3 - [\vec{v} - (\vec{u}-\vec{f}_1)]\cdot\vec{e}_3,q}_{\Gamma_I} + \braket{\sigma_f\vec{e}_3\cdot\vec{e}_3+p,\bzeta\cdot\vec{e}_3}_{\Gamma_I} \nn\\
    &+ \braket{-\sigma_f\vec{e}_3\cdot\btau - \beta[\vec{v} - (\vec{u}-\vec{f}_1)]\cdot\btau,\bzeta\cdot\btau}_{\Gamma_I} \nn\\
    =&~ 0.
\end{align}

Since \eqref{bdryVanish} holds for all $\psi \in \Sigma$, we can single out each term in \eqref{bdryVanish} by zeroing the test functions that do not appear in each subsequent pairing, considering purely normal or purely tangential $\bxi \in \vec{U}$ or $\bzeta \in \vec{H}^1_{\#,\ast}(\Omega_f)$, as needed. In particular, this gives us
\begin{align}\label{uno}
    \braket{-\sigma_b\vec{e}_3\cdot\vec{e}_3-p,\bxi\cdot\vec{e}_3}_{\Gamma_I} = 0 &~~~~~ \forall\, \bxi \in \{\vec{u} \in \vec{U}~:~\vec{u}\cdot\btau = 0 \text{ on } \Gamma_I\},\\\label{dos}
    \braket{-\sigma_b\vec{e}_3\cdot\btau - \beta[\vec{v} - (\vec{u}-\vec{f}_1)]\cdot\btau,\bxi\cdot\btau}_{\Gamma_I} = 0 &~~~~~ \forall\, \bxi \in \{\vec{u} \in \vec{U}~:~\vec{u}\cdot\vec{e}_3 = 0 \text{ on } \Gamma_I\},\\\label{tres}
    \braket{-k\nabla p \cdot \vec{e}_3 - [\vec{v} - (\vec{u} - \vec{f}_1)] \cdot\vec{e}_3,q}_{\Gamma_I} = 0 &~~~~~ \forall\, q \in H^1_\#(\Omega_b),\\\label{cuatro}
    \braket{\sigma_f\vec{e}_3\cdot\vec{e}_3+p,\bzeta\cdot\vec{e}_3}_{\Gamma_I} =0 &~~~~~ \forall\, \bzeta \in \{\vec{v} \in \vec{H}^1_{\#,\ast}(\Omega_f)~:~\vec{v}\cdot\btau = 0 \text{ on } \Gamma_I\},\\\label{cinco}
    \braket{-\sigma_f\vec{e}_3\cdot\btau - \beta[\vec{v} - (\vec{u}-\vec{f}_1)]\cdot\btau,\bzeta\cdot\btau}_{\Gamma_I} = 0 &~~~~~\forall\, \bzeta \in \{\vec{v} \in \vec{H}^1_{\#,\ast}(\Omega_f)~:~\vec{v}\cdot\vec{e}_3 = 0 \text{ on } \Gamma_I\}.
\end{align}

Subsequently, let $g \in H^{1/2}(\Gamma_I)$. Then, if we take $\bxi \equiv g \mathbf e_3$ in \eqref{uno}, we have that
\begin{equation} 
p=-\sigma_b\mathbf e_3\cdot \mathbf e_3 \in H^{-1/2}(\Gamma_I).
\end{equation}
after recalling \eqref{www}. In the same way, we obtain the following equations:
\begin{align}\label{ICrecovery1}
    \beta[\vec{v} - (\vec{u} - \vec{f}_1)]\cdot\btau = -\sigma_b \vec{e}_3\cdot\btau &\in H^{-1/2}(\Gamma_I),\\
    -k\nabla p\cdot\vec{e}_3 = [\vec{v} - (\vec{u} - \vec{f}_1)]\cdot\vec{e}_3 &\in H^{-1/2}(\Gamma_I),\\
    p = -\sigma_f\vec{e}_3\cdot\vec{e}_3 &\in H^{-1/2}(\Gamma_I),\\
    \beta[\vec{v} - (\vec{u} - \vec{f}_1)]\cdot\btau = -\sigma_f \vec{e}_3\cdot\btau &\in H^{-1/2}(\Gamma_I).
\end{align}
We moreover observe that $\sigma_b\vec{e}_3\cdot\vec{e}_3 = \sigma_f\vec{e}_3\cdot\vec{e}_3 \in H^{-1/2}(\Gamma_I)$ from  \eqref{uno} and \eqref{cuatro}, and that $\sigma_b\vec{e}_3\cdot\btau = \sigma_f\vec{e}_3\cdot\btau \in H^{-1/2}(\Gamma_I)$ from \eqref{dos} and \eqref{cinco}. Hence, we indeed guarantee the conservation of normal stresses between $\Omega_b$ and $\Omega_f$, and we may rewrite \eqref{ICrecovery1} by
\begin{equation}\label{ICrecovery2}
\begin{cases}
    [\vec{v} - (\vec{u} - \vec{f}_1)]\cdot\vec{e}_3 = -k\nabla p\cdot\vec{e}_3 &\in H^{-1/2}(\Gamma_I),\\
    \beta[\vec{v} - (\vec{u} - \vec{f}_1)]\cdot\btau = -\sigma_f \vec{e}_3\cdot\btau &\in H^{-1/2}(\Gamma_I),\\
    p = -\sigma_f\vec{e}_3\cdot\vec{e}_3 &\in H^{-1/2}(\Gamma_I),\\
    \sigma_b\vec{e}_3 = \sigma_f\vec{e}_3 &\in \vec{H}^{-1/2}(\Gamma_I).
\end{cases}
\end{equation}
Finally, since we have $\vec{w} = \vec{u} - \vec{f}_1 \in \vec{U}$, we may plug this in for the quantity $\bw=\vec{u} - \vec{f}_1$ (as in \eqref{resolvent1}), appearing in the first and second equations of \eqref{ICrecovery2} to yield the remaining conditions in $\mathcal{D}(\cA)$:
\begin{equation}\label{ICrecovery3}
\begin{cases}
    (\vec{v} - \vec{w})\cdot\vec{e}_3 = -k\nabla p\cdot\vec{e}_3 &\in H_{\#}^{-1/2}(\Gamma_I),\\
    \beta(\vec{v} - \vec{w})\cdot\btau = -\sigma_f \vec{e}_3\cdot\btau &\in H_{\#}^{-1/2}(\Gamma_I),\\
    p = -\sigma_f\vec{e}_3\cdot\vec{e}_3 &\in H_{\#}^{-1/2}(\Gamma_I),\\
    \sigma_b\vec{e}_3 = \sigma_f\vec{e}_3 &\in \vec{H}_{\#}^{-1/2}(\Gamma_I).
\end{cases}
\end{equation}

We now note that the given $p \in H^1_\#(\Omega_b)$, $\vec{v} \in \vec{H}^1_{\#,\ast}(\Omega_f)$, and $\pi \in L^2(\Omega_f)$ satisfy the relevant conditions to characterize $\pi$ via \eqref{pi}. Namely,  the divergence-free condition from $\bV$   \eqref{vpiRegularity} yields Laplace's equation for $\pi$; the needed trace regularities for well-posedness of \eqref{pi} are observed in \eqref{vResReg}; dotting $\nu_f\Delta \mathbf v+\nabla \pi \in \mathbf V$ and restricting to $\Gamma_I$ yields the second condition of \eqref{pi};  and the third boundary conditions of $\eqref{ICrecovery3}$ above provides the third boundary conditions for \eqref{pi}. Thus, the function $\pi,$ now written as $\pi\equiv  \Pi_1p + \Pi_2\vec{v} + \Pi_3\vec{v} \in L^2(\Omega_f)$ indeed satisfies the elliptic problem
\[
\begin{cases}
    \Delta \pi = 0 &\in L^2(\Omega_f),\\
    \partial_{\vec{e}_3}\pi = \nu\Delta\vec{v}\cdot\vec{e}_3 &\in H^{-3/2}(\Gamma_f),\\
    \pi = p + 2\nu\vec{e}_3\cdot\vec{D}(\vec{v})\vec{e}_3 &\in H^{-1/2}(\Gamma_I).
\end{cases}
\]
Therefore, the resolvent system \eqref{resolvent1} has the solution $[\vec{u},\vec{w},p,\vec{v}]^T \in \mathcal{D}(\cA)$, ultimately demonstrating that $(I-\cA): \mathcal D(\cA) \to X$ is surjective.

As we have recovered the domain for our weak solution, we may invoke elliptic regularity, bearing in mind the derivatives and invocations of the trace theorem. This provides the additional regularities:
\begin{itemize}
    \item  $\restri{-\btau\cdot\sigma_f\vec{e}_3}=\left.\beta(\vec{v}-\vec{w})\cdot\btau\right|_{\partial \Omega_f} \in H^{1/2}(\Gamma_I)$;~~the trace theorem applied to $\bv$ and $\bw$,
    
    \item $\restri{\partial_{\vec{e}_3}p}= \restri{(\vec{v} - \vec{w})\cdot\vec{e}_3}  \in H^{1/2}(\Gamma_I)$; ~~the trace theorem applied to $\bv$ and $\bw$,
    
    \item $p \in H_\#^2(\Omega_b)$; ~~invoking elliptic regularity for $p$ with $H^{1/2}$ Neumann trace from above,
    
    \item $\restri{-\vec{e}_3 \cdot \sigma_f \vec{e}_3} = \restri{p} \in H^{3/2}(\Gamma_I)$; ~~from the pressure coupling condition and the bullet above
    
    \item adding normal and tangential parts for the fluid normal stress, we have $$\restri{\sigma_f\vec{e}_3} = \restri{\sigma_b\vec{e}_3} \in \vec{H}^{1/2}(\Gamma_I).$$
\end{itemize}

\subsection{Concluding the Argument for Semigroup Generation}

As we have shown dissipativity (Section \ref{dissip}) and maximality (Section \ref{max1}) for the dynamics generator $\mathcal A$ as defined in Section \ref{gendef}, we can invoke the Lumer-Philips theorem \cite{pazy,kesavan} to obtain the generation of a $C_0$-semigroup on the underlying Hilbert space $X$. Theorem \ref{th:main1} is thus proved. 
\end{proof}

The statements about {\em strong solutions} in Corollary \ref{coro}---(Ia) and (II)---follow immediately from the standard theory of semigroups. Namely, when $\mathbf y_0 \in \mathcal D(\cA)$ (and $\mathcal F=[\mathbf 0, \mathbf F_b,S,\mathbf F_f]^T \in H^1(0,T;X)$ for the inhomogeneous problem) we observe that the function $\mathbf y(t) =e^{\mathcal A t}\mathbf y_0$ is the unique solution (point-wise in time) to the Cauchy problem
$$\dot{\mathbf y} = \mathcal A\mathbf y+\mathcal F$$ with $\mathbf y(0)=\mathbf y_0$; in addition, $e^{\mathcal A t}\mathbf y_0 \in C^1((0,T);X)\cap C([0,T];\mathcal D(\cA))$. From the definition of $\cA$ and $\mathcal D(\cA)$, the equations in \eqref{systemfull}--\eqref{systemfull1} hold point-wise in time, $a.e.$ in $\mathbf x$. Moreover, the boundary conditions in \eqref{uplow} and \eqref{IC1*}--\eqref{IC4*} hold point-wise in time, in the sense of the definition $\mathcal D(\cA)$ in Definition \ref{diffdomain}. For all $t \in (0,T)$, we then write:
\begin{align}
     \vec{u}_{tt} =&~ \rho_b^{-1}\vec{F}_b-\mathcal E_0\bu-\alpha\rho_b^{-1}\nabla p_b \in \mathbf L^2(\Omega_b),\\
    [p_b + \alpha c_0^{-1} \nabla \cdot \vec{u}]_t =&~ c_0^{-1}S-A_0p_b \in L^2(\Omega_b),\\
     \vec{v}_t  =&~ \rho_f^{-1}\vec{F}_f + \nu\rho_f^{-1} \Delta \vec{v} - \rho_f^{-1}\nabla p_f\in \mathbf L^2(\Omega_f),\quad \quad \nabla\cdot\vec{v} = 0\in L^2(\Omega_f), 
\end{align}
The vector $[\bu_t,p_b,\bv]^T \in C\big([0,T]; \mathbf H^1_{\#}(\Omega_b) \times H^1_{\#}(\Omega_b) \times \mathbf H^1_{\#}(\Omega_f)\big)$ can be used to test the above equations. Indeed, after multiplying (resp.) we integrate the first equation in $L^2(0,T;\mathbf L^2(\Omega_b))$, the second in $L^2(0,T;L^2(\Omega_b))$, and the last in $L^2(0,T;\mathbf L^2(\Omega_f))$. Spatial integration by parts on the RHS above is justified in $\mathcal D(\cA)$. Thus, the the formal calculations for the energy identity in Section \ref{ebal} are justified. Recalling
\begin{align*} e(t) \equiv &~ \frac{1}{2}\Big[\rho_b\|\vec{u}_t\|_{0,b}^2 + \|\vec{u}\|_E^2 + c_0\|p_b\|_{0,b}^2 + \rho_f\|\vec{v}\|_{0,f}^2\Big]\ \\ 
d_0^t \equiv &~ \int_0^t\big[k\|\nabla p_b\|_{0,b}^2 
+ 2\nu \|\vec{D}(\vec{v})\|_{0,f}^2 
+ \beta \|(\vec{v}-\vec{u}_t)\cdot\boldsymbol{\tau} \|_{\Gamma_I}^2 \big] d\tau,\end{align*}
we obtain, for strong solutions, the energy identity
\begin{equation}
e(t)+d_0^t=e(0)+\int_0^t\big[(\mathbf F_f,\bv)_{L^2(\Omega_f)}+(\mathbf F_b,\bu_t)_{\mathbf L^2(\Omega_b)}+( S, p_b)_{L^2(\Omega_b)}\big] d\tau.
\end{equation}
In anticipation of limit passage in the next section, we can relax the above inner products to the appropriate duality pairings:
\begin{equation}\label{onetouse}\small
e(t)+d_0^t=e(0)+\int_0^t\big[\langle\mathbf F_f,\bv\rangle_{(\vec{H}^1_{\#,\ast}(\Omega_f)\cap \vec{V})'\times (\vec{H}^1_{\#,\ast}(\Omega_f)\cap \vec{V})}+(\mathbf F_b,\bu_t)_{\mathbf L^2(\Omega_b)}+\langle S, p_b\rangle_{ [H^{1}_{\#,*}(\Omega_b)]'\times  H^{1}_{\#,*}(\Omega_b)}\big] d\tau.
\end{equation}

\section{Weak Solutions}\label{weaksolsec}
It remains to prove the statements  concerning weak solutions. Namely Corollary \ref{coro}, parts $(Ib)$ and $(III)$, which will hold for fixed $c_0>0$. Subsequently, we will show that there is a weak solution for the dynamics with $c_0$, taking the limit $c_0 \searrow 0$. That will constitute the central part of the proof of Theorem \ref{th:main2}.

\subsection{Finishing Proof of Corollary \ref{coro}}\label{sols1}
We prove $(III)$, which will subsume the proof of $(Ia)$ and finish the proof of Corollary \ref{coro}. Let $c_0 > 0$. Namely, we must construct a weak solution from heretofore obtained strong solutions. 

The first step is to observe that a strong solution (emanating from $\mathcal D(\cA)$) yields a weak solution. To that end, let $\mathcal F = [\mathbf 0, \mathbf F_b,S,\mathbf F_f]^T \in H^1(0,T;X)$ and take initial data $[\vec{u}_0,\vec{u}_1,p_0,\vec{v}_0]^T \in \mathcal{D}(\cA)$. From \eqref{coro} $(II)$, we have a strong solution $\mathbf y=[\vec{u},\vec{w},p,\vec{v}]^T \in C([0,T];\mathcal D(\mathcal A))\cap C^1((0,T);X)$ to the Cauchy problem \eqref{cauchy}.  It is then direct to show that $\vec{y}$ corresponds to  a weak solution, as in Definition \ref{weaksols}. Namely, take the $X$-inner product, given by \eqref{innerX}, of the full system \eqref{cauchy} with test function $[\bxi,\bxi_t,q,\bzeta]^T$, where $[\bxi,q,\bzeta]^T \in \Vtest$. This gives
\begin{align}
    \rho_b((\bw_t,\bxi))_{\Omega_b} +&~ \rho_b((\cE_0(\bu),\bxi))_{\Omega_b} + \rho_b((\alpha\rho_b^{-1}\nabla p,\bxi))_{\Omega_b} \nn\\
    &+~c_0((p_t,q))_{\Omega_b} + c_0(\alpha c_0^{-1}\nabla\cdot \vec{w}, q))_{\Omega_b}  + c_0((A_0p,q))_{\Omega_b} \nn\\
    &+~\rho_f((\vec{v}_t,\bzeta))_{\Omega_f} - \rho_f((\rho_f^{-1}G_1(p),\bzeta))_{\Omega_b} - \rho_f((\rho_f^{-1}[\nu\Delta + G_2 + G_3]\vec{v},\bzeta))_{\Omega_f} \nn\\
    &+~ ((\vec{u}_t,\bxi_t))_E -((\vec{w},\bxi_t))_E\\
    =&~((\mathbf F_b, \bzeta_t))_{\mathbf L^2(\Omega_b)}+((S,q))_{L^2(\Omega_b)}+((\mathbf F_f,\bzeta))_{\mathbf L^2(\Omega_f)}. \nn
\end{align}
From which, applying Green's Theorem and temporal integration by parts, gives
\begin{align*}
    \rho_b(\vec{w}, \bxi)\big|_{t = 0} -& \rho_b((\bw,\bxi_t))_{\Omega_b} + ((\sigma_b(\vec{u},p), \bxi))_{\Omega_b} - (\langle \sigma_b(\vec{u},p)\vec{n}_b, \bxi\rangle)_{\partial\Omega_b}\\
    &+~(c_0p + \alpha\nabla\cdot\vec{w},q)\big|_{t=0} - ((c_0p + \alpha\nabla\cdot\vec{w},q_t))_{\Omega_b} + ((k\nabla p, \nabla q))_{\Omega_b} - (\langle k\partial_{\vec{n}_b}p,q\rangle)_{\partial \Omega_b}\\
    &+~\rho_f(\vec{v},\bzeta)_{\Omega_f}\big|_{t=0} - \rho_f((\vec{v},\bzeta_t))_{\Omega_f} - ((\nu\Delta\vec{v},\bzeta))_{\Omega_f} + ((\nabla\pi(p_b,\vec{v}),\bzeta))_{\Omega_f}\\
    &+~ ((\vec{u}_t,\bxi_t))_E -((\vec{w},\bxi_t))_E\\
   =&((\mathbf F_b, \bzeta_t))_{\mathbf L^2(\Omega_b)}+((S,q))_{L^2(\Omega_b)}+((\mathbf F_f,\bzeta))_{\mathbf L^2(\Omega_f)}.
\end{align*}
Note that we may now substitute from the first equation of \eqref{cauchy}: $\vec{u}_t = \vec{w}$. So we obtain:
\begin{align*}
    \rho_b(\vec{u}_t, \bxi)\big|_{t = 0} -& \rho_b((\bu_t,\bxi_t))_{\Omega_b} + ((\sigma_b(\vec{u},p), \bxi))_{\Omega_b} - (\langle \sigma_b(\vec{u},p)\vec{n}_b, \bxi\rangle)_{\partial\Omega_b}\\
    &+~(c_0p + \alpha\nabla\cdot\vec{u}_t,q)\big|_{t=0} - ((c_0p + \alpha\nabla\cdot\vec{u}_t,q_t))_{\Omega_b} + ((k\nabla p, \nabla q))_{\Omega_b} - (\langle k\partial_{\vec{n}_b}p,q\rangle)_{\partial \Omega_b}\\
    &+~\rho_f(\vec{v},\bzeta)_{\Omega_f}\big|_{t=0} - \rho_f((\vec{v},\bzeta_t))_{\Omega_f} - ((\nu\Delta\vec{v},\bzeta))_{\Omega_f} + ((\nabla\pi(p_b,\vec{v}),\bzeta))_{\Omega_f}\\
   =&((\mathbf F_b, \bzeta_t))_{\mathbf L^2(\Omega_b)}+((S,q))_{L^2(\Omega_b)}+((\mathbf F_f,\bzeta))_{\mathbf L^2(\Omega_f)}. 
\end{align*}
Finally, realize the boundary conditions in $\mathcal{D}(\cA)$, then move the remaining terms to the RHS to obtain
\begin{align}\label{weak1}
    -&~ \rho_b((\vec{u}_t,\bxi_t))_{\Omega_b} + ((\sigma_b(\vec{u},p), \nabla \bxi))_{\Omega_b} - ((c_0 p + \alpha\nabla\cdot\vec{u}, \partial_t q))_{\Omega_b} + ((k\nabla p,\nabla q))_{\Omega_b} \nn\\
    &- \rho_f((\vec{v},\bzeta_t))_{\Omega_f} + 2\nu((\vec{D}(\vec{v}),\vec{D}(\bzeta)))_{\Omega_f} + (({p,(\bzeta-\bxi)\cdot\vec{e}_3}))_{\Gamma_I} - (({\vec{v}\cdot\vec{e}_3,q}))_{\Gamma_I} \nn\\
    &- (({\vec{u}\cdot\vec{e}_3,\partial_tq}))_{\Gamma_I} + \beta (({\vec{v}\cdot\btau,(\bzeta - \bxi)\cdot\btau}))_{\Gamma_I} + \beta(({\vec{u}\cdot\btau, (\bzeta_t - \bxi_t)\cdot\btau}))_{\Gamma_I} \nn\\
    =&~ \rho_b(\vec{u}_t,\bxi)_{\Omega_b}\big|_{t=0} + (c_0 p + \alpha\nabla\cdot\vec{u},q)_{\Omega_b}\big|_{t=0} + \rho_f(\vec{v},\bzeta)_{\Omega_f}\big|_{t=0} + ({\vec{u}\cdot\vec{e}_3,q})_{\Gamma_I}\big|_{t=0} \nn\\
    &~- \beta({\vec{u}\cdot\btau, (\bzeta - \bxi)\cdot\btau})_{\Gamma_I}\big|_{t=0} + ((\vec{F}_b,\bxi))_{\Omega_b} + (( S,q))_{\Omega_b} + (( \vec{F}_f,\bzeta))_{\Omega_f}
\end{align}
Thus, for the strong solution $\mathbf y =[\bu,\bu_t,p,\bv]^T$, we obtain that $[\vec{u},p,\vec{v}]$ satisfies Definition \ref{weaksols} and is thus a weak solution to \eqref{systemfull}--\eqref{endfull}. Moreover, as the solution $\mathbf y$ is strong, it satisfies the energy identity \eqref{onetouse}. Using Cauchy-Schwarz/dual estimates, we obtain immediately that {\small
\begin{align} \label{eident1} 
    \|\vec{u}_t(t)\|_{0,b}^2 +& \|\vec{u}(t)\|_E^2 + c_0\|p(t)\|_{0,b}^2 + \|\vec{v}(t)\|_{0,b}^2 \\
    &+ k\int_0^t \|\nabla p\|_{0,b}^2 d\tau + 2\nu\int_0^t\|\vec{D}(\vec{v})\|_{0,f}^2\tau + \int_0^t \beta\|(\vec{v} - \vec{u}_t)\cdot\btau\|_{\Gamma_I}d\tau\nn\\
    \lesssim&~ \|\bu_1\|_{0,b}^2 + \|\bu_0\|_E^2 + c_0 \|p_0\|_{0,b}^2 + \|\bv_0\|_{0,f}^2+\int_0^t[||\mathbf F_b||_{\vec{L}^2(\Omega_b)}^2+||\mathbf F_f||_{(\vec{H}^1_{\#,\ast}(\Omega_f)\cap \vec{V})'}^2+||S||^2_{[H^1_{\#,*}(\Omega_b)]'}]d\tau .\nn
\end{align}}
We note that we have relaxed the requirements on the source terms above, in preparation for limit passage in the next step. 
Similarly, for the arguments to follow, we read-off the regularities associated to the above estimate: \begin{itemize} \item  $\vec{u} \in L^\infty(0,T;\vec{U}) \cap W^{1,\infty}(0,T;\vec{L}^2(\Omega_b))$, 
\item $p \in L^2(0,T;H_\#^1(\Omega_b))$ with $\sqrt{c_0}p \in L^\infty(0,T;L^2(\Omega_b))$, \item and  $\vec{v} \in L^\infty(0,T;\vec{V}) \cap L^2(0,T; \vec{V} \cap \vec{H}_{\#,\ast}^1(\Omega_f))$.
\end{itemize}
From these estimates, we observe that $[\bu,p,\bv]^T \in \mathcal{V}_{\text{sol}}$. Thus, we have shown that a strong solution corresponding to $\mathbf y_0 \in \mathcal D(\cA)$ and $\mathcal F \in H^1(0,T;X)$ is a weak solution in the sense of Definition \ref{weaksols}. 

The next step is to show that for weaker data, that a weak solution can be constructed. This will be done by approximation, as is standard \cite{ball,pazy}. Let  $\mathbf y_0 \in X$ and $\mathbf F_f \in L^2\big(0,T;(\vec{H}^1_{\#,\ast}(\Omega_f)\cap \vec{V})'\big)$, $\mathbf F_b \in L^2(0,T;\vec{L}^2(\Omega_b))$, and $S \in L^2(0,T; [H^{1}_{\#,*}(\Omega_b)]')$. 
Since  $\mathcal A$ generates on the Hilbert space $X$, $\mathcal{D}(\cA)$ is dense in $X$ \cite{pazy}. Thus there is a sequence of data $\vec{y}_0^n \in \mathcal{D}(\cA)$ such that $\vec{y}_0^n \to \vec{y}_0$. In addition, we can consider a sequence $\mathcal F^n = [\mathbf 0, \mathbf F^n_b,S^n,\mathbf F^n_f] \in C^{\infty}\left([0,T]; \mathbf 0 \times \mathbf L^2(\Omega_b) \times L^2(\Omega_b) \times \mathbf L^2(\Omega_f)\right) \subset H^1(0,T;X)$ so that 
$$\small \mathcal F^n \to \mathcal F ~\text{ in }~\{\mathbf 0\} \times L^2\big(0,T;(\vec{H}^1_{\#,\ast}(\Omega_f)\cap \vec{V})'\big) \times  L^2(0,T;\vec{L}^2(\Omega_b)) \times S \in L^2(0,T; [H^{1}_{\#,*}(\Omega_b)]').$$
Then, for each $(\mathbf y^n_0,\mathcal F^n)$ there a weak solution $\vec{y}^n \in \mathcal{V}_{\text{sol}}$ which satisfies the weak form
\begin{align}\label{weakn}
    -&~ \rho_b((\vec{u}_t^n,\bxi_t))_{\Omega_b} + ((\sigma_b(\vec{u}^n,p^n), \nabla \bxi))_{\Omega_b} - ((c_0 p^n + \alpha\nabla\cdot\vec{u}^n, \partial_t q))_{\Omega_b} + ((k\nabla p^n,\nabla q))_{\Omega_b} \nn\\
    &- \rho_f((\vec{v}^n,\bzeta_t))_{\Omega_f} + 2\nu((\vec{D}(\vec{v}^n),\vec{D}(\bzeta)))_{\Omega_f} + (({p^n,(\bzeta-\bxi)\cdot\vec{e}_3}))_{\Gamma_I} - (({\vec{v}^n\cdot\vec{e}_3,q}))_{\Gamma_I} \nn\\
    &- (({\vec{u}^n\cdot\vec{e}_3,\partial_tq}))_{\Gamma_I} + \beta (({\vec{v}^n\cdot\btau,(\bzeta - \bxi)\cdot\btau}))_{\Gamma_I} + \beta(({\vec{u}^n\cdot\btau, (\bzeta_t - \bxi_t)\cdot\btau}))_{\Gamma_I} \nn\\
    =&~ \rho_b(\bu_1^n,\bxi)_{\Omega_b}\big|_{t=0} + (c_0 p_0^n + \alpha\nabla\cdot \bu_0^n,q)_{\Omega_b}\big|_{t=0} + \rho_f(\bv_0^n,\bzeta)_{\Omega_f}\big|_{t=0} + ({\bu_0^n\cdot\vec{e}_3,q})_{\Gamma_I}\big|_{t=0} \nn\\
    &~- \beta({\bu_0^n\cdot\btau, (\bzeta - \bxi)\cdot\btau})_{\Gamma_I}\big|_{t=0}+ ((\vec{F}^n_b,\bxi))_{\Omega_b} + (( S^n,q))_{\Omega_b} + (( \vec{F}^n_f,\bzeta))_{\Omega_f}
\end{align}
for all $[\bxi,q,\bzeta]^T \in \Vtest$, and energy identity 
\begin{align}\label{eidentn}
    \|\vec{u}_t^n(t)\|_{0,b}^2 +& \|\vec{u}^n(t)\|_E^2 + c_0\|p^n(t)\|_{0,b}^2 + \|\vec{v}^n(t)\|_{0,f}^2 \nn\\
    & + k\int_0^t \|\nabla p^n\|_{0,b}^2 d\tau + 2\nu\int_0^t\|\vec{D}(\vec{v}^n)\|_{0,f}^2d\tau + \int_0^t \beta\|(\vec{v}^n - \vec{u}^n_t)\cdot\btau\|_{\Gamma_I}^2d\tau \nn\\
    \lesssim &~ \|\bu_1^n\|_{0,b}^2 + \|\bu_0^n\|_E^2 + c_0 \|p_0^n\|_{0,b}^2 + \|\bv_0^n\|_{0,f}^2 \\ \nn &+\int_0^t[||\mathbf F^n_b||_{\vec{L}^2(\Omega_b)}^2+||\mathbf F^n_f||_{(\vec{H}^1_{\#,\ast}(\Omega_f)\cap \vec{V})'}^2+||S^n||^2_{[H^1_{\#,*}(\Omega_b)]'}]d\tau,
\end{align}

Now, we invoke the convergence $\mathbf y_0^n \to \mathbf y_0 \in X$ (topologically equivalent to $\mathbf H^1_{\#}(\Omega_b) \times \mathbf L^2(\Omega_b) \times L^2(\Omega_b)\times \mathbf L^2(\Omega_f)$ here) and $$\small \mathcal F^n \to \mathcal F ~\text{ in }~\{\mathbf 0\} \times L^2\big(0,T;(\vec{H}^1_{\#,\ast}(\Omega_f)\cap \vec{V})'\big) \times \mathbf F_b \in L^2(0,T;\vec{L}^2(\Omega_b)) \times S \in L^2(0,T; [H^{1}_{\#,*}(\Omega_b)]').$$
This yields the uniform-in-$n$ estimate 
\begin{align}\label{eidentn*}
    \|\vec{u}_t^n(t)\|_{0,b}^2 +& \|\vec{u}^n(t)\|_E^2 + c_0\|p^n(t)\|_{0,b}^2 + \|\vec{v}^n(t)\|_{0,f}^2 \nn\\
    & + k\int_0^t \|\nabla p^n\|_{0,b}^2 d\tau + 2\nu\int_0^t\|\vec{D}(\vec{v}^n)\|_{0,f}^2d\tau + \int_0^t \beta\|(\vec{v}^n - \vec{u}^n_t)\cdot\btau\|_{\Gamma_I}^2d\tau \nn\\
    \lesssim &~ \|\bu_1\|_{0,b}^2 + \|\bu_0\|_E^2 + c_0 \|p_0\|_{0,b}^2 + \|\bv_0\|_{0,f}^2 \\ \nn &+\int_0^t[||\mathbf F_b||_{\vec{L}^2(\Omega_b)}^2+||\mathbf F_f||_{(\vec{H}^1_{\#,\ast}(\Omega_f)\cap \vec{V})'}^2+||S||^2_{[H^1_{\#,*}(\Omega_b)]'}]d\tau.
\end{align}
By Banach-Alaoglu, for the sequence $\{\bu^n,\bu^n_t,p^n,\bv^n\}$, we thus have the following weak-$*$ (subsequential) limits and limit points, upon relabeling:
\begin{align*}
\bu^n \rightharpoonup^* &~\overline{\bu} \in L^{\infty}(0,T;\mathbf U)\\
\bu_t^n \rightharpoonup^* &~\overline{\bw} \in L^{\infty}(0,T;\mathbf L^2(\Omega_b))\\
\bv^n \rightharpoonup^* &~\overline{\bv} \in L^{\infty}(0,T;\mathbf L^2(\Omega_f))\\
\sqrt{c_0}p^n  \rightharpoonup^* &~\sqrt{c_0}\overline{p} \in L^{\infty}(0,T;L^2(\Omega_b)).
\end{align*}
From the dissipator, we also have the following weak limits and limit points, which are identified as above (by uniqueness of limits):
\begin{align*}
\bv^n \rightharpoonup &~\overline{\bv} \in L^{2}(0,T;\mathbf  H^1_{\#,*}(\Omega_f))\\
p^n  \rightharpoonup &~\overline{p} \in L^{2}(0,T;H^1_{\#,*}(\Omega_b))\\
(\bv^n-\bu_t^n)\cdot \btau \rightharpoonup &~ \overline{\mathbf h}\cdot \btau \in L^2(0,T;L^2(\Gamma_I)).
\end{align*}
Identifying $[\overline{\bu}]_t$ and $\overline{\bw}$ is done in the standard way for the (temporal) distributional derivative (e.g., \cite[Remark 2.1.1]{kesavan}). Moreover, by continuity of the trace mapping $\gamma_0: \mathbf H_{\#}^1(\Omega_f) \to \mathbf H^{1/2}(\partial \Omega_f)$, we have immediately that $\gamma_0[\bv^n] \rightharpoonup \gamma_0[\overline{\bv}] \in H^{1/2}(\partial \Omega_f)$, and, in particular
$$\bv^n\cdot \btau \rightharpoonup \overline{\bv}\cdot \btau \in L^2(0,T;L^2(\Gamma_I)),~~\text{ for }~~\btau = \mathbf e_1, \mathbf e_2.$$ From which we can give meaning to the weak limit for $\bu_t^n \cdot \btau \in L^2(0,T;L^2(\Gamma_I)),~\btau = \mathbf e_1, \mathbf e_2$ via
$$\bu^n_t\cdot \mathbf e_i = \bv^n\cdot \mathbf e_i-\overline{\mathbf h}\cdot \mathbf e_i \rightharpoonup [\bv-\overline{\mathbf h}]\cdot \mathbf e_i.$$
Again, since approximants $\bu^n_t$ have well-defined traces on $\Gamma_I$, and the quantities $[\mathbf v- \overline{\mathbf h}]\cdot \btau$ are well defined as weak limits, we can make the identification $\overline \bu_t \cdot \btau = [\bv-\overline{\mathbf h}]\cdot \btau,$ for $\btau = \mathbf e_1,\mathbf e_2.$

Now, since $p^n  \rightharpoonup ~\overline{p} \in L^{2}(0,T;H^1_{\#,*}(\Omega_b))$ and $\bu^n \rightharpoonup^* ~\overline{\bu} \in L^{\infty}(0,T;\mathbf U)$, we can infer that 
\begin{equation}\label{weake1}\sigma_b(\bu^n,p^n) \rightharpoonup \sigma_b(\overline{\bu},\overline p) \in  L^2(0,T; [L^2(\Omega_b)]^{3\times 3}),\end{equation} noting the embedding $L^2(0,T;Z) \hookrightarrow L^1(0,T;Z)$. 
From $\bv^n \rightharpoonup ~\overline{\bv} \in L^{2}(0,T;\mathbf \mathbf H^1_{\#,*}(\Omega_f))$, it is immediate that
\begin{equation}\label{weake2}\mathbf D(\bv^n) \rightharpoonup \mathbf D(\overline{\bv}) \in L^2(0,T;[L^2(\Omega_f)]^{3\times 3}).\end{equation}

We now consider the terms in \eqref{weakn} individually, and justify their convergence. We recall that
   $$ [\bxi,q,\bzeta]^T \in \mathcal{V}_\text{test} = C^1_0([0,T); \vec{U} \times  H^1_{\#,*}(\Omega_b) \times (\vec{H}^1_{\#,\ast}(\Omega_f)\cap \vec{V})).$$
In particular, the functions in $\mathcal V_{\text{test}}$ are $C^1$-smooth in-time into $H^1$-type spaces, and can be considered as test functions in each of the weak and weak-$*$ limits above. Thus, for the interior terms in \eqref{weakn} as $n\to\infty$, we have
\begin{align*}
    ~ \rho_b((\vec{u}_t^n,\bxi_t))_{\Omega_b}  \to&~ \rho_b((\vec{u}_t,\bxi_t))_{\Omega_b}  \\ 
    \text{since }& \text{ $\bxi_t \in L^1(0,T;\mathbf L^2(\Omega_b))$} ~~ \text{and $\bu^n_t\rightharpoonup^* \overline{\bu}_t \in L^{\infty}(0,T;\mathbf L^2(\Omega_b))$}  \\[.2cm]
     ((\sigma_b(\vec{u}^n,p^n), \nabla \bxi))_{\Omega_b} \to&~ ((\sigma_b(\overline{\vec{u}},\overline p), \nabla \bxi))_{\Omega_b} \\ 
       \text{since} &  \text{  $\nabla \bxi \in L^2(0,T; [L^2(\Omega_b))]^{3\times 3})$} ~~ \text{and \eqref{weake1}}  
       \end{align*}
       \begin{align*}
     ((c_0 p^n + \alpha\nabla\cdot\vec{u}^n, \partial_t q))_{\Omega_b} \to &~    ((c_0 \overline p + \alpha\nabla\cdot \overline{\vec{u}}, \partial_t q))_{\Omega_b}\\ 
       \text{since} &  \text{  $q_t \in L^1(0,T; L^2(\Omega_b))$} ~~ \text{and $c_0p^n+\nabla\cdot \bu^n\rightharpoonup^* c_0\overline p +\nabla \cdot \overline{\bu} \in L^{\infty}(0,T; L^2(\Omega_b))$}  \\[.2cm]
     ((k\nabla p^n,\nabla q))_{\Omega_b} \to & ~ ((k\nabla \overline p, \nabla q))_{\Omega_b} \\
     \text{since} &  \text{  $\nabla q \in L^2(0,T; \mathbf L^2(\Omega_b))$} ~~ \text{and $\nabla p \rightharpoonup \nabla \overline p \in L^{2}(0,T; \mathbf L^2(\Omega_b))$}  \\[.2cm]
     \rho_f((\vec{v}^n,\bzeta_t))_{\Omega_f} \to &~  \rho_f((\overline{\vec{v}},\bzeta_t))_{\Omega_f} \\ 
            \text{since} & \text{ $\bzeta_t \in L^1(0,T; \mathbf L^2(\Omega_b)))$ and $\bv^n \rightharpoonup^* \overline{\bv} \in L^{\infty}(0,T;\mathbf L^2(\Omega_f))$} \\[.2cm]
     2\nu((\vec{D}(\vec{v}^n),\vec{D}(\bzeta)))_{\Omega_f}\to&~2\nu((\vec{D}(\overline{\vec{v}}),\vec{D}(\bzeta)))_{\Omega_f} \\
            \text{since} & \text{ $\mathbf D(\bzeta) \in L^2(0,T; [L^2(\Omega_b))]^{3\times 3})$} ~~ \text{and \eqref{weake2}} 
   \end{align*}
   Now we consider the trace terms. We bear in mind the standard trace theorem on $H^1$-type spaces, i.e., $\mathbf {H}_\#^1(\Omega_i) \xrightarrow{\gamma_0} \vec{H}^{1/2}(\partial \Omega_i) \hookrightarrow \vec{L}^2(\partial \Omega_i)$,  but we suppress the trace operators $\gamma_0[\cdot]$ themselves; we then have
   \begin{align*}
       ~  (({p^n,(\bzeta-\bxi)\cdot\vec{e}_3}))_{\Gamma_I} \to&~  (({p^n,(\bzeta-\bxi)\cdot\vec{e}_3}))_{\Gamma_I}  \\ 
    \text{since }& \text{ $(\bzeta-\bxi)  \in L^2(0,T;\mathbf L^2(\Gamma_I))$} ~~ \text{and $p^n \rightharpoonup \overline{p} \in L^{2}(0,T;\mathbf L^2(\Gamma_I))$}  \\[.2cm]
     (({\vec{v}^n\cdot\vec{e}_3,q}))_{\Gamma_I} \to &~  ((\overline{\vec{v}}\cdot\vec{e}_3,q))_{\Gamma_I} \\
       \text{since }& \text{ $q \in L^2(0,T; L^2(\Gamma_i))$} ~~ \text{and $\bv^n\rightharpoonup \overline{\bv} \in L^{2}(0,T;\mathbf L^2(\Gamma_I))$}  \\[.2cm]
     (({\vec{u}^n\cdot\vec{e}_3,\partial_tq}))_{\Gamma_I} \\
       \text{since }& \text{ $q_t \in L^1(0,T; L^2(\Gamma_I))$} ~~ \text{and $\bu^n\rightharpoonup^* \overline{\bu} \in L^{\infty}(0,T;\mathbf L^2(\Gamma_I))$} 
       \end{align*}
       \begin{align*}
     \beta (({\vec{v}^n\cdot\btau,(\bzeta - \bxi)\cdot\btau}))_{\Gamma_I} \\
      \text{since }& \text{ $(\bzeta-\bxi)  \in L^2(0,T;\mathbf L^2(\Gamma_I))$} ~~ \text{and  $\bv^n\rightharpoonup \overline{\bv} \in L^{2}(0,T;\mathbf L^2(\Gamma_I))$ }  \\[.2cm]
     \beta(({\vec{u}^n\cdot\btau, (\bzeta_t - \bxi_t)\cdot\btau}))_{\Gamma_I} \\
    \text{since }& \text{ $(\bzeta-\bxi)_t  \in L^1(0,T;\mathbf L^2(\Gamma_I))$} ~~~~ \text{and $\bu^n\rightharpoonup^* \overline{\bu} \in L^{\infty}(0,T;\mathbf L^2(\Gamma_I))$}  
   \end{align*}
   Finally, from the assumptions about the strong convergence of the data in the appropriate senses $\mathbf y^n_0\to \mathbf y_0$ and $\mathcal F^n \to \mathcal F$ (as above), we observe weak/weak-$*$ convergence as appropriate (as we did in going from \eqref{eidentn} to \eqref{eidentn*}), yielding 
\begin{align*}
&~ \rho_b(\bu_1^n,\bxi)_{\Omega_b}\big|_{t=0} + (c_0 p_0^n + \alpha\nabla\cdot \bu_0^n,q)_{\Omega_b}\big|_{t=0} + \rho_f(\bv_0^n,\bzeta)_{\Omega_f}\big|_{t=0} + ({\bu_0^n\cdot\vec{e}_3,q})_{\Gamma_I}\big|_{t=0} \nn\\
    &~- \beta({\bu_0^n\cdot\btau, (\bzeta - \bxi)\cdot\btau})_{\Gamma_I}\big|_{t=0}+ ((\vec{F}^n_b,\bxi))_{\Omega_b} + (( S^n,q))_{\Omega_b} + (( \vec{F}^n_f,\bzeta))_{\Omega_f}\\
    \to\\
    &~ \rho_b(\bu_1,\bxi)_{\Omega_b}\big|_{t=0} + (c_0 p_0 + \alpha\nabla\cdot \bu_0,q)_{\Omega_b}\big|_{t=0} + \rho_f(\bv_0,\bzeta)_{\Omega_f}\big|_{t=0} + ({\bu_0\cdot\vec{e}_3,q})_{\Gamma_I}\big|_{t=0} \nn\\
    &~- \beta({\bu_0\cdot\btau, (\bzeta - \bxi)\cdot\btau})_{\Gamma_I}\big|_{t=0}+ (\langle \vec{F}_b,\bxi\rangle)_{\Omega_b} + (\langle S,q\rangle )_{\Omega_b} + (\langle \vec{F}_f,\bzeta\rangle)_{\Omega_f}
\end{align*}

Thus, we may pass to the limit as $n \to \infty$ in \eqref{weakn} to obtain for all $[\bxi,q,\bzeta]^T \in \mathcal V_{\text{test}}$:
\begin{align}\label{weaknnn}
    -&~ \rho_b((\overline{\vec{u}}_t ,\bxi_t))_{\Omega_b} + ((\sigma_b(\overline{\vec{u}} ,\overline p ), \nabla \bxi))_{\Omega_b} - ((c_0 \overline p  + \alpha\nabla\cdot\overline{\vec{u}} , \partial_t q))_{\Omega_b} + ((k\nabla \overline p ,\nabla q))_{\Omega_b} \nn\\
    &- \rho_f((\overline{\bv} ,\bzeta_t))_{\Omega_f} + 2\nu((\vec{D}(\overline{\vec{v}} ),\vec{D}(\bzeta)))_{\Omega_f} + (({\overline p ,(\bzeta-\bxi)\cdot\vec{e}_3}))_{\Gamma_I} - (({\overline{\vec{v}} \cdot\vec{e}_3,q}))_{\Gamma_I} \nn\\
    &- ((\overline{\vec{u}} \cdot\vec{e}_3,\partial_tq))_{\Gamma_I} + \beta (({\overline{\vec{v}} \cdot\btau,(\bzeta - \bxi)\cdot\btau}))_{\Gamma_I} + \beta(({\overline{\vec{u}} \cdot\btau, (\bzeta_t - \bxi_t)\cdot\btau}))_{\Gamma_I} \nn\\
    =&~ \rho_b(\bu_1 ,\bxi)_{\Omega_b}\big|_{t=0} + (c_0 p_0  + \alpha\nabla\cdot \bu_0 ,q)_{\Omega_b}\big|_{t=0} + \rho_f(\bv_0 ,\bzeta)_{\Omega_f}\big|_{t=0} + ({\bu_0 \cdot\vec{e}_3,q})_{\Gamma_I}\big|_{t=0} \nn\\
    &~- \beta({\bu_0 \cdot\btau, (\bzeta - \bxi)\cdot\btau})_{\Gamma_I}\big|_{t=0}+ (( \vec{F} _b,\bxi))_{\Omega_b} + (\langle S ,q\rangle )_{\Omega_b} + (\langle \vec{F} _f,\bzeta\rangle )_{\Omega_f}
\end{align}
 Observing our weak limit points (and subsequent identifications), and invoking weak-lower-semicontinuity of the norm, we obtain the estimate
\begin{align}\label{eidentn**}
    \|\overline{\vec{u}}_t(t)\|_{0,b}^2 +& \|\overline{\vec{u}}(t)\|_E^2 + c_0\|\overline p(t)\|_{0,b}^2 + \|\overline{\vec{v}}(t)\|_{0,f}^2 \nn\\
    & + k\int_0^t \|\nabla \overline p\|_{0,b}^2 d\tau + 2\nu\int_0^t\|\vec{D}(\overline{\vec{v}})\|_{0,f}^2d\tau + \int_0^t \beta\|(\overline{\vec{v}} - \overline{\vec{u}}_t)\cdot\btau\|_{\Gamma_I}^2d\tau \nn\\
    \lesssim &~ \|\bu_1\|_{0,b}^2 + \|\bu_0\|_E^2 + c_0 \|p_0\|_{0,b}^2 + \|\bv_0\|_{0,f}^2 \\ \nn &+\int_0^t[||\mathbf F_b||_{\vec{L}^2(\Omega_b)}^2+||\mathbf F_f||_{(\vec{H}^1_{\#,\ast}(\Omega_f)\cap \vec{V})'}^2+||S||^2_{[H^1_{\#,*}(\Omega_b)]'}]d\tau.
\end{align}
The above demonstrates that our element $[\overline{\bu}, \overline q, \overline{\bv}]^T \in \mathcal V_{\text{sol}}$ is indeed a weak solution of  \eqref{systemfull}--\eqref{endfull}, as per Definition \ref{weaksols}. 

Finally, we observe the trace regularity for $\overline{\bu}_t$, as quoted in Theorem \ref{th:main2}. Namely, our weak solution clearly has that $\gamma_0[(\mathbf v-\mathbf u_t) \cdot \btau] \in L^2(0,T; L^2(\Gamma_I))$ from the final estimate in \eqref{eidentn**}. 
We have thus completed the proof of Corollary \ref{coro} $(III)$. This immediately yields $(I)$, and, further we have obtained the result Theorem \ref{th:main2} {\em in the case when $c_0>0$}. The final thing that remains to be shown is the construction of weak solutions when $c_0=0$.

\subsection{Additional Regularity of Time Derivatives} \label{duals}

As we have now {\em constructed} a particular weak solution, we may observe that that weak solution $[\overline{\bu}, \overline q, \overline{\bv}]^T \in \mathcal V_{\text{sol}}$ in fact satisfies several other identities. Then, we aim to interpret certain time derivatives distributionally, yielding  additional regularity in certain dual senses. 

First, taking $\bzeta=\mathbf 0$ and $q=0$:
\begin{align*}
    -&~ \rho_b((\overline{\vec{u}}_t ,\bxi_t))_{\Omega_b} + ((\sigma_b(\overline{\vec{u}} ,\overline p ), \nabla \bxi))_{\Omega_b} 
  - (({\overline p ,\bxi\cdot\vec{e}_3}))_{\Gamma_I} 
   + \beta (([\overline{\bu}_t-{\overline{\vec{v}}] \cdot\btau,\bxi\cdot\btau}))_{\Gamma_I}\nn\\
    =&~ \rho_b(\bu_1 ,\bxi)_{\Omega_b}\big|_{t=0}  + (( \vec{F} _b,\bxi))_{\Omega_b} 
\end{align*}
Rewriting, and noting that $\bxi \in C_0^1([0,T);\mathbf U)$, we have
\begin{align*}
    -~ \rho_b((\overline{\vec{u}}_t ,\bxi_t))_{\Omega_b} =&~- ((\sigma_b(\overline{\vec{u}} ,\overline p ), \nabla \bxi))_{\Omega_b} 
  + (({\overline p ,\bxi\cdot\vec{e}_3}))_{\Gamma_I} - \beta (([\overline{\bu}_t-{\overline{\vec{v}}] \cdot\btau,\bxi\cdot\btau}))_{\Gamma_I} \\ &~+ \rho_b(\bu_1 ,\bxi)_{\Omega_b}\big|_{t=0}  + (( \vec{F} _b,\bxi))_{\Omega_b} 
\end{align*}
Finally, restricting to $\bxi \in C_0^{\infty}(0,T; \mathbf U)$, we obtain the estimate (bounding all RHS terms above via Cauchy-Schwarz, the trace theorem, and Poincar\'e's inequality):
\begin{equation*}\small
\left|((\overline{\bu}_t,\bxi_t))_{\Omega_b}\right| \lesssim C\left(||\overline{\bu}||_{L^2(0,T;\mathbf U)}+||\overline p||_{L^2(0,T;\mathbf H^1_{\#,*}(\Omega_b)}+||\overline{\bu}_t-\overline{\bv}||_{L^2(0,T;L^2(\Gamma_I))}+||\bF_b||_{L^2(0,T;\mathbf L^2(\Omega_b))}\right)||\bxi||_{L^2(0,T;\mathbf U)}
\end{equation*}
From which we infer that, indeed, for our constructed weak solution $[\overline \bu, \overline p, \overline{\bv}] \in \mathcal V_{\text{sol}}$, we have  
\begin{equation}
\overline{\bu}_{tt} \in L^2(0,T; \mathbf U'),
\end{equation}
with associated bound.

Now, we repeat the same procedure, taking $\bxi=\mathbf 0$ and $q=0$:
\begin{align*}
    - \rho_f((\overline{\bv} ,\bzeta_t))_{\Omega_f} +&~ 2\nu((\vec{D}(\overline{\vec{v}} ),\vec{D}(\bzeta)))_{\Omega_f} + (({\overline p ,\bzeta\cdot\vec{e}_3}))_{\Gamma_I}+ \beta (({[\overline{\vec{v}}-\overline{\bu}_t] \cdot\btau,\bzeta \cdot\btau}))_{\Gamma_I}  \nn\\
    =&~  \rho_f(\bv_0 ,\bzeta)_{\Omega_f}\big|_{t=0} + (\langle \vec{F} _f,\bzeta\rangle )_{\Omega_f}
\end{align*}
Estimating the term $ \rho_f((\overline{\bv} ,\bzeta_t))_{\Omega_f}$ as before, we obtain
\begin{equation}
\overline{\bv}_t \in L^2(0,T; [\mathbf H^1_{\#,*}(\Omega_f) \cap \bV]').
\end{equation}

The final step addresses the regularity of $[c_0\overline p +\nabla \cdot \overline{\bv}]_t$ and $\gamma_0[\overline{\bu}_t\cdot \mathbf e_3]$.
We take $\bzeta=\bxi=\mathbf 0$:
\begin{align*}
   - ((c_0 \overline p  + \alpha\nabla\cdot\overline{\vec{u}} , \partial_t q))_{\Omega_b} &+ ((k\nabla \overline p ,\nabla q))_{\Omega_b} 
   - (({\overline{\vec{v}} \cdot\vec{e}_3,q}))_{\Gamma_I} 
    - ((\overline{\vec{u}} \cdot\vec{e}_3,\partial_tq))_{\Gamma_I}  \nn\\
    =&~~   (c_0 p_0  + \alpha\nabla\cdot \bu_0 ,q)_{\Omega_b}\big|_{t=0}  + ({\bu_0 \cdot\vec{e}_3,q})_{\Gamma_I}\big|_{t=0} 
   + (\langle S ,q\rangle )_{\Omega_b} 
\end{align*}    
  We recall that $q \in C_0^1([0,T); H^1_{\#,*}(\Omega_b))$ above. Restricting to $q \in C_0^{\infty}((0,T); C^{\infty}(\Omega_b))$, we obtain 
  \begin{align*}
   - ((c_0 \overline p  + \alpha\nabla\cdot\overline{\vec{u}} , \partial_t q))_{\Omega_b} 
    - ((\overline{\vec{u}} \cdot\vec{e}_3,\partial_tq))_{\Gamma_I}  
    =    
    (({\overline{\vec{v}} \cdot\vec{e}_3,q}))_{\Gamma_I}  - ((k\nabla \overline p ,\nabla q))_{\Omega_b} 
   +(\langle S ,q\rangle )_{\Omega_b} .
\end{align*}     
While the RHS above can be estimated in terms of given regularities for test functions and solutions, the LHS shows that we cannot decoupled the fluid content time derivative $[c_0\overline p+\nabla \cdot \overline{\bu}]_t$ and the trace $ \overline{\mathbf u}_t\cdot \mathbf e_3\big|_{\Gamma_I}$ in the sense of our weak formulation.
  
  However, further restricting to $q \in C_0^{\infty}((0,T)\times \Omega_b)$, we do obtain the
  estimate
  \begin{equation}\small
  \big|((c_0 \overline p  + \alpha\nabla\cdot\overline{\vec{u}} , \partial_t q))_{\Omega_b}\big| \lesssim ||\nabla \overline p||_{L^2(0,T;\mathbf L^2(\Omega_b))}||\nabla q||_{L^2(0,T;\mathbf L^2(\Omega_b))}
  +||S||_{L^2(0,T; H^{-1}(\Omega_b))}||q||_{L^2(0,T; H^1_{0}(\Omega_b))}
  \end{equation}
  From which we may deduce the regularity
  $$S \in L^2(0,T;H^{-1}(\Omega_b)) \implies (c_0 \overline p  + \alpha\nabla\cdot\overline{\vec{u}})_t \in L^2(0,T:H^{-1}(\Omega_b)).$$
\begin{remark} We note that, in general, we only have $S \in L^2(0,T;[H^1_{\#,*}(\Omega_b)]')$ for weak solutions, thus the requirement in the conditional above constitutes some {\em additional regularity} for the source $S$; if, in particular, $S \in L^2(0,T;L^2(\Omega_b))$, then the estimate (and regularity) above hold. It is also important to note that to obtain this estimate, we had to eliminate the interface terms---a source of complication for estimating these terms directly from the weak form. \end{remark}
    
    Lastly, and perhaps most interestingly, we will estimate the normal component $\gamma_0[\overline{\bu}_t\cdot \mathbf n]$ in a certain distributional sense. We have already obtained the trace regularity 
    $\gamma_0[\overline{\bu}_t\cdot \btau] \in L^2(0,T;L^2(\Gamma_I))$ for our solution. 
    We rewrite the earlier identity to obtain:
    \begin{align*}
      - ((\overline{\vec{u}} \cdot\vec{e}_3,\partial_tq))_{\Gamma_I}   =& ~
    ((c_0 \overline p  + \alpha\nabla\cdot\overline{\vec{u}} , \partial_t q))_{\Omega_b} - ((k\nabla \overline p ,\nabla q))_{\Omega_b} 
     + (({\overline{\vec{v}} \cdot\vec{e}_3,q}))_{\Gamma_I} \nn\\
   &+  (c_0 p_0  + \alpha\nabla\cdot \bu_0 ,q)_{\Omega_b}\big|_{t=0}  + ({\bu_0 \cdot\vec{e}_3,q})_{\Gamma_I}\big|_{t=0} 
   + (\langle S ,q\rangle )_{\Omega_b} 
\end{align*} 
Restricting to $q \in C_0^{\infty}([0,T]; H^1_{\#,*}(\Omega_b))$ and using the trace theorem, we obtain the estimate:
\begin{align}
\big|((\overline{\vec{u}} \cdot\vec{e}_3,\partial_tq))_{\Gamma_I} \big| \lesssim &~\big[||\overline p||_{L^2(0,T;L^2(\Omega_b))}+||\overline{\bu}||_{L^2(0,T;L^2(\Omega_b))}\big]||q_t||_{L^2(0,T;L^2(\Omega_b))}\\ \nn
 & + \big[||\overline p||_{L^2(0,T;H^1_{\#,*}(\Omega_b))}+||S||_{L^2(0,T;[H^1_{\#,*}(\Omega_b)]')}\big]||q||_{L^2(0,T;H^1_{\#,*}(\Omega_b))} \\ \nn
 & + ||\overline{\bv}||_{L^2(0,T;\mathbf H^1_{\#,*}(\Omega_f))}||q||_{L^2(0,T;H^1_{\#,*}(\Omega_b))}
\end{align}
Invoking the energy estimate on the solution \eqref{eidentn**}, we see that, indeed, $\overline{\vec{u}}_t \cdot\vec{e}_3$ does exists as an element in a dual space. Indeed, the estimate above yields
\begin{equation}
\gamma_0[\overline{\vec{u}}_t \cdot\vec{e}_3] \in \big[\gamma_0[W]\big]',
\end{equation}
where $W=H^1(0,T;L^2(\Omega_b))\cap L^2(0,T;H^1_{\#,*}(\Omega_b)).$ 
This is to say that for the space $W$ above, the trace in question resides in the dual of the {\em trace space} for $W$.

We conclude this short section by noting that the above ``regularities" of the time derivatives associated to the weak form in Definition \ref{weaksols} are non-standard; they are convoluted by the coupled nature of the problem. In short, these time derivatives all exist {\em in some} dual sense, but are not the standard spaces associated to the uncoupled, individual equations. This is precisely at issue in showing additional regularity, and, ultimately uniqueness (considering zero data).

\subsection{Finishing the Proof of Theorem \ref{th:main2}}
We have now obtained weak solutions---with estimates, as described in Theorem \ref{th:main2}---for a fixed $c_0>0$. To finish the proof of our second main theorem on weak solutions, we must obtain a weak solution when $c_0=0$. We will do this through a limiting procedure, as was introduced in \cite{bw} and alluded to in \cite{showfiltration}. Let $\{c_0^n\} \searrow 0$ with $0 \le c_0^n \le 1$. In this instance we will not take smooth approximants, but rather a sequence of weak solutions. 

Fix initial data of the form
$$\bu(0)=\bu_0 \in \mathbf U,~~ \bu_t(0)=\bu_1 \in \mathbf L^2(\Omega_b),~~ [c_0^np+\nabla \cdot \bu](0)=d_0 \in L^2(\Omega_b),~~\bv(0)=\bv_0 \in \bV.$$
Thus, the initial datum $[\bu_0,\bu_1,d_0,\bv_0]$ is fixed for each $c_0^n$ as $c_0^n \searrow 0$. 
\begin{remark}
We note that above, specifying $d_0 \in L^2(\Omega_b)$ and $\bu_0 \in \mathbf U$ is equivalent to specifying $c_0^np(0)$. However, we lose the ability to specify $p(0)$ in the limit as $c_0^n \to 0$, as we shall observe below. This is in order to retain compatibility between $d_0$ which will have the property that $d_0 = \nabla \cdot \bu_0$ when $c_0=0$. 
\end{remark}
Now, each particular value of $c_0^n$ corresponds to a particular weak solution $[\bu^n,p^n,\bv^n] \in \mathcal V_{\text{sol}}$ satisfying Definition \ref{weaksols} and guaranteed to exist by the proof in the preceding section. In particular, recall that    $\mathcal{V}_\text{sol} = \mathcal{V}_b\times \mathcal{Q}_b\times\mathcal{V}_f,$ with
\begin{align*}
    \mathcal{V}_b &= \{\vec{u} \in L^\infty(0,T;\vec{U}) ~:~ \vec{u} \in W^{1,\infty}(0,T;\vec{L}^2(\Omega_b))\}, \\
    \mathcal{Q}_b &= \left\{p \in L^2(0,T; H^1_{\#,*}(\Omega_b))~:~c_0^{1/2}p \in L^\infty(0,T;L^2(\Omega_b))\right\}, \\
    \mathcal{V}_f &= L^\infty(0,T;\vec{V}) \cap L^2(0,T; \vec{H}^1_{\#,\ast}(\Omega_f) \cap \vec{V}).
\end{align*}

We obtain for our solution, $[\bu^n,p^n,\bv^n] \in \mathcal V_{\text{sol}}$, that
\begin{align}\label{weaknn}
    -&~ \rho_b((\vec{u}_t^n,\bxi_t))_{\Omega_b} + ((\sigma_b(\vec{u}^n,p^n), \nabla \bxi))_{\Omega_b} - ((c^n_0 p^n + \alpha\nabla\cdot\vec{u}^n, \partial_t q))_{\Omega_b} + ((k\nabla p^n,\nabla q))_{\Omega_b} \nn\\
    &- \rho_f((\vec{v}^n,\bzeta_t))_{\Omega_f} + 2\nu((\vec{D}(\vec{v}^n),\vec{D}(\bzeta)))_{\Omega_f} + (({p^n,(\bzeta-\bxi)\cdot\vec{e}_3}))_{\Gamma_I} - (({\vec{v}^n\cdot\vec{e}_3,q}))_{\Gamma_I} \nn\\
    &- (({\vec{u}^n\cdot\vec{e}_3,\partial_tq}))_{\Gamma_I} + \beta (({\vec{v}^n\cdot\btau,(\bzeta - \bxi)\cdot\btau}))_{\Gamma_I} + \beta(({\vec{u}^n\cdot\btau, (\bzeta_t - \bxi_t)\cdot\btau}))_{\Gamma_I} \nn\\
    =&~ \rho_b(\bu_1,\bxi)_{\Omega_b}\big|_{t=0} + (d_0,q)_{\Omega_b}\big|_{t=0} + \rho_f(\bv_0,\bzeta)_{\Omega_f}\big|_{t=0} + ({\bu_0\cdot\vec{e}_3,q})_{\Gamma_I}\big|_{t=0} \nn\\
    &~- \beta({\bu_0\cdot\btau, (\bzeta - \bxi)\cdot\btau})_{\Gamma_I}\big|_{t=0}+ ((\vec{F}_b,\bxi))_{\Omega_b} + (\langle S,q\rangle )_{\Omega_b} + (\langle \vec{F}_f,\bzeta\rangle )_{\Omega_f}
\end{align}
for all $[\bxi,q,\bzeta]^T \in \Vtest$, as well  as the following estimate
\begin{align}\label{eidentnn}
    \|\vec{u}_t^n(t)\|_{0,b}^2 +& \|\vec{u}^n(t)\|_E^2 + c^n_0\|p^n(t)\|_{0,b}^2 + \|\vec{v}^n(t)\|_{0,f}^2 \nn\\
    & + k\int_0^t \|\nabla p^n\|_{0,b}^2 d\tau + 2\nu\int_0^t\|\vec{D}(\vec{v}^n)\|_{0,f}^2d\tau + \int_0^t \beta\|(\vec{v}^n - \vec{u}^n_t)\cdot\btau\|_{\Gamma_I}^2d\tau \nn\\
    \lesssim &~ \|\bu_1\|_{0,b}^2 + \|\bu_0\|_E^2 + c_0^n \|p_0\|_{0,b}^2 + \|\bv_0\|_{0,f}^2 \\ \nn &+\int_0^t[||\mathbf F_b||_{\vec{L}^2(\Omega_b)}^2+||\mathbf F_f||_{(\vec{H}^1_{\#,\ast}(\Omega_f)\cap \vec{V})'}^2+||S||^2_{[H^1_{\#,*}(\Omega_b)]'}]d\tau,
\end{align}
We now note that $[c_0^n]^{1/2}p^n$ is bounded in $L^{\infty}(0,T;L^2(\Omega_b))$ (from the energy estimate). Thus $\sqrt{c_0^n}p^n$ has a weak-$*$ subsequential limit. Additionally, restricting to that subsequence, and relabeling, we see that $p^n$ is bounded in $L^2(0,T; H^1_{\#,*}(\Omega_b))$ and so also has a weak subsequential limit, say $\bar p$. Again, restricting to this further subsequence and relabeling, we have a coincident limit for the sequence $\sqrt{c_0^n}p^n$ in $L^{\infty}(0,T;L^2(\Omega_b))\cap L^2(0,T;H^1_{\#,*}(\Omega_b))$. But since $c_0^n \searrow 0$ and the sequence $p^n$ is bounded in $L^2(0,T;H^1_{\#,*}(\Omega_b))$, it must be the case that $\sqrt{c_0^n}p^n$ goes to zero in both senses (by uniqueness of limits). 

Returning to \eqref{eidentnn}, we obtain the uniform-in-$n$ bound for the sequence $\{\bu^n,\bu^n_t,p^n,\bv^n\}$:
\begin{align}\label{eidentnn*}
    \|\vec{u}_t^n(t)\|_{0,b}^2 +& \|\vec{u}^n(t)\|_E^2  + \|\vec{v}^n(t)\|_{0,f}^2 \nn\\
    & + k\int_0^t \|\nabla p^n\|_{0,b}^2 d\tau + 2\nu\int_0^t\|\vec{D}(\vec{v}^n)\|_{0,f}^2d\tau + \int_0^t \beta\|(\vec{v}^n - \vec{u}^n_t)\cdot\btau\|_{\Gamma_I}^2d\tau \nn\\
    \lesssim &~ \|\bu_1\|_{0,b}^2 + \|\bu_0\|_E^2  + \|\bv_0\|_{0,f}^2 \\ \nn &+\int_0^t[||\mathbf F_b||_{\vec{L}^2(\Omega_b)}^2+||\mathbf F_f||_{(\vec{H}^1_{\#,\ast}(\Omega_f)\cap \vec{V})'}^2+||S||^2_{[H^1_{\#,*}(\Omega_b)]'}]d\tau,
\end{align}
As in the previous section, it is immediate to obtain the following weak-$*$ (further subsequential) limits and limit points, upon relabeling:
\begin{align*}
\bu^n \rightharpoonup^* &~\overline{\bu} \in L^{\infty}(0,T;\mathbf U)\\
\bu_t^n \rightharpoonup^* &~\overline{\bu}_t \in L^{\infty}(0,T;\mathbf L^2(\Omega_b))\\
\bv^n \rightharpoonup^* &~\overline{\bv} \in L^{\infty}(0,T;\mathbf L^2(\Omega_f))\\
\bv^n \rightharpoonup &~\overline{\bv} \in L^{2}(0,T;\mathbf  H^1_{\#,*}(\Omega_f))\\
p^n  \rightharpoonup &~\overline{p} \in L^{2}(0,T;H^1_{\#,*}(\Omega_b))\\
(\bv^n-\bu_t^n)\cdot \btau \rightharpoonup &~ \overline{\mathbf v}-{\mathbf u}_t\cdot \btau \in L^2(0,T;L^2(\Gamma_I)).
\end{align*}
From above, we have that $c_0^np^n \rightharpoonup^* 0 \in L^{\infty}(0,T;L^2(\Omega_b))$. Thus, since $q_t \in C_0^1([0,T);H^1_{\#,*}(\Omega_b) \subset L^1(0,T;L^2(\Omega_b))$
$$((c^n_0 p^n + \alpha\nabla\cdot\vec{u}^n, \partial_t q))_{\Omega_b} \to ((\alpha\nabla\cdot\vec{u}, \partial_t q))_{\Omega_b}.$$
The remaining terms converge precisely as in Section \ref{weaksolsec}. This yields the following weak solution
\begin{align}
    -&~ \rho_b((\overline{\vec{u}}_t ,\bxi_t))_{\Omega_b} + ((\sigma_b(\overline{\vec{u}} ,\overline p ), \nabla \bxi))_{\Omega_b} - ((c_0 \overline p  + \alpha\nabla\cdot\overline{\vec{u}} , \partial_t q))_{\Omega_b} + ((k\nabla \overline p ,\nabla q))_{\Omega_b} \nn\\
    &- \rho_f((\overline{\bv} ,\bzeta_t))_{\Omega_f} + 2\nu((\vec{D}(\overline{\vec{v}} ),\vec{D}(\bzeta)))_{\Omega_f} + (({\overline p ,(\bzeta-\bxi)\cdot\vec{e}_3}))_{\Gamma_I} - (({\overline{\vec{v}} \cdot\vec{e}_3,q}))_{\Gamma_I} \nn\\
    &- ((\overline{\vec{u}} \cdot\vec{e}_3,\partial_tq))_{\Gamma_I} + \beta (({\overline{\vec{v}} \cdot\btau,(\bzeta - \bxi)\cdot\btau}))_{\Gamma_I} + \beta(({\overline{\vec{u}} \cdot\btau, (\bzeta_t - \bxi_t)\cdot\btau}))_{\Gamma_I} \nn\\
    =&~ \rho_b(\bu_1 ,\bxi)_{\Omega_b}\big|_{t=0} + (c_0 p_0  + \alpha\nabla\cdot \bu_0 ,q)_{\Omega_b}\big|_{t=0} + \rho_f(\bv_0 ,\bzeta)_{\Omega_f}\big|_{t=0} + ({\bu_0 \cdot\vec{e}_3,q})_{\Gamma_I}\big|_{t=0} \nn\\
    &~- \beta({\bu_0 \cdot\btau, (\bzeta - \bxi)\cdot\btau})_{\Gamma_I}\big|_{t=0}+ (( \vec{F} _b,\bxi))_{\Omega_b} + (\langle S ,q\rangle )_{\Omega_b} + (\langle \vec{F} _f,\bzeta\rangle )_{\Omega_f}
\end{align}
and estimate
\begin{align}
    \|\overline{\vec{u}}_t(t)\|_{0,b}^2 +& \|\overline{\vec{u}}(t)\|_E^2 + c_0\|\overline p(t)\|_{0,b}^2 + \|\overline{\vec{v}}(t)\|_{0,f}^2 \nn\\
    & + k\int_0^t \|\nabla \overline p\|_{0,b}^2 d\tau + 2\nu\int_0^t\|\vec{D}(\overline{\vec{v}})\|_{0,f}^2d\tau + \int_0^t \beta\|(\overline{\vec{v}} - \overline{\vec{u}}_t)\cdot\btau\|_{\Gamma_I}^2d\tau \nn\\
    \lesssim &~ \|\bu_1\|_{0,b}^2 + \|\bu_0\|_E^2 + c_0 \|p_0\|_{0,b}^2 + \|\bv_0\|_{0,f}^2 \\ \nn &+\int_0^t[||\mathbf F_b||_{\vec{L}^2(\Omega_b)}^2+||\mathbf F_f||_{(\vec{H}^1_{\#,\ast}(\Omega_f)\cap \vec{V})'}^2+||S||^2_{[H^1_{\#,*}(\Omega_b)]'}]d\tau.
\end{align}
The rest of the analysis (and discussion) in Section \ref{sols1} and Section \ref{duals} now apply to the constructed weak solution for $c_0=0$. We have thus concluded the proof of Theorem \ref{th:main2}.

\section{Appendix}\label{app2}

We include a statement of the version of Babu\v{s}ka-Brezzi which is invoked in the maximality portion of the generation argument.See \cite{ciarletbook,kesavan} for standard presentations, though here we refer to \cite{avalos*}.
\begin{theorem}[Babu\v{s}ka-Brezzi]\label{bbt}
    Let $\Sigma$ and $V$ be real Hilbert spaces. Let $a:\Sigma\times\Sigma \to \mathbb{R}$ and $b:\Sigma\times V \to \mathbb{R}$ be continuous bilinear forms. Let
    \[
        Z = \{\sigma \in \Sigma ~:~ b(\sigma,v) = 0 \text{ for every } v \in V\}.
    \]
    Assume that $a(\cdot,\dot)$ is $Z$-elliptic and that $b(\cdot,\cdot)$ satisfies the {\it inf-sup} condition: there exists a $\beta>0$ such that for every $v \in V$, we have
    \begin{equation}
        \sup_{\substack{\tau \in \Sigma\\\tau \neq 0}} \frac{b(\tau,v)}{\|\tau\|_\Sigma} \geq \beta \|v\|_V.
    \end{equation}
    Let $\kappa \in \Sigma$ and let $\ell \in V$. Then, there exists a unique pair $(\sigma,u) \in \Sigma \times V$ such that
    \begin{equation}
        \begin{cases}
        a(\sigma,\tau) + b(\tau,u) = (\kappa,\tau)_\Sigma \quad &\forall\, \tau \in \Sigma,\\
        \phantom{a(\sigma,\tau) +}~ b(\sigma,v) = (\ell,v)_V \quad &\forall\, v \in V.
        \end{cases}
    \end{equation}
\end{theorem}


\scriptsize
\begin{thebibliography}{99}

\bibitem{yotovLagrange}
\newblock I. Ambartsumyan, E. Khattatov, I. Yotov, P. Zunino,
\newblock A Lagrange multiplier method for a Stokes-Biot fluid-poroelastic structure interaction model,
\newblock {\em Numerische Mathematik}, 513--553, 2018.


\bibitem{yotovNon-NewtBS}
\newblock I. Ambartsumyan, V.J. Ervin, T.  Nguyen, and I. Yotov,
\newblock A nonlinear Stokes-Biot model for the interaction of a non-Newtonian fluid with poroelastic media,
\newblock {\em ESAIM: Mathematical Modelling and Numerical Analysis}, 53(6), pp.1915-1955, 2019.

\bibitem{auriault}
\newblock J.L. Auriault, J.L. and E. Sanchez-Palencia,
\newblock Etude du comportement macroscopique d'un milieu poreux satur\'e d\'eformable,
\newblock {\em Journal de M\'ecanique}, 16(4), pp.575--603. 1977.

\bibitem{avalos*} 
\newblock G. Avalos and M. Dvorak, 
\newblock A new maximality argument for a coupled fluid-structure interaction, with implications for a divergence-free finite element method. 
\newblock {\em Applicationes Mathematicae}, 3(35), pp.259--280, 2008.

\bibitem{AGM} 
\newblock G. Avalos, P.G. Geredeli, and B. Muha,
\newblock Wellposedness, spectral analysis and asymptotic stability of a multilayered heat-wave-wave system. 
\newblock {\em Journal of Differential Equations}, 269(9), pp.7129--7156, 2020.


\bibitem{newGeorge} G. Avalos, E. Gurvich, S. McKnight, and J.T. Webster, 2024. 

\bibitem{avalos}
\newblock G. Avalos, R. Triggiani,
\newblock The coupled PDE system arising in fluid/structure interaction. I. Explicit semigroup generator and its spectral properties,
\newblock {\em Fluids and waves, Contemp. Math} (440),  pp.15--54, 2007.

\bibitem{avalos08}
\newblock G. Avalos, I. Lasiecka, R. Triggiani,
\newblock Higher regularity of a coupled parabolic-hyperbolic fluid-structure interactive system,
\newblock {\em  Georgian Mathematical Journal}, pp.403--437, 2008.



\bibitem{earlynumerics}
\newblock S. Badia, A. Quaini, and A. Quarteroni, 2009. 
\newblock Coupling Biot and Navier-Stokes equations for modelling fluid-poroelastic media interaction. 
\newblock {\em J. Computational Physics}, 228(21),pp.7986--8014.

\bibitem{ball}
\newblock J. M. Ball,
\newblock Strongly continuous semigroups, weak solutions, and the variation of constants formula
\newblock {\em Proceedings of the American Mathematical Society}, pp. 370--373, 1977


\bibitem{biot}
\newblock M.A. Biot,
\newblock General theory of three-dimensional consolidation.
\newblock \emph{Journal of Applied Physics}, 12(2):155--164, 1941.

\bibitem{biot2}
\newblock M.A. Biot and D.G. Willis,
\newblock The elastic coefficients of the theory of consolidation. 
\newblock {\em J. ASME}, 1957.

\bibitem{biot3}
 M.A. Biot,
\newblock Theory of elasticity and consolidation for a porous anisotropic solid.
\newblock {\em Journal of Applied Physics}, 26(2), pp.182--185, 1955.


\bibitem{bw}
\newblock L. Bociu, J.T. Webster,
\newblock Nonlinear quasi-static poroelasticity.
\newblock {\em Journal of Differential Equations}, 296, pp.242--278.

\bibitem{multilayered}
\newblock L. Bociu, S. Canic, B. Muha, and J. T. Webster. 
\newblock Multilayered poroelasticity interacting with stokes flow. 
\newblock \emph{SIAM Journal on Mathematical Analysis}, 53(6) pp.6243--6279, 2021.

\bibitem{bmw}
\newblock L. Bociu, B. Muha, and J. T. Webster. 
\newblock Weak Solutions in Nonlinear Poroelasticity with Incompressible Constituents.
\newblock {\em Nonlinear Analysis: Real World Applications}, 67, p.103563, 2022.



\bibitem{bgsw}
\newblock L. Bociu, G. Guidoboni, R. Sacco, and J. T. Webster,
\newblock Analysis of Nonlinear Poro-Elastic and Poro-Visco-Elastic Models,
\newblock {\em Archives of Rational Mechanics and Analysis}, 222 (2016), pp.1445--1519

\bibitem{bmbm} Buka\v{c}, Martina, Boris Muha, and Abner J. Salgado. Analysis of a diffuse interface method for the Stokes-Darcy coupled problem. {\em ESAIM: Mathematical Modelling and Numerical Analysis}, 57, no. 5 (2023): 2623--2658.

\bibitem{bukac}
Buka\v{c}, M., Yotov, I. and Zunino, P., 2015. An operator splitting approach for the interaction between a fluid and a multilayered poroelastic structure. Numerical Methods for Partial Differential Equations, 31(4), pp.1054-1100.

\bibitem{yotov2015}
\newblock M. Buka\v{c}, I. Yotov, R. Zakerzadeh, and P. Zunino,
\newblock Partitioning strategies for the interaction of a fluid with a poroelastic material based on a Nitsche's coupling approach,
\newblock {\em Computer Methods in Applied Mechanics and Engineering}, 292, pp.138-170, 2015.

\bibitem{canicbio1}
\newblock S. Canic, C.J. Hartley, D. Rosenstrauch, J. Tambaca, G. Guidoboni, A. Mikelic,
\newblock Blood Flow in Compliant Arteries: An Effective Viscoelastic Reduced Model,
\newblock {\em Numerics and Experimental Validation. Ann. Biomed. Eng.}; 34: 575--592 (2006).

\bibitem{sunny2} 
S. Canic, Y. Wang, and M. Bukac. A Next-Generation Mathematical Model for Drug Eluting Stents. SIAM J. Appl. Math., 81(4), 1503?1529 2021. 


\bibitem{cao} Cao, Y., Chen, S. and Meir, A.J., 2014. Steady flow in a deformable porous medium. Mathematical Methods in the Applied Sciences, 37(7), pp.1029-1041.

\bibitem{applied2} Castelletto, N., Klevtsov, S., Hajibeygi, H. and Tchelepi, H.A., 2019. Multiscale two-stage solver for Biot's poroelasticity equations in subsurface media. {\em Computational Geosciences}, 23, pp.207--224.

\bibitem{yotovMultipoint}
\newblock S. Caucao, T.  Li,  and I. Yotov, 
\newblock A multipoint stress-flux mixed finite element method for the Stokes-Biot model,
\newblock {\em Numerische Mathematik}, 152(2), pp.411-473, 2022.

\bibitem{filtration2}
\newblock Cesmelioglu, A.
\newblock Analysis of the coupled Navier-Stokes/Biot problem. 
\newblock {\em Journal of Mathematical Analysis and Applications}, 456(2), pp.970--991, 2017.

\bibitem{filtnum}
Cesmelioglu, A. and Chidyagwai, P., 2020. Numerical analysis of the coupling of free fluid with a poroelastic material. {\em Numer. Meth. PDEs}, 36(3).

\bibitem{infsup2}
\newblock A. Cesmelioglu, V. Girault, B. Rivi\`ere,
\newblock Time-Dependent Coupling of Navier-Stokes and Darcy Flows,
\newblock {\em ESAIM: Mathematical Modelling and Numerical Analysis}, 2013.

\bibitem{ciarletbook}
\newblock P.G. Ciarlet,
\newblock {\em Linear and Nonlinear Functional Analysis},
\newblock SIAM, 2013.

\bibitem{newhomog} Collis, J., Brown, D.L., Hubbard, M.E. and O'Dea, R.D., 2017. Effective equations governing an active poroelastic medium. {\em Proceedings of the Royal Society A: Mathematical, Physical and Engineering Sciences}, 473(2198), p.20160755.

\bibitem{coussy}
\newblock O. Coussy,
\newblock {\em Poromechanics},
\newblock John Wiley and Sons, (2004).

\bibitem{poroapps}
\newblock E. Detournay, A.H.-D. Cheng,
\newblock Fundamentals of poroelasticity.
\newblock {\em Analysis and design methods}, Elsevier, pp. 113--171, 1993.

\bibitem{evans} Evans, L.C., 2022. Partial differential equations (Vol. 19). {\em American Mathematical Society.}

\bibitem{gilbert4} M. Fang, R. P. Gilbert, A. Panchenko and A. Vasilic, Homogenizing the time-harmonic acoustics of bone: The monophasic case. Mathematical and Computer Modelling, 46, 3-4, (2007), 331--340.

\bibitem{numerical2} Fred, V., Rodrigo, C., Gaspar, F. and Kundan, K., 2021. Guest editorial to the special issue: computational mathematics aspects of flow and mechanics of porous media. {\em Computational Geosciences}, 25(2), pp.601--602.

\bibitem{gilbert1} R. P. Gilbert and A. Panchenko. Acoustics of a stratified poroelastic composite. Zetitschrift fur Analysis und ihre Anwendungen, 18 (1999), 977--1001

\bibitem{gilbert2} R. P. Gilbert and A. Panchenko. Effective acoustic equations for a nonconsolidated medium with microstructure. In: Acoustics, mechanics and the related topics of mathematical analysis, World Scientific, River Edge, NJ, (2002), 164-170.

\bibitem{gilbert3} R. P. Gilbert, A. Panchenko, and A. Vasilic. Homogenizing acoustics of cancellous bone with an interstitial non-Newtonian fluid. Nonlinear Analysis: Theory, Methods, and Applications, 74 (2011), 1005-1018.


\bibitem{infsup1}
 V. Girault, B. Rivi\`ere,
\newblock DG approximation of coupled Navier-Stokes and Darcy equations by Beavers-Joseph-Saffman interface condition,
\newblock {\em SIAM}, 2009.

\bibitem{Grisvard}
\newblock P. Grisvard,
\newblock {\em Elliptic Problems in Nonsmooth Domains},
\newblock SIAM, (2011).


\bibitem{poroplate}
\newblock E. Gurvich and J. T. Webster
\newblock Weak solutions for a poro-elastic plate system.
\newblock {\em Applicable Analysis}, 2021.



\bibitem{thermo}
\newblock D.B. Henry, A. Perissinitto Jr., and O. Lopes,
\newblock On the essential spectrum of a semigroup of thermoelasticity,
\newblock {\em Nonlinear Analysis: Theory, Methods \& Applications}, 21 (1993), pp.65--75.

\bibitem{showrecent} Hosseinkhan, A. and Showalter, R.E., 2023. Semilinear Degenerate Biot-Signorini System. SIam J. Mathematical Analysis, 55(5), pp.5643--5665.

\bibitem{jnr} Gahn, M., J\"ager, W. and Neuss-Radu, M., 2022. Derivation of Stokes-plate-equations modeling fluid flow interaction with thin porous elastic layers. Applicable Analysis, 101(12), pp.4319-4348.

\bibitem{kesavan} 
\newblock S. Kesavan,
\newblock {\em Topics in functional analysis and applications,} 
\newblock New Age International Publishers, Vol. 3, 2019.

\bibitem{rectplate} Kuan, J., \v{C}ani\'c, S. and Muha, B., 2023. Existence of a weak solution to a regularized moving boundary fluid-structure interaction problem with poroelastic media. Comptes Rendus. M\'ecanique, 351(S1), pp.1--30.

\bibitem{sunny3} 
Jeffrey Kuan, Suncica Canic and Boris Muha Fluid-poroviscoelastic structure interaction problem with nonlinear coupling. Submitted 2023. \url{https://arxiv.org/abs/2307.16158}

\bibitem{redbook}
\newblock I. Lasiecka and R. Triggiani, 
\newblock {\em Control theory for partial differential equations: Volume 1}
\newblock Cambridge University Press, 2000.

\bibitem{yotovBJS}
\newblock W.J.Layton, F. Schieweck, and I.  Yotov,
\newblock Coupling fluid flow with porous media flow. 
{\em SIAM Journal on Numerical Analysis}, 40(6), pp.2195-2218, 2002.



\bibitem{yotovNSB}
\newblock Tongtong Li, Sergio Caucao, and Ivan Yotov,
\newblock An augmented fully-mixed formulation for the quasistatic Navier--Stokes--Biot model,
\newblock {\em arXiv preprint arXiv}:2209.02894, 2022.


\bibitem{yotov2022}
\newblock T. Li, and I. Yotov, 
\newblock A mixed elasticity formulation for fluid-poroelastic structure interaction,
\newblock {\em ESAIM: Mathematical Modelling and Numerical Analysis}, 56(1), pp.1-40, 2022.



\bibitem{lionsmag}
\newblock J.L. Lions, E. Magenes
\newblock {\em Non-homogeneous boundary value problems and applications I}
\newblock Springer-Verlag Berlin Heidelberg New York, Vol. 1, 1972

\bibitem{mikelic}
\newblock A.~Marciniak--Czochra, A.~Mikeli\'c.
\newblock A rigorous derivation of the equations for the clamped Biot--Kirchhoff--Love poroelastic plate.
\newblock \emph{Archive for Rational Mechanics and Analysis, 215}(3):1035--1062, 2015.

\bibitem{MRT} Mardal, K.A., Rognes, M.E. and Thompson, T.B., 2021. Accurate discretization of poroelasticity without Darcy stability: Stokes-Biot stability revisited. BIT Numerical Mathematics, 61, pp.941-976.

\bibitem{mclean} McLean, W.C.H., 2000. Strongly elliptic systems and boundary integral equations. {\em Cambridge University Press}.


\bibitem{mikelicBJS}
\newblock A. Mikelic, W. J\"{a}ger,
\newblock On the interface boundary condition of Beavers, Joseph, and Saffman
\newblock {\em SIAM Journal on Applied Mathematics}, 60(4), pp.1111-1127., 2000
 


\bibitem{applied}
Murad, M.A. and Cushman, J.H., 1996. Multiscale flow and deformation in hydrophilic swelling porous media. {\em International Journal of Engineering Science}, 34(3), pp.313-338.




\bibitem{pazy}
\newblock A.~Pazy,
\newblock {\em Semigroups of linear operators and applications to partial differential equations},
\newblock Springer Science \& Business Media, 2012.



\bibitem{yotovBSeye}
\newblock R. Ruiz-Baier, M.  Taffetani, H.D. Westermeyer, and I. Yotov, 
\newblock The Biot-Stokes coupling using total pressure: Formulation, analysis and application to interfacial flow in the eye. 
\newblock {\em Computer Methods in Applied Mechanics and Engineering}, 389, p.114384, 2022.

\bibitem{galdi}
\newblock G. Galdi,
\newblock {\em An introduction to the mathematical theory of the Navier-Stokes equations: Steady-state problems.}\\
\newblock Springer Science \& Business Media, 2011.


\bibitem{wheel} Russell, T.F. and Wheeler, M.F., 1983. Finite element and finite difference methods for continuous flows in porous media. In: {\em The mathematics of reservoir simulation} (pp. 35--106). Society for Industrial and Applied Mathematics.

\bibitem{GGbook}
\newblock Sacco,  R., Guidoboni, G., and Mauri, A.G., 2019.
\newblock {\em A Comprehensive Physically Based Approach to Modeling in Bioengineering and Life Sciences,}
\newblock Elsevier Academic Press.

\bibitem{Sanchez-Palencia}
\newblock Sanchez-Palencia, E., 1980. 
\newblock Non-homogeneous media and vibration theory. 
\newblock {\em Lecture Notes in Physics} 127, Springer-Verlag

\bibitem{indiana}
Showalter, R.E., 1974. Degenerate evolution equations and applications. {\em Indiana University Mathematics J.}, 23(8), pp.655-677.


\bibitem{show2000}
\newblock R.E.~Showalter.
\newblock Diffusion in poro-elastic media.
\newblock \emph{Journal of mathematical analysis and applications, 251}(1):310--340, 2000. 

\bibitem{showfiltration}
\newblock R.E. Showalter
\newblock Poroelastic filtration coupled to Stokes flow
\newblock {\em Control theory of partial differential equations}, (2005), pp.243--256

\bibitem{showsu} Showalter, R.E. and Su, N., 2001. Partially saturated flow in a poroelastic medium. Discrete and Continuous Dynamical Systems Series B, 1(4), pp.403-420.

\bibitem{both} Storvik, E., Both, J.W., Nordbotten, J.M. and Radu, F.A., 2022. A Cahn-Hilliard-Biot system and its generalized gradient flow structure. Applied Mathematics Letters, 126, p.107799.

\bibitem{temam}
\newblock R. Temam,
\newblock Navier-Stokes Equations: Theory and Numerical Analysis,
\newblock {American Mathematical Soc.}, Vol. 343, 2001 

\bibitem{temam1}
\newblock R. Temam,
\newblock Infinite-Dimensional Dynamical Systems in Mechanics and Physics
\newblock {\em Springer-Verlag}, New York, 1988.

\bibitem{mikelicnon} van Duijn, C.J. and Mikelic, A., 2019. Mathematical Theory of Nonlinear Single-Phase Poroelasticity. Preprint hal-02144933, Lyon June.


\bibitem{terzaghi}
\newblock K. von Terzaghi.
\newblock {\em Theoretical Soil Mechanics}
\newblock Wiley, New York, 1943

\bibitem{sunny1}
Yifan Wang, Suncica Canic, Martina Bukac, Charles Blaha, Shuvo Roy. Mathematical and Computational Modeling of a Poroelastic Cell Scaffold in a Bioartificial Pancreas. Fluids, vol. 7, issue 7, p. 222 (2022)




\end{thebibliography}
\end{document}